\documentclass{article}

\usepackage[margin=1.2in]{geometry}
\usepackage{amsmath,amssymb,amsfonts,graphicx,amsthm,nicefrac,mathtools,bm}
\usepackage{hyperref}
\usepackage[numbers]{natbib}
\usepackage[english]{babel}
\usepackage{times}
\usepackage[T1]{fontenc}
\usepackage{bbm,mdframed}
\usepackage[dvipsnames]{xcolor}
\usepackage{xspace}
\usepackage{rotating,multirow}
\usepackage{graphics,tikz}
\usepackage{grffile}
\usepackage{subcaption}
\usepackage{epstopdf} % for postscript graphics files 
\usepackage{enumitem} 
\usepackage[ruled,linesnumbered]{algorithm2e}
\graphicspath{{./plots/}}
\usepackage[toc,title,page]{appendix}
\usepackage[export]{adjustbox}
\usepackage{multicol}
\usepackage{float}
\usepackage{setspace}

\usepackage[normalem]{ulem}

\newtheorem{theorem}{Theorem}

\newtheorem{lemma}{Lemma}
\newtheorem{condition}{Condition}
\newtheorem{proposition}{Proposition}
\newtheorem{corollary}{Corollary}

\theoremstyle{definition}
\newtheorem{definition}{Definition}
\theoremstyle{remark}
\newtheorem{remark}{Remark}

\makeatletter

\newtheorem*{rep@theorem}{\rep@title}
\newcommand{\newreptheorem}[2]
{\newenvironment{rep#1}[1]
	{\def\rep@title{#2 \ref{##1}} \begin{rep@theorem}}%
		{\end{rep@theorem}}}
\makeatother
\newreptheorem{theorem}{Theorem}
\newreptheorem{lemma}{Lemma}
\newreptheorem{corollary}{Corollary}
\newreptheorem{proposition}{Proposition}

\newcommand{\figref}[1]{Figure~\ref{fig:#1}}

\newcommand{\secref}[1]{Section~\ref{sec:#1}}
\newcommand{\appref}[1]{Appendix~\ref{app:#1}}

\newcommand{\defref}[1]{Definition~\ref{def:#1}}

\newcommand{\lemref}[1]{Lemma~\ref{lem:#1}}

\newcommand{\propref}[1]{Proposition~\ref{prop:#1}}

\newcommand{\thmref}[1]{Theorem~\ref{thm:#1}}

\newcommand{\coma}{\textnormal{coma}}

\let\oldnl\nl% Store \nl in \oldnl
\newcommand{\nonl}{\renewcommand{\nl}{\let\nl\oldnl}}
\newcommand{\PP}[1]{\textnormal{Pr}\!\left\{{#1}\right\}} % Probability
\newcommand{\EE}[1]{\mathbb{E}\left[{#1}\right]} % Expectation

\newcommand{\EEst}[2]{\mathbb{E}\left[{#1}\ \middle| \ {#2}\right]} % Conditional expectation
\newcommand{\PPst}[2]{\text{Pr}\!\left\{{#1}\ \middle| \ {#2}\right\}} % Conditional probability
\newcommand{\ident}{\mathbf{I}}
\newcommand{\ones}{\mathbf{1}}

\def\independenT#1#2{\mathrel{\rlap{$#1#2$}\mkern2mu{#1#2}}}
\newcommand\independent{\protect\mathpalette{\protect\independenT}{\perp}}

\newcommand{\ignore}[1]{}
\let\emptyset\varnothing

\usepackage{fancyhdr}
\newcommand{\thedate}{\today}
\newcommand{\theauthor}{Jinjin Tian$^{1}$, Xu Chen$^{2}$, Eugene Katsevich$^{3}$, Jelle Goeman$^{2}$, Aaditya Ramdas$^{1}$  \bigskip \\
$^1$ Carnegie Mellon University \\
$^2$ Leiden University Medical Center\\
$^3$ University of Pennsylvania \smallskip \\
\texttt{\{jinjint,aramdas\}@stat.cmu.edu}\\
\texttt{\{X.Chen.MS, J.J.Goeman\}@lumc.nl}\\
\texttt{ekatsevi@wharton.upenn.edu}}
\newcommand{\thetitle}{
Large-scale simultaneous inference under 
dependence
}

\date{\thedate}
\author{\theauthor}
\title{\thetitle}

\fancypagestyle{plain}{%
	\fancyhf{}%
	\lhead{}
	\rhead{}
}

\usepackage[mmddyy]{datetime}

\newcommand{\ta}[1]{t_{\alpha}(#1)}
\newcommand{\tac}[1]{\overline{t}_{\alpha}(#1)}

\makeatletter
\newcommand{\dotfrac}[2]{
	\mathchoice
	{\ooalign{$\genfrac{}{}{0pt}{0}{#1}{#2}$\cr\leavevmode\cleaders\hb@xt@ .22em{\hss $\displaystyle\cdot$\hss}\hfill\kern\z@\cr}}
	{\ooalign{$\genfrac{}{}{0pt}{1}{#1}{#2}$\cr\leavevmode\cleaders\hb@xt@ .22em{\hss $\textstyle\cdot$\hss}\hfill\kern\z@\cr}}
	{\ooalign{$\genfrac{}{}{0pt}{2}{#1}{#2}$\cr\leavevmode\cleaders\hb@xt@ .22em{\hss $\scriptstyle\cdot$\hss}\hfill\kern\z@\cr}}
	{\ooalign{$\genfrac{}{}{0pt}{3}{#1}{#2}$\cr\leavevmode\cleaders\hb@xt@ .22em{\hss $\scriptscriptstyle\cdot$\hss}\hfill\kern\z@\cr}}
}
\makeatother

\tikzstyle{none} = [rectangle, rounded corners, minimum width=0.5cm, minimum height=0.5cm,text centered, draw=black, fill=red!20]
\tikzstyle{noneoff} = [rectangle, rounded corners, minimum width=0.5cm, minimum height=0.5cm,text centered, draw=black, fill=red!8]

\tikzstyle{ada} = [rectangle, rounded corners, minimum width=0.5cm, minimum height=0.5cm, text centered, draw=black, fill=orange!20]
\tikzstyle{dis} = [rectangle, rounded corners, minimum width=0.5cm, minimum height=0.5cm, text centered, draw=black, fill=blue!20]
\tikzstyle{adaoff} = [rectangle, rounded corners, minimum width=0.5cm, minimum height=0.5cm, text centered, draw=black, fill=orange!7]

\tikzstyle{adadis} = [rectangle, rounded corners, minimum width=0.5cm, minimum height=0.5cm, text centered, draw=black, fill={rgb:red,10;green,20;yellow,54}]
\tikzstyle{adadisoff} = [rectangle, rounded corners, minimum width=0.5cm, minimum height=0.5cm, text centered, draw=black, fill=green!5]

\tikzstyle{arrow} = [thick,->,>=stealth, text width=3cm]
\tikzstyle{dotarrow} = [dashed, text width=3cm]

\tikzstyle{title} = [rectangle, rounded corners, minimum width=1cm, minimum height=1cm, text centered, draw=white, fill=none]

\begin{document}
	
\author{\theauthor}

\maketitle

\begin{abstract}
    % The classical global null testing can be view as the problem of testing against combination of multiple $p$-values.
    %In large scale hypotheses testing, one would like to select a set that contains as many true signals as possible, together with some error control of false signals. 
    Simultaneous inference allows for the exploration of data while deciding on criteria for proclaiming discoveries. It was recently proved that all admissible post-hoc inference methods for the true discoveries must employ closed testing. In this paper, we investigate efficient closed testing with local tests of a special form: thresholding a function of sums of test scores for the individual hypotheses. Under this special design, we propose a new statistic that quantifies the cost of multiplicity adjustments, and we develop fast (mostly linear-time) algorithms for post-hoc inference. Paired with recent advances in global null tests based on generalized means, our work instantiates a series of simultaneous inference methods that can handle many dependence structures and signal compositions. We provide guidance on the method choices via theoretical investigation of the conservativeness and sensitivity for different local tests, as well as simulations that find analogous behavior for local tests and full closed testing.
    %One result of independent interest is the following: if $P_1,\dots,P_m$ are one-sided Gaussian $p$-values derived from the coordinates of an arbitrary $m$-dimensional Gaussian, then their arithmetic average $P$ behaves like a $p$-value for small thresholds, satisfying $\text{Pr}(P \leq t) \leq t$ for $t \leq \frac{1}{2m}$.
    
    %Under various signal settings in practice, we specifically examine an important subclass where local test uses generalized mean based combines (thus admits \emph{separability}), summarized by \citet{vovk2020combining}. In this paper we dive deeper into this class of combinations by deriving tight asymptotic calibration thresholds in the positively equicorrelated Gaussian setting. Using this calibration, we explicitly quantify the effect of dependence on type-I error and power, which not only confirms known behaviours of classical methods, but also uncovers new behaviours of other generalized mean combinations.  
    % : for FWER control, post-hoc FDP control of an arbitrary set, and to find the largest set with a given FDP bound.
    % simultaneous error control methods under different types of dependence with either FWER control or FDP control using closed testing technique.
    % (...Simulations...Real data...)
    
\end{abstract}

% \tableofcontents

\section{Introduction}\label{sec:intro}
%\com{Motivate paper via closed testing. Computation: Vovk and Wang, and Wilson--- incomplete/inefficient. Compare numerically.}
In large-scale hypothesis testing problems, choosing the criteria for proclaiming discoveries, or even picking an error metric, can be tricky before researchers look at their data. A much more flexible approach is simultaneous (and thus post-hoc) inference, which allows the researcher to examine the whole data set and compare data-dependent guarantees on any subsets that they like before finally rejecting a set of null hypotheses along with the associated guarantee. Simultaneous inference methods are typically designed to control the false discovery proportion (FDP) for all possible choices of selections simultaneously \citep{goeman2011multiple,blanchard2020post,katsevich2020simultaneous}. It was recently proved that optimal post-hoc methods must be based on closed testing \citep{marcus1976closed,genovese2006exceedance,goeman2019simultaneous}. Nevertheless, one big obstacle that prevents closed testing from being popular in practice is its exponential computation time in the worst case. Further, the complex nature of the closure process makes it hard to theoretically quantify conservativeness and power. 

The key to dealing with these obstacles lies in the building block of closed testing, which is a local test for every subset of hypotheses, that tests for the presence of a signal in at least one of the hypotheses in the subset (in other words, global null testing for each subset of hypotheses). The design of such local tests, i.e., the choice of the global null test to apply, is critical, as its special structure may allow fast (quadratic, linearithmic, or even linear) time shortcuts to be derived; and its robustness to dependence and power under various settings will be largely preserved after closure.

A practical choice for such a local test is a $p$-value combination test, that combines the evidence against the individual hypotheses in the subset into a single test statistic. Formally speaking, consider a set of hypotheses $H_1, \dots, H_m$, each as a collection of probability measures defined on the same space $(\Omega, \mathcal{F})$, where $Q^{\star}$ is the true (unknown) distribution that generates the data. A hypothesis $H_i$ is true if $Q^{\star} \subseteq{H_i}$, and the global null hypothesis is specified by
\begin{equation}
    \bigcap_{i=1}^m H_i := \{\ H_i\ \textnormal{is true, for all } i\in {\color{black}\{1,2,\dots,m\}}\}.
\end{equation}
Assume that we construct some test statistic, or score, $T_i$ which captures evidence refuting $H_i$, and satisfying
\begin{equation}\label{valid}
    \sup_{Q^{\star} \in H_i}  \textnormal{Pr}_{Q^{\star}}\{T_i\leq C_i(x)\} \leq x, \quad \forall x \in [0,1],
\end{equation}
%\jellecomment{Better use $\sup_{Q^{\star} \in H_i}  \textnormal{Pr}_{Q^{\star}}$ instead of appending a composite hypothesis to $\textnormal{Pr}$}
for some corresponding critical value $C_i$. (The scores are high when $H_i$ is true.) One common choice is a $p$-value, where $T_i = P_i$, {\color{black}with $P_i$ being a valid $p$-value for $H_i$}, and $C_i(x) \equiv x$. Then, global null testing can be done in the following way: combine those scores using a function $f$ and find a calibration function $C$ such that 
\begin{equation}\label{gnsum}
    \sup_{Q^{\star} \in \bigcap_{i=1}^m H_i}  \textnormal{Pr}_{Q^{\star}}\{f(T_1, \dots, T_m) \leq C(m,x)\} \leq x, \quad \forall x \in [0,1]
\end{equation}
is true under the assumed dependence structure (if any) among the scores\footnote{Note that generally the functions $f$ and $C$ can also depend on the scores themselves, however in this paper we consider specifically the case when $f$ and $C$ is fixed, and $f$ as function the scores only, and $C$ as function of the cardinality of hypotheses set only. These cases already consist of a large proportion of existed global null tests, and simplify the analysis throughout the paper.}. We call a global null test in the form of \eqref{gnsum} as \emph{monotonic} if $f$ is monotonic in each of its arguments; \emph{symmetric} if $f$ remains unchanged on permuting its arguments. Monotonicity and symmetry are two rather common features of a global null test. Given both monotonicity and symmetry of local tests, quadratic time shortcuts for finding simultaneous FDP confidence bounds (FDP shortcuts) have been developed by \citet{goeman2011multiple} and later a quadratic time variant for simultaneous FWER control (FWER shortcuts) was presented by \citet{dobriban2020fast}.

In this paper, we investigate how inference can benefit from a more specific structure of the local test, that of \emph{separability} (see Appendix~\ref{app:compdef} for formal definitions of the aforementioned terms). In particular, we consider the following special case of \eqref{gnsum}:
\begin{equation}\label{sumform}
    \sup_{Q^{\star} \in \bigcap_{i=1}^m H_i}  \textnormal{Pr}_{Q^{\star}}\{\sum_{i=1}^m h(T_i) \leq C(m,x)\} \leq x, \textnormal{ for all } x \in [0,1],
\end{equation}
where $h$ is a monotonic function of scores. Given the local tests of form \eqref{sumform}, we show that both FDP and FWER shortcuts can be reduced to linear time, after an initial sorting step (\thmref{shortcut},\ref{thm:gammaselection}). %Given only monotonicity of local tests, we develop a novel statistic called \emph{coma} (defined in \secref{coma}) to explicitly quantify the price of simultaneity under closed testing, which can be computed in linear time. 

 Design \eqref{sumform} applies to a majority of existing global null tests, including famous examples like Fisher’s combination test \citep{fisher1992statistical}, Stouffer's combination method \citep{stouffer1949american}, R{\"u}schendorf's results \citep{ruschendorf1982random} about the arithmetic mean of $p$-values; as well as recent advances like the harmonic mean \citep{wilson2019harmonic}, Cauchy \citep{liu2020cauchy} and   L\'evy \citep{wilson2021evy} combinations. A particular work that is closely related to ours is a summary of all the above-mentioned global null tests: the generalized mean based combination methods \citep{vovk2020combining}. The fast shortcuts we developed allow bringing those canonical and new global null tests, to post-hoc large-scale real-world applications. Consequently, we obtain a class of novel methods for simultaneous inference, which we found rich enough to contain powerful solutions that adapt to various dependence assumptions and signal distributions. 
 
 We further study the adaptivity in a subclass of our methods via careful quantification of the balance between conservativeness caused by the need to protect against unknown dependence and test power. Specifically, we calibrate against the intermediate setting of arbitrary Gaussian correlation (rather than the two extremes, independence and arbitrary dependence), and investigate the asymptotic power under our derived calibration. The theoretical findings regarding local tests are then empirically confirmed to be preserved after closure.
 
 One result of independent interest is the following: if $P_1,\dots,P_m$ are one-sided Gaussian $p$-values derived from the coordinates of an arbitrary $m$-dimensional Gaussian, then their arithmetic average $P$ behaves like a $p$-value for small thresholds, satisfying $\text{Pr}(P \leq t) \leq t$ for $t \leq \frac{1}{2m}$.

The paper outline is as follows. In \secref{post-hoc}, we derive linear time algorithms for three kinds of tasks for closed testing using a local test of form \eqref{sumform}: 1) simultaneity assessment (e.g., \ compute the cost of simultaneity for a single subset of hypotheses chosen pre-hoc or post-hoc), 2) simultaneous inference (e.g., \ type-I error bounds and FDP error bounds calculation for a single subset of hypotheses), and 3) automatic post-hoc selection (e.g., \ selection of the largest set of hypotheses with a predefined error level for its post-hoc FDP bound). Then we focus on the multivariate Gaussian setting to formally evaluate a class of local tests satisfying our requirements based on generalized means. Specifically, in \secref{cal}, we derive the asymptotic valid calibrated threshold for positively equicorrelated Gaussians, which allows us to calculate the price paid to protect against different levels of dependence using different combinations choices. Then we calculate closed-form asymptotic power expressions under different signal settings in \secref{pow}, and reason about the sweet spot for different combination methods. Finally, we confirm that our qualitative conclusions about local tests are preserved after closure, using simulations in \secref{exp}. A conclusion including takeaways for practitioners and future directions is provided in \secref{con}.

\section{Simultaneous inference via closed testing}\label{sec:post-hoc}

Recall that we are interested in testing hypotheses $H_1, \dots, H_m$, each represented by a collection of probability measures on the measurable space $(\Omega, \mathcal{F})$, where $Q^\star$ is the true (unknown) measure that generates the data. We call a hypothesis $H_i$ \emph{null} if $Q^\star \in H_i$, and \emph{non-null} otherwise. We denote $H_S := \bigcap_{i \in S}H_i$ as the intersection hypotheses corresponding to index set $S$, which is \emph{null} if and only if $H_i$ is \emph{null} for all $i\in S$. In particular, we let $H_{\emptyset}$ equals the set of all probability measures on $(\Omega, \mathcal{F})$, so the null hypothesis $H_{\emptyset}$  is always true.   Let $\mathcal{H}_0 := \{i: Q^\star \in H_i\}$ denote the (unknown) set of null hypotheses that are true.

 The non-null hypotheses are usually of more interest, often serving as an important reference for variable selection and scientific discovery. Therefore we often call the non-null hypotheses as \emph{signals}. A common goal is to identify a large set of hypotheses that contains mostly signals. In other words, we wish to proclaim a set of ``discoveries'' while controlling the number or fraction of false discoveries (i.e., \ the null hypotheses that were incorrectly proclaimed as discoveries). 

For a set $S\subseteq [m]:=\{1,2,\dots,m\}$ indexing the hypotheses, define its (unknown) number of false and true discoveries as
\begin{equation}\label{ftdpdef}
\epsilon(S) := |S\cap \mathcal{H}_0|,\quad \delta(S) := |S\setminus \mathcal{H}_0|,
\end{equation}
respectively.
We wish to find $t_{\alpha}(S)\in\{0,1\}$ and $e_{\alpha} \in \{0,1,\dots,|S|\}$ such that:
\begin{align}
\label{talphadef}
%  \textnormal{Type-I error control:} \quad 
%  &\PP{\ones\{\delta(S) \neq 0\} \geq t_{\alpha}(S)} \geq 1-\alpha,\\
% % &\PP{t_\alpha(S) > \ones\{\delta(S) \neq 0\}} \leq \alpha.\\
% \label{talphadef}
 \textnormal{Type-I error control:} \quad 
 &\PP{\delta(S) \geq t_{\alpha}(S)} \geq 1-\alpha,\\
\label{ealphadef}
\textnormal{False Discovery Proportion (FDP) control:} \quad  &\PP{\epsilon(S) \leq e_{\alpha}(S)} \geq 1-\alpha, 
\end{align}
where $t_\alpha(S)$ indicates whether we reject $H_S$ or not, and $e_{\alpha}(S)$ provides the upper bound of the number of non-signals in $S$. Specifically, Type-I error control guarantees that, with high probability, $S$ is not rejected if it contains only nulls, while the FDP control guarantees that, with high probability, the number of false discoveries in set $S$ is upper bounded. Naturally, we prefer $t_\alpha(S)$ to be one if possible and $e_\alpha(S)$ to be as small as possible. The slightly odd formalism for~\eqref{talphadef} is simply to draw parallels with the definitions that follow.

To freely examine several arbitrary sets $S$ and then select a set, we need extra corrections to ensure post-hoc validity of error guarantees. In other words, we would need to convert the above high probability guarantees for an individual set $S$ into one for all possible sets simultaneously. Formally, we desire
\begin{equation}\label{tbalphadef}
   \textnormal{Simultaneous Type-I error control:} \quad  \PP{\delta(S) \geq \overline{t}_{\alpha}(S) \textnormal{ for all } S \subseteq{[m]}} \geq 1-\alpha,
\end{equation}
for some $\overline{t}_{\alpha}(S)\in\{0,1\}$ as before, and we would like to design an $\overline{e}_{\alpha}(S) \in \{0,1,\dots,|S|\}$ such that
\begin{equation}\label{ebalphadef}
    \textnormal{Simultaneous FDP control:} \quad \PP{\epsilon(S) \leq \overline{e}_{\alpha}(S) \textnormal{ for all } S \subseteq{[m]}} \geq 1-\alpha.
\end{equation}
 Closed form expressions for $\overline{t}(\cdot)$ and $\overline{e}(\cdot)$ can be derived in special cases~\citep{katsevich2020simultaneous}, but only bounds based on closed testing can be admissible \citep{goeman2019only}. Closed testing was initially proposed by \citet{marcus1976closed}, who suggested using
\begin{equation}\label{tbaralphasol}
\overline{t}_{\alpha}(S) = \ones\{ t_\alpha(J) =1 \textnormal{ for all } J \supseteq{S}\}.
\end{equation}
It was later noticed by \citet{goeman2011multiple} that the same procedure also yields an expression for $\overline{e}_{\alpha}(S)$: 
\begin{equation}\label{fdpbound}
\overline{e}_{\alpha}(S) = \max{\{|I|: I\subseteq{S},\  \overline{t}_{\alpha}(S) {\color{black}= 0} \}},
\end{equation}
which is the size of the largest subset of $S$ that is not rejected by closed testing. In this closed testing framework, $t_{\alpha}$ defined in \eqref{talphadef} is also called as a \emph{local test}, which is just a valid $\alpha$-level test of the composite hypothesis $H_S$, while $\overline{t}_{\alpha}$ is the corresponding post-hoc version. We denote the set of composite hypotheses rejected locally (before closure) as $\mathcal{U}_{\alpha}$, and as $\mathcal{X}_{\alpha}$ after closure, that is
\begin{equation}\label{sets}
\mathcal{U}_{\alpha} = \{S \subseteq{[m]}: t_{\alpha}(S)=1\}, \quad \textnormal{and} \quad \mathcal{X}_{\alpha} = \{S\subseteq{[m]}: \overline{t}_{\alpha}(S)=1\}.
\end{equation}
In this paper, we focus on the case when local test $t_{\alpha}$ is of the following form:
\begin{equation}\label{sepsym}
t_{\alpha}(S) = \ones\left\{ \sum_{i=1}^{|S|} h(T_{i}) \leq C(|S|,\alpha)\right\},
\end{equation}
where $h(\cdot)$ is a monotonically increasing function.

\begin{remark}
In fact, the form \eqref{sepsym} satisfies three common and reasonable designs of global null test, which are \emph{symmetry}, \emph{monotonicity} and \emph{separability}. {\color{black}Specifically: the summation structure corresponds to \emph{separability}; the monotonicity of $h$ corresponds to \emph{monotonicty}; and the index-invariant fact about $h$ corresponds to \emph{symmetry}.} We refer the interested readers to \appref{compdef} for details, and definitions of the aforementioned terms.
\end{remark}

%Note that and it is equivalently to required to satisfy
% \begin{equation}\label{talphadef2}
% \textnormal{Pr}_{\bigcap_{i \in S}H_i}\{t_\alpha(S) =1\} :=\sup_{\theta \in \bigcap_{i \in S}H_i} P_{\theta}(t_\alpha(S) =1) \leq \alpha,
% \end{equation}
% which is equivalent to our definition for $t_{\alpha}$ in \eqref{talphadef}. 

Before we proceed, we introduce a special class of local tests based on generalized means as discussed by \citet{vovk2020combining}, since we will repeatedly use them as motivating examples. Consider the following combinations of $p$-values $p_1, \dots, p_m$, indexed by $r \in [-\infty,\infty]$:
\begin{align}\label{Mrm}
    M_{r}(p_1, \dots, p_m) := 
    \begin{cases}
    \max_{i \in [m]} {p_i} , & \text{if } r = \infty;\\
    (\prod_{i=1}^{m}p_i)^{1/m}, & \text{if } r = 0;\\
    (\frac1m\sum_{i=1}^{m} p_i^r)^{1/r}, &r \in (-\infty,0) \cup (0, \infty);\\
    m\min_{i \in [m]}{p_i} , & \text{if } r = -\infty,\\
    \end{cases}
\end{align}
which corresponds to the arithmetic mean when $r = 1$; geometric mean when $r = 0$; and harmonic mean when $r = -1$. For simplicity, we use $ M_{r,m}$ to stand for $M_{r}(p_1, \dots, p_m)$ throughout the paper. Denote 
\begin{equation}\label{tralpha}
t_{\alpha}^{(r)}(S) := \ones\{M_{r}((p_i)_{i \in S}) \leq c_r(|S|,\alpha) \},
\end{equation}
where $c_{r}(|S|,\alpha)$ is a critical value that depends only on $|S|$, $\alpha$ for different $r$. {\color{black}Then according to \citet{vovk2020combining}},  $t_{\alpha}^{(r)}(S)$ is a valid local test, and the corresponding class
\begin{equation}\label{Talpha}
\mathcal{T}_{\alpha} := \{t_{\alpha}^{(r)}: r \in [-\infty, \infty]\}    
\end{equation}
is rich enough to contain many famous local test choices like the Bonferroni ($r=-\infty$) method, the Fisher's combination ($r=0$), and the recent harmonic mean combination method ($r = -1$); and its members also have simple enough structure such that we can summarize their nature with a univariate parameter $r$. 

\subsection{The cost of multiplicity adjustment arising from post-hoc inference}\label{sec:coma}
In practice, one may be concerned that simultaneity has a large statistical cost (paid in power). To address this concern, we propose a novel statistic called \emph{coma}, which stands for the \emph{\underline{CO}st of \underline{M}ultiplicity \underline{A}djustment} arising from requiring valid post-hoc inference. The statistic is invariant to the testing level $\alpha$, and only costs linear time to compute. 

%As an immediate notice, we have
% \begin{equation}\label{ha}
% \overline{e}_{\alpha}([m]) = h_{\alpha}([m]), \quad \textnormal{where}\quad h_{\alpha}(S)= \max{\{i\in S: K_i(S) \notin U_{\alpha}\}},
% \end{equation}
% with $K_i(S)$ as the set of indices of $i$ largest $p$-values in $S$. For simplicity, we write $h_{\alpha}([m])$ as $h_{\alpha}$. From  \citet[Lemma 1]{goeman2019simultaneous}, $h_{\alpha}$ is the largest size of an intersection hypothesis that is not rejected by the closed testing procedure, i.e.  if $|S|>h_{\alpha}$, then $S\in \Xa$.

To construct \emph{coma} such that it is invariant to test level $\alpha$, we intentionally use the adjusted $p$-value, which is defined as the smallest $\alpha$ under which the test would be rejected. Formally, for a certain set $S$ among a series of hypotheses $H_1, \dots, H_m$, denote the adjusted $p$-value based on $S$ using local testing rule $t_{\alpha}$ as
\[
p(S) := \inf\{\alpha \in [0,1]: t_{\alpha}(S) = 1\},
\]
and the adjusted $p$-value for $S$ after going through the closed testing procedure as
\begin{equation}\label{pcdef}
    \overline{p}(S) := \inf{\{\alpha \in [0,1]: \tac{S} = 1\}}.
\end{equation}
Then \emph{coma} is defined as follows.

\begin{definition}[cost of multiplicity adjustment]
For any $S \subseteq{[m]}$,  define 
\begin{equation}
   \coma(S) := \overline{p}(S)/p(S) 
\end{equation} 
as the cost of multiplicity adjustment when testing $H_S$.
\end{definition}

Note that  $\coma(S)$ is a data-dependent quantity that depends on the choice of local test. As for a quick example, $\coma(S) = \frac{m}{|S|}$ if $t_{\alpha}$ is Bonferroni and $p(S)$ is small enough. This example concurs with the intuition that the cost of multiplicity grows with the total dimension $m$; however, it decreases with the subset dimension $|S|$. The following result presents a more general expression for \emph{coma}.
% , which only costs linear computation time if the local test is monotonic. 

\begin{theorem}\label{thm:adjustedp}
For any $S\subseteq{[m]}$, if the local test is of form \eqref{sepsym}, then we have a linear time expression
    \begin{equation}\label{pclinear}
     \overline{p}(S) = \max_{0 \leq i \leq |S^c|} p(S \cup J^\star_i),
     \end{equation}
     where $J^\star_i$ is the set of indices of hypotheses associated with the $i$ largest $p$-values in $S^c$.     
\end{theorem}

% Therefore we arguably suggest one with a low requirement of the resolution of discoveries to consider post-hoc instead of pre-hoc as it introduces simultaneity without much cost. 

\begin{figure}[H]
    \centering
	\includegraphics[width=0.7\textwidth]{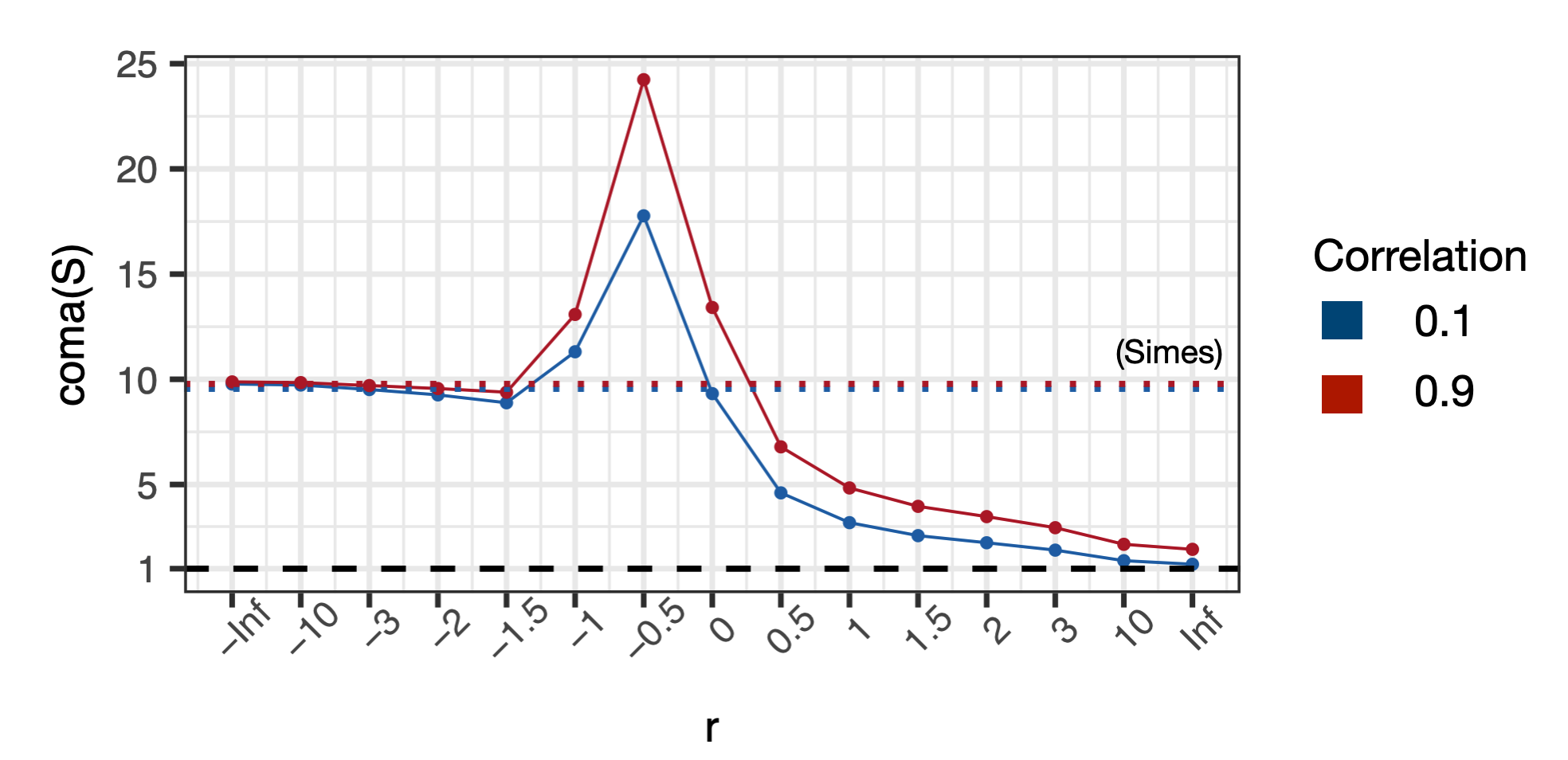}
	\caption{$\coma(S)$ versus different local test procedures under different extend of dependency. {\color{black} The dashed horizontal lines represents the value of $\coma(S)$ with the local test as Simes, whereas the solid lines plot the value of $\coma(S)$ with the local test as $t_{\alpha}^{(r)}$ versus different $r$. When $r=-\infty$ (written as $-\text{Inf}$),  $t_{\alpha}^{(r)}$ recovers Bonferroni.} We simulate the data to follow equicorrelated Gaussian, where we set total number of hypotheses $m = 200$, and size of set $S$ as 20. We set signal proportion outside $S$ as $0.3$, signal proportion inside $S$ as $0.7$ with signal strength (i.e. the mean of Gaussian) $\mu = 2$. The results are averaged over $5\times 10^3$ trials.  }\label{fig:comar}
\end{figure}
% \com{Annotate Bonferroni and Simes clearer, now it is confusing to read.}
\begin{figure}[H]
    \centering
	\includegraphics[width=\textwidth]{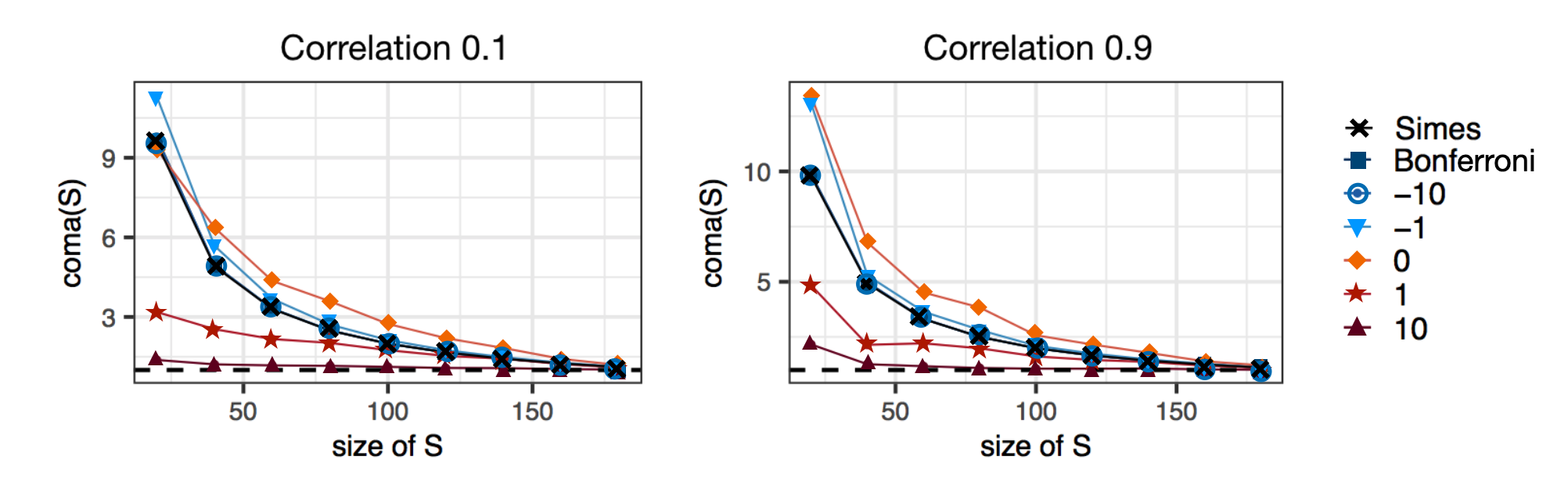}
	\caption{$\coma(S)$ versus the size of $S$ using different local test procedures under  We simulate the data to follow equicorrelated Gaussian, where we set total number of hypotheses $m = 200$, and size of set $S$ as 20. We set signal proportion outside $S$ as $0.3$, signal proportion inside $S$ as $0.7$ with signal strength (i.e. the mean of Gaussian) $\mu = 2$. The results are averaged over $5\times 10^3$ trials. {\color{black}We can see that, the lines for Simes, Bonferroni and $r=-10$ almost overlap with each other, while the line for $r=-1$ is slightly higher (and the line for $r=0$ is on the top). These observations are consistent with results in \figref{comar}.}}\label{fig:comasize}
\end{figure}
\thmref{adjustedp} is proved in \appref{pfadjustedp}. %The validity of part (a) of \thmref{shortcut} in fact holds true with a broader class of local tests that adopts \emph{monotonicty} (\condref{mono}), and hence that is what we proved in the \appref{pfadjustedp}. 
In the following, we examine how \emph{\coma} varies with the size of $S$ for local tests based on generalized means $\mathcal{T}_{\alpha}$ in \eqref{Talpha}, using calibration derived by \citet{vovk2020combining} under arbitrary dependence.  %$t_{\alpha}^{(r)}$ as defined in \eqref{tralpha}, with corresponding critical value derived by \citet{vovk2020combining} under arbitrary dependence. 
\figref{comar} plots \emph{\coma} versus the choice of local test for a set $S$ of size $20$ out of $200$ hypotheses in total, using equicorrelated Gaussian data. We can see that \emph{\coma} with local test $t_{\alpha}^{(r)}$ with positive $r$ is generally smaller than that with negative $r$, while the order statistics based procedures, Simes and Bonferroni, behave similarly. This indicates one would prefer to use $t_{\alpha}^{(r)}$ with positive $r$ if one does not want too different results on changing from pre-hoc to post-hoc. On the other hand, \figref{comasize} plots  \emph{\coma} versus the size of the target set $S$ (with the total number of hypotheses remaining as $200$). Except for the consistent observation that positive $r$ have lower \emph{\coma}, we can additionally see that $\coma(S)$ generally decreases with the size of $S$, which agrees with our intuition that lower resolution post-hoc inference should cost less. 
\subsection{Fast shortcuts for post-hoc inference and selection}\label{sec:shortcutslinear}
Another practical concern with regard to imposing simultaneity is the heavy computation time, which is exponential in $m$ in general. In this section, we present fast (linear time shortcuts for calculating both $\overline{t}_{\alpha}$ and $\overline{e}_{\alpha}$, for local tests of form \eqref{sepsym}. 

\begin{theorem}\label{thm:shortcut}%(\textbf{Linear time post-hoc inference for a single set $S$})
Consider testing $m$ hypotheses with presorted scores post-hoc via closed testing using local test $t_{\alpha}$ of form \eqref{sepsym}. For a set $S\subseteq{[m]}$,  Algorithm~\ref{shortcuts-fdp} returns the simultaneous FDP bound $\overline{e}_{\alpha}(S)$ in  \eqref{fdpbound}, with at most $O(m)$ computation. 
% \begin{itemize}
%     \item[(a)] The simultaneous test 
%     \begin{equation}\label{tclinear}
%      \overline{t}_{\alpha}(S) = \prod_{1 \leq i \leq |S^c|}t_{\alpha}( I_i),\quad \text{where } I_i = S \cup J_i^\star,
%      \end{equation}
%      where $J_i^\star$ is the set of indices  associated with the $i$ largest scores in $S^c$;
     
    %  \item[(b)]
% \end{itemize}

%\jellecomment{Theorem 1 and 2b prove $O(m)$ computation in terms of FLOPS and make use o the sum nature of the test, but the proof of Theorem 2a proves computation in $O(m)$ \emph{calculations of the test statistic}, but does not use the sum nature. This is all very confusing, and since we are discussing sum tests in this paper, I would again strongly suggest removing Theorem 2a.}
\end{theorem}
Note that we sort the scores in ascending order in Algorithm~\ref{shortcuts-fdp} in order to have easier tracking of indices since the algorithm is a step-down procedure. The proof for \thmref{shortcut} is in \appref{shortcutpf}. 

%The validity of part (a) of \thmref{shortcut} in fact holds true with a broader class of local tests that adopts \emph{monotonicity} (\condref{mono}), and hence that is what we proved in the \appref{shortcutpf}.

\begin{remark}
Note that local test $t_{\alpha}^{(r)}$ does not admit form \eqref{sepsym} when $r = \pm \infty$, therefore the shortcut in \thmref{shortcut} for evaluating corresponding $\overline{e}_{\alpha}^{(r)}$ is not applicable. However, they lead to \emph{consonant}\footnote{A closed testing is \emph{consonant} if the local tests for every composite hypothesis $S\in 2^{[m]}$ are chosen in such a way that rejection of $S$ after closure implies a rejection of at least one of its elementary hypothesis after closure.} closed testing as proved by \lemref{closedconsonant} in \appref{closedconso}, and one interesting fact pointed out by ~\citet{goeman2011multiple} is that if for consonant closed testing, the simultaneous FDP bound for a given set reduces to finding the number of its elementary hypotheses that the closed testing cannot reject, therefore reducing to identifying the set of elementary hypotheses being rejected after closure. For $r = -\infty$, this is just Holm's method, while for $r = \infty$, this is just checking whether we can reject the largest $p$-value to decide either to reject all or nothing. %Hence, in the following, we only present the shortcuts for $r \neq \pm{\infty}$, when things are nontrivial. 
\end{remark}

%\textnormal{sign}(r)\frac{s\alpha^r}{r+1}, & r\in (-1,0) \cup(0,\infty) ;\\
%s\log{(\alpha/e)}, &  r = 0;\\
%-s\alpha_{-1,s}/\alpha, & r = -1;\\
%-(\frac{(r+1)\alpha}{rs})^r, & r \in (-\infty,-1).
 %We add similar notation to the corresponding terms (i.e. $\overline{t}_{\alpha}^{(r)}$, $e_{\alpha}^{(r)}$, $\overline{e}_{\alpha}^{(r)}$). 
% \begin{lemma}\label{lemma:localproperty}
%     The local test $t_{\alpha}^{(r)}$ defined in \eqref{tralpha} is monotonic, symmetric for all $r \in \mathbb{R}$, and separable if $r\neq \pm \infty$.
% \end{lemma}

 %Therefore, the shortcuts for evaluating corresponding $\overline{e}_{\alpha}^{(r)}$ are quite trivial when $r=\pm \infty$.

\begin{algorithm}[h!]
	\SetAlgoLined
	\KwIn{A sequence of sorted scores $T_1, \dots, T_m$ which satisfies $T_{1}\geq\dots\geq T_{m}$;
	a local test rule of form \eqref{sepsym} with a monotonically increasing transformation function $h$ and thresholding function ${\color{black}C}$;
	confidence level $\alpha$;
	candidate rejection set $S = \{i_1, i_2, \dots, i_s\}$ and its complement $S^{c} = \{j_1, j_2, \dots, j_{m-s}\}$ with \smash{$i_1< i_2 < \dots i_s,\quad j_1< j_2 < \dots j_{m-s}$.}}
	\KwOut{High probability ($1-\alpha$) simultaneous bound $\overline{e}_{\alpha}(S)$ on the number of false discoveries in $S$. }
	{\textbf{Initialization:}\\
	\nonl transformed candidate set scores: $u_{1}, \dots,u_{s}, \textnormal{where}\  u_{d} = h(T_{i_d})\  \textnormal{for}\ 1\leq d\leq s$;\\
	\nonl transformed complementary set scores: $v_{1},\dots,v_{m-s},\ \textnormal{where}\  v_{d} = h(T_{j_d})\  \textnormal{for}\ 1\leq d\leq m-s$;\\
	\nonl ill-defined transformed scores: $v_0 = \max(u_1, v_1); \quad {\color{black}v_{m-s+t}, u_{s+t} \equiv \min(u_s, v_{m-s}) - 1,\ \forall 1\leq t\leq m-s}$;\\
	\nonl iteration related indices $k \leftarrow 1; \quad b \leftarrow -1;$\\
	\nonl accumulated scores $Q =0$.}
	
	\For{{\color{black}$a = 1, \dots, m$}}{
		\eIf{$u_{k+b+1} \geq v_{a-k-b}$ \textbf{or} $a=1$}{
			$Q = Q + u_{k+b+1}$
			
			$b = b+1$
		}{
			$Q = Q + v_{a-k-b}$
		}
		\While{$k\leq \min(s,a)$ \textbf{and} $Q > {\color{black}C}(a,\alpha)$}{ 
		\eIf{$b>0$}{
		$b \leftarrow b-1$}{
		$Q \leftarrow Q + u_{k+1} - v_{a-k}$}
		
		$k \leftarrow k+1$
	}
	}
	\Return{$k-1$} 
	\caption{Shortcut for evaluating post-hoc false discoveries bound $\overline{e}_{\alpha}(S)$}\label{shortcuts-fdp}
\end{algorithm}

We have presented procedures for fast inference on a single set $S$ picked freely by users, which in turn, enables effective post-hoc selection among multiple sets of interest: linear and quadratic shortcuts for automatic selection of the largest set $S$ with a prespecified bound $\overline{e}_\alpha$ can also be developed. For users who have no idea of which candidate set to evaluate, \thmref{gammaselection} allows them for efficient automatic selection among a sequence of incremental sets: finding the largest one among them with FDP bounded by $\gamma \in [0,1)$. 
% \begin{theorem}\label{thm:selection}
% For the closed testing with $m$ hypotheses using monotonic, symmetric and separable local test level at level $\alpha$, we have Algorithm~\ref{shortcuts} returns the largest set $S$ with $\overline{e}_{\alpha}(S)=0$, which is equivalent to the collection of all the elementary hypotheses with FWER control at level $\alpha$, with at most $O(m)$ time.
% \end{theorem}
% \begin{enumerate}[label=(\alph*)]
%      \item 
    
%     \item The largest set $S$ with $\overline{e}_{\alpha}(S)=|S|$, or equivalently $\overline{t}_{\alpha} =0$, is equivalent to the indices of hypotheses associated with the $\overline{e}_{\alpha}([m])$ smallest scores, which takes at most $O(m)$ computation to calculate.
%     \end{enumerate}

\begin{theorem}\label{thm:gammaselection}
    Consider testing $m$ hypotheses post-hoc via closed testing at level $\alpha$, and a series of  incremental candidate sets  to reject: $S_1\subset S_2  \dots \subset S_n \subseteq [m]$ with $|S_i|=i$ for all $i \in [n]$. Then we have:
    \begin{itemize}
        \item [(a)] Given any desired FDP bound $\gamma \in [0,1)$, Algorithm 3 returns the largest set $S_k$ such that $\overline{e}_{\alpha}(S_k) \leq \gamma|S_k|$. 
    \end{itemize}
    If we additionally require local test to be of form \eqref{sepsym}, then 
    \begin{enumerate}
        \item[(b)] Algorithm 3 costs at most {\color{black}$O(mn)$} computation; 
        % \com{Should this be $O(n*m)$ instead? Or does the number n of candidate sets not relate to the computation time?}
        \item[(c)] Algorithm 3 reduces to Algorithm 2 if $\gamma \equiv 0$ and $S_k$ is the indexes of hypotheses with $k$ smallest scores, which cost at most $O(m)$ computation with presorted scores.
    \end{enumerate}
\end{theorem}
The validity of Algorithm 3 in \thmref{gammaselection} does not require any assumption on local test or presorting $p$-values, and needs $m$ iterations in the worst case. In practice we expect fewer iterations will be needed as the false discoveries are ruled out in batches quickly.
% and fast computation each iteration with additional requirements on local test and on $\gamma$ level.
Particularly, for the special case stated in part (c) in \thmref{gammaselection}, the task costs only at most linear time. The proof of \thmref{gammaselection} is in \appref{pfgammaselection}. 

\begin{remark}
Algorithm 2 in \thmref{gammaselection} is also the shortcut for finding the largest hypotheses set to reject with strong FWER control among all $m$ hypotheses.
\end{remark}
 
\section{Calibration of local tests for multivariate Gaussians}\label{sec:calibration}
The performance of closed testing based post-hoc inference largely depends on the building blocks--- local tests. Therefore, in order to provide better guidance of applying our newly derived shortcuts introduced in \secref{post-hoc}, we look into the properties of different global null tests, particularly the generalized mean based ones (i.e., $t_{\alpha}^{(r)}$ defined in \eqref{tralpha}) since our shortcuts apply to these.
% and has recently arouse heated discussion \citep{vovk2020combining,wilson2019harmonic,wilson2020generalized,chen2020trade}. 
\citet{vovk2020combining} first summarized the class of generalized mean based combination methods, and derived closed form calibration under arbitrary dependence for different combination choice, using results based on robust risk aggregation. 
% Before we summarize their results, we would like to point out that the case $m=2$ is curiously not given for the harmonic mean i.e. $M_{-1,2}$. \lemref{validm2} addresses this missing piece, with proof located in  \appref{pfvalidm2}.
% \begin{lemma}\label{lem:validm2}
% For any pair of valid $p$-values $p_1$, $p_2$, we have 
% \begin{equation}
%     \textnormal{Pr}_{H_1 \cap H_2}\left\{M_{-1,2}(p_1,p_2) \leq \frac{\alpha}{2}\right\} \leq \alpha,
% \end{equation}
% where $M_{-1}$ represents the harmonic mean function \eqref{Mrm}. Particularly, if $p_1$ and $p_2$ are marginally standard uniform, then a copula exists such that the equality is achieved.
% \end{lemma}
Now we specifically summarize the results for calibrating under arbitrary dependence \citep{vovk2020combining} in the following \lemref{valid}, as it will be our benchmark to compare with. 
\begin{remark}
Though a refined version of \lemref{valid} (which gives best possible calibration) can be found in Proposition 8.1 \citep{vovk2022admissible}, it does not admit closed-form expression as \lemref{valid} does. Therefore we adopt \lemref{valid} throughout the paper for simpler theoretical analysis. 
\end{remark}

%Recall $M_{r,m}$ stated in \eqref{Mrm} is valid by a linear transformation, and this transformation is precise or at least asymptotic precise in the sense that there is a set of $p$-values such that \eqref{valid} achieves equality at some $t$. Particularly, we summarize the main results 
\begin{lemma}\label{lem:valid} (\citet{vovk2020combining})
{\color{black}For $m$ hypotheses, $\alpha / \alpha_{r,m}$ is a valid critical value for the global null test $t_{\alpha}^{(r)}$ defined in \eqref{tralpha}}, where
\begin{align}\label{alphark}
    \alpha_{r,m} := 
    \begin{cases}
    (r+1)^{1/r}, & \text{if } r \in (-1, \infty];\\
    ( (y_m +m)^2/(y_m+1))\ones\{m\geq 3\} + m\ones\{m \leq 2\}, & \text{if } r = -1;\\
    \frac{r}{r+1}m^{1+1/r}, &r \in [-\infty,-1), 
    \end{cases}
\end{align}
and $y_m$ is the unique strictly positive solution of $y^2 = m\left((y+1)\log{(y+1)}-y\right)$. Particularly, for \smash{$r \in \{-\infty, 0, \infty\}$}, we define $\alpha_{r,m}$ as $\lim_{r\to \infty} (r+1)^{1/r} = 1$, $\lim_{r\to 0} (r+1)^{1/r} = e$, and \smash{$\lim_{r\to -\infty} \frac{r}{r+1}m^{1+1/r} = m$}.
\end{lemma}
% \com{In Lemma 1, the coefficient $\alpha_{r,m}$ is given such that $a_{r,m} M_{r,m}$ is a valid $p$-value under arbitrary dependence. Does that imply that the calibrated threshold is $\alpha/a_{r,m}$, given the significance level $\alpha$? If so, I would suggest to write Lemma 1 using calibrations for consistency, since the term ``valid $p$-value'' in Lemma 1 is only used once in the paper.}

Follow-up work \citep{wilson2020generalized, chen2020trade} explored the conservativeness of such calibration under some special dependence structures: \citet{wilson2019harmonic} derived asymptotic valid (in the sense of $m\to \infty$) calibration under independence using generalized central limit theorem, and empirically studied their performance when the independence assumption is broken; \citet{chen2020trade} compared the generalized mean based combination with order statistics based combination, and proved that only Cauchy combination (and its analog harmonic mean) and Simes combination pay no price for calibration to achieve validity under assumptions from independence to full dependence (i.e. correlation one); \citet{vesely2021permutation} studied the special case that permutation tests can be used. \figref{summary} summarizes all the cases (including ours) where theoretically valid calibration has been derived. Note that, before our work, almost no results have derived in cases other than the two extremes---the independence case and the arbitrary dependence case: \citet{chen2020trade} provided some theoretical justification in the pairwise Gaussian scenario but only for harmonic mean. As for common intermediate dependence structures like multivariate Gaussian case, most work only explored experimentally. Therefore, as shown in \figref{summary}, we work towards filling in the gap by deriving calibration under one of the intermediate cases, the equicorrelated Gaussian setting, which contains both two extremes as well as different dependence levels. Later, we also investigate the performance of our calibration by analyzing the asymptotic type-I error and power under different settings. Particularly, our calibration recovers existing work in scenarios where independence provably has the highest inflated type-I error to be calibrated among others. At the same time, our theoretical performance investigation justifies the interesting behaviors noticed in early experimental studies \citep{wilson2019harmonic, chen2020trade}, that is, for the generalized mean based methods,  choice of positive $r$ performs better under heavy dependence and calibrating under independence gives a high false-positive rate overall. In contrast, the choice of negative $r$ performs poorly under heavy dependence, and calibrating under independence gives a low false-positive rate overall.

\begin{figure}[h!]
    \centering
	\includegraphics[width=\textwidth]{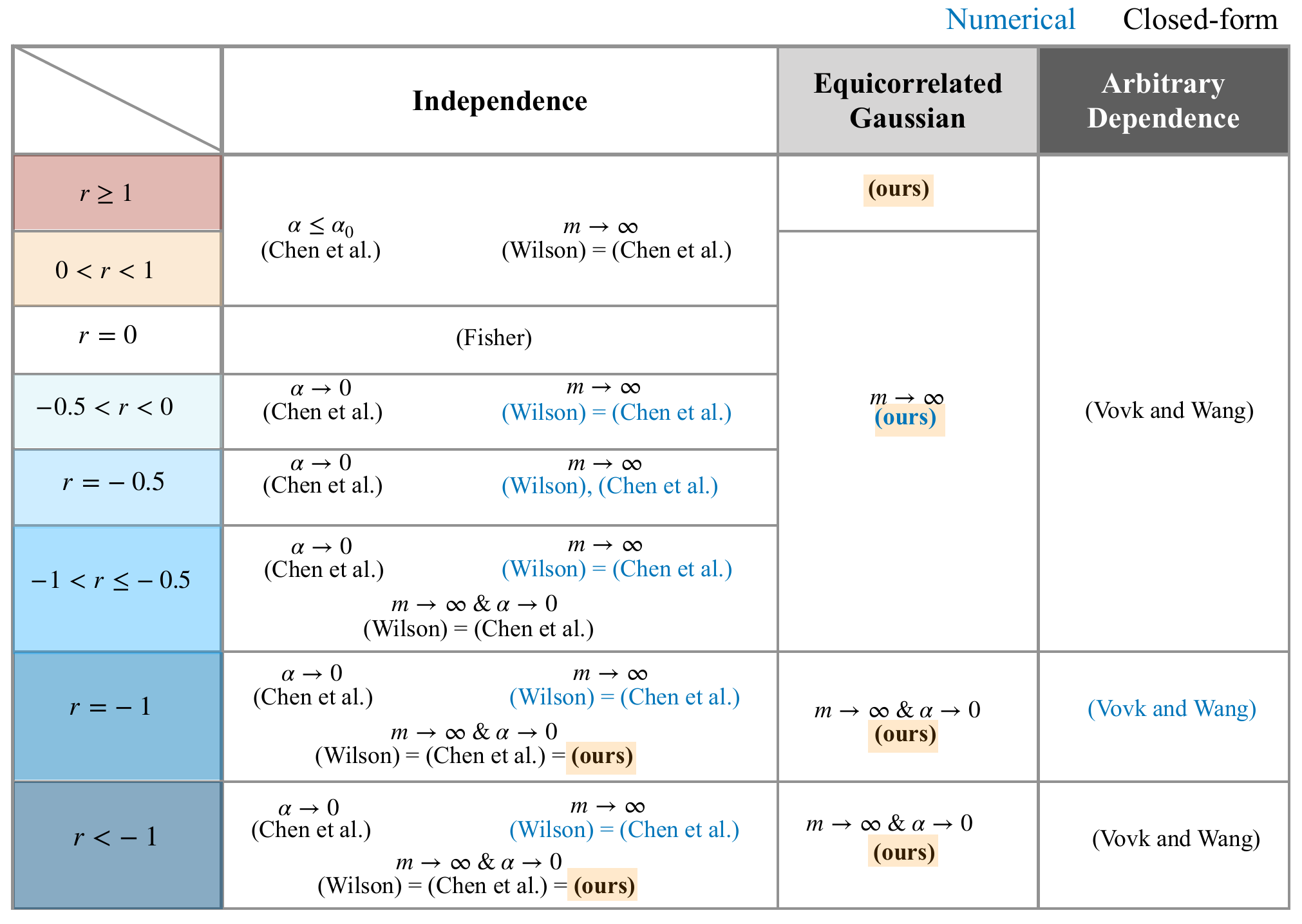}
	\caption{Summary of regimes in which we know how to calibrate generalized means of $p$-values. We omit explicit expressions as there is sometimes no analytical formula, but thresholds can be calculated numerically (blue text). We refer readers to the corresponding references mentioned in the text for explicit expressions.}\label{fig:summary}
\end{figure}

\subsection{Model set-up}
Before presenting the main results, we first motivate our choice of the equicorrelated Gaussian model. Consider a Gaussian sequence model for the observations:
\begin{equation}\label{model}
    (X_{m1}, X_{m2}, \dots, \dots X_{mm}) \sim N_m(\bm{\mu}_m, \Sigma_m),
\end{equation}
where $\bm{\mu}_m = (\mu_{m1}, \dots, \mu_{mm})$, and each entry $\mu_{mi} \stackrel{iid}{\sim}\mu_m B_m$, with $\mu_m>0$ as a scalar, and $B_m$ as a Bernoulli random variable with parameter $\pi_m$. Additionally, we assume $\Sigma_m \in \mathcal{M}_m$,
% \[
% \Sigma_m = \begin{bmatrix}
% 1 & \rho_{1, 2} & \dots & \rho_{1, m}\\
% \rho_{2, 1} & 1 & \rho_{2,3} & \rho_{2, m}\\
%  \vdots & & \ddots & \\
% \rho_{m, 1} & \rho_{m,2} & \cdots & 1\\
% \end{bmatrix},
% \]
% where $\rho_{i,j} \in [-\frac1m,1]$ for all $i, j$, 
where $\mathcal{M}_m$ is the set of all $m\times m$ positive semidefinite correlation matrices. {\color{black}We denote the $(i,j)$-th entry of $\Sigma_m$ as $\rho_{ij}$.} Additionally, we denote the set of all equicorrelation matrices as $\mathcal{M}^E_m$, which is the subset of $\mathcal{M}_m$ with all equal non-diagonal elements.

Suppose we are testing the global null hypothesis. 
\[
\bigcap_{i=1}^{m}H_{mi} := \{\mu_{mi} = 0,\ \forall\  i\}
\]
at level $\alpha$. We consider a one-sided $p$-value $p_{mi} = \Phi(-X_{mi})$ for each elementary hypothesis (where $\Phi$ is the CDF of a standard normal), and combine them using a generalized mean  $\overline{t}^{(r)}_{\alpha}$ \eqref{tralpha}. 
%We consider one-sided $p$-value $p_{mi} = \Phi(-X_{mi})$ for each elementary hypotheses (where $\Phi$ is the CDF of a standard normal), and combine them as a valid $p$-value for testing the global null in form of $\overline{t}^{(r)}_{\alpha}$ \eqref{tralpha}. Denote the type-I error given the correlation matrix $\Sigma$ with regard to different $r$ as follows:
% \begin{equation}\label{Hrhom}
% \widetilde{\alpha}_m(\Sigma, \alpha, r) := \text{Pr}_{\bigcap_{i=1}^{m}H_{mi}}\left\{\textnormal{sign}(r)\frac1m \sum_{i = 1}^m p_{mi}^r \leq \textnormal{sign}(r) C(\alpha, r, m)\right\},
% \end{equation}
Denote the corresponding type-I error given the correlation matrix $\Sigma$ with respect to different $r$ as follows:
\begin{align}\label{Hrhom}
\widetilde{\alpha}_m(\Sigma, r, c) := 
 \text{Pr}_{\bigcap_{i=1}^{m}H_{mi}}\left\{ \left(\frac1m \sum_{i = 1}^m p_{mi}^r\right)^{1/r} \leq  c \right\},
\end{align}
where $\text{Pr}_{\bigcap_{i=1}^{m}H_{mi}}:= \sup_{Q^\star \in \bigcap_{i=1}^{m}H_{mi}}
 \text{Pr}_{Q^\star}$, and 
 $c$ is a correction/calibration threshold to account for dependence, which could be an absolute constant or potentially depends on $r,m$, and $\alpha$. 
\begin{remark}
From the monotonicity of the generalized mean with respect to $r$, one can easily verify by contradiction that $\widetilde{\alpha}_m(\Sigma, r, c)\leq\alpha$ implies $c \leq \alpha$.
\end{remark}

% Particularly, given a fixed $r$ and $m$, we expect that $C$ should be monotonically increasing (decreasing) in $\alpha$ if $r \geq 0$ ($r<0$). 
% A reasonable choice of $C(\alpha,r,m)$ is to set it as $\alpha$ for $r\geq 0$. 

\begin{proposition}\label{prop:worstcorr}
	Fix any $m \geq 1$ and any $r\geq 1$. If $c < \frac{1}{2 m}$, then 
	\begin{equation}
	  \sup_{\Sigma \in \mathcal{M}_m} \widetilde{\alpha}_m(\Sigma, r, c) =\sup_{\Sigma \in \mathcal{M}^E_m} \widetilde{\alpha}_m(\Sigma, r, c)=\widetilde{\alpha}_m(\ones_m \ones_m^T, r, c),
	\end{equation}
where $\ones_m$ is the $m$-dimensional vector of all ones.
\end{proposition}	
%\com{add $\lim_{t \to \infty}$ for Cauchy (cite).}

\propref{worstcorr} indicates that, for all $r\geq 1$ and appropriately small $\alpha$, we only need to calibrate against the fully dependent case to have validity across the whole correlation space $\mathcal{M}_m$. The proof of \propref{worstcorr} is in \appref{worstcorrpf}, where we used the convexity of function $\Phi(-x)^r$ when $r \geq 1$ and $x>0$, and the fact that (multivariate) Gaussianity is preserved under linear transformations. It is unclear whether the restriction on $\alpha$ can be entirely removed, but it could perhaps be slightly relaxed by constant factors. {\color{black} A special case that could be particularly interesting is when $r=1$, which we record below for emphasis.}
% \com{In proposition 1, is $\widetilde{\alpha}(\ones_m\ones_m^\top,r,c)$ equal to $c$? I think this statement includes corollary $1$ as a special case.}
% \begin{corollary}
% Choosing $C(\alpha,1,m)=\alpha$ for $r=1$, we obtain that the arithmetic average of $p$-values obtained from arbitrarily dependent Gaussians is a valid $p$-value.
% \end{corollary}

\begin{corollary}
Let $\Sigma \in \mathcal{M}_m$ be an arbitrary positive semidefinite Gaussian correlation matrix (with possibly negative entries). Let $X \sim N(0,\Sigma)$ and let $P_i = \Phi(-X_i)$ for $i=1,\dots,m$. Then, the arithmetic average $(r=1)$ of the $p$-values, $\bar P := \frac1{m} \sum_{i=1}^m P_i$ satisfies
\[
\sup_{\Sigma \in \mathcal{M}_m} \mathrm{Pr}(\bar P \leq \alpha) \leq \alpha,\quad \text{ for any } \alpha < \tfrac1{2m}.
\]
\end{corollary}

It is possible that a variant of \propref{worstcorr} also holds for $r< 1$, but we have found it to be technically intractable to prove currently. Nevertheless, for the sake of simplicity and interpretability, we next consider an intermediate case of equicorrelated Gaussians, which we observed to be worse than the other commonly used correlation structures when $r< 1$ in extensive simulations, while also encompassing the perfectly correlated case in \propref{worstcorr}. In particular, we consider only positive correlation as the semi-positive definite requirement on the correlation matrix forces the range of negative $\rho$ to be in $(-\frac1m,0)$, which vanishes as $m\to\infty$.

\begin{definition}\label{def:modeldef}
\textbf{Positively equicorrelated Gaussian}: For each $m$, the observations $X_{m1}, \dots, X_{mm}$ follow the model in \eqref{model} but with a positive equicorrelated $\Sigma_m$ having {\color{black}its elementary entry} $\rho_{ij}\equiv \rho \in [0,1]$ for all $i\neq j \in [m]$. We denote such data distribution as $G_{\mu_m, \pi_m, \rho}$.
\end{definition}

Formally, in the following \secref{cal} and \secref{pow}, we consider the model defined in \defref{modeldef}, and we write $\widetilde{\alpha}_m(\Sigma, r, c)$ in \eqref{Hrhom} as $\widetilde{\alpha}_m(\rho, r, c)$ for simplicity. We intend to study the asymptotic \smash{($m \to \infty$)} behaviour of calibrated $\widetilde{\alpha}_m(\rho, r, c)$ given fixed $\alpha$. We would like to investigate how their power varies as a function of correlation $\rho$ with respect to different $r$, and different signal settings.

%\com{Define the class and $\rho$ and the domain of $\rho$ properly.}

\subsection{Calibration derivation}\label{sec:cal}

In this subsection, we derive the asymptotic calibration under the positively equicorrelated Gaussian model in \defref{modeldef}. First, we formally define the asymptotic calibration of our concern, and then we present our closed-form solution under the positively equicorrelated Gaussian model.

Typically, the asymptotic ($m\to\infty$) Type-I error would be defined as 
\begin{equation}\label{uniform}
   \mathcal{A}^{\star}(r):=\limsup_{m\to \infty}\sup_{\rho\in[0,1]} \widetilde{\alpha}_m(\rho, r,c).
%   }_{\widetilde{\alpha}_m^{\star}(\rho, r, \alpha)}.
\end{equation}
% $\widetilde{\alpha}_m^{\star}(\rho, r, \alpha)$ via plugging $C_m^{\star}(\alpha, r)$ in place of $C(\alpha,r,m)$ in \eqref{Hrhom}. 
However, we found \eqref{uniform} to be intractable; specifically, before taking the outer limit, we found taking the supremum with respect to $\rho$ for fixed $m$ to be analytically infeasible under the positively equicorrelated Gaussian model. Therefore, we settle for an alternative (weaker) definition of target type-I error as the following surrogate limit:
\begin{equation}\label{point}
    \mathcal{A}(r):= \sup_{\rho\in[0,1]} \limsup_{m\to \infty} \ \widetilde{\alpha}_m(\rho,  r, c).
\end{equation}
Note that $\mathcal{A}^\star(r) \geq \mathcal{A}(r)$ deterministically\footnote{To see this, observe that $\sup_{\rho \in [0,1]} \widetilde{\alpha}_m(\rho, r,c) \geq \widetilde{\alpha}_m(\rho, r,c)$ for all $\rho \in [0,1]$ and all $m$. Taking $\limsup_m$ on both sides maintains the inequality, as does taking a further $\sup_\rho$ on both sides.}, that is control over the surrogate asymptotic type-I error is weaker. Denote the highest calibrated threshold $c$ that achieves $\mathcal{A}(r)\leq \alpha$ as $c_r(m,\alpha)$, that is 
\begin{equation}\label{Cm}
c_r(m,\alpha) := \sup\{c :\sup_{\rho\in[0,1]}\limsup_{m\to\infty} \ \widetilde{\alpha}_m(\rho, r, c) \leq \alpha\},
\end{equation}
and the corresponding limiting type-I error as
\begin{equation}\label{alpha}
    \widetilde{\alpha}(\rho, r, \alpha):=\limsup_{m\to\infty} \ \widetilde{\alpha}_m(\rho, r, c_r(m,\alpha)).
\end{equation}
%\com{sup, lim opposite order}
% \begin{equation}\label{alpha}
% 	\widetilde{\alpha}(\rho, r, \alpha) := \lim_{m \to \infty} \widetilde{\alpha}_m(\rho, r, \alpha).
% \end{equation}
%\com{change the function to $\bar \alpha$; same for $\bar \beta$, else $\alpha$ is a free variable in $[0,1]$, sometimes a function, and sometimes a constant target type-1 error.}
%we intend to find the explicit calibrated expression $C_r(m,\alpha)$ such that the asymptotic type-I error is controlled at level $\alpha$ for all $\rho \in [0,1]$, that is 
In the following, we derive a closed-form expression for $c_r(m,\alpha)$, and the  corresponding $\widetilde{\alpha}(\rho,r,\alpha)$ under the positively equicorrelated Gaussian model. Note that in this setting, the observations can be written as
\begin{equation}\label{decomposition}
X_{i} = \sqrt{\rho}\ Z_{0} + \sqrt{1-\rho}\ Z_i, \quad \text{for all}\ i = 1, 2, \dots, m, 
\end{equation}
where $Z_0\sim N(0,1)$ , $Z_i \stackrel{\text{iid}}{\sim} N(0,1)$ for all $i = 1, 2, \dots ,m$, and $Z_0 \independent{\{Z_i\}_{i=1}^m}$. The corresponding one-sided $p$-values are
\begin{equation}\label{decomp2}
p_{i} = \Phi(-X_{mi}) = \Phi\left(-\sqrt{\rho}\ Z_0 - \sqrt{1-\rho}\ Z_i\right).
\end{equation}
Here we drop index $m$ as the distribution of $X$ does not change with $m$ and the same holds true for $p$-values. An important note from this decomposition is the following conditional independence,
\begin{equation}
p_1, p_2, \dots, p_m \textnormal{ are i.i.d.\ conditional on }\ Z_0,
\end{equation}
which allows us to utilize generalized law of large numbers and obtain \thmref{asymh}.
We write the expectation of $p_i^r$ when conditioning $Z_0 = z_0$ as a function of $z_0$ that is
\begin{align}\label{murho1}
g_{\rho, r}(z_0) := \EEst{p_i^r}{ Z_0 = z_0}=  \int \Phi(-\sqrt{\rho}\ z_0 - \sqrt{1-\rho}\ x)^r \phi(x) dx, 
\end{align}
noting that $g_{0,r}(z_0)$ is a constant, we enforce $g_{0,r}^{-1}(\cdot) \equiv \infty$. 

\begin{theorem}\label{thm:asymh}Under the positively equicorrelated Gaussian setting, we have that, given $\alpha \in (0, 1)$, 
	\begin{itemize}		
		\item[(a)] if $r > 0$, then $\widetilde{\alpha}(\rho, r, \alpha) = \Phi(-g_{\rho, r}^{-1}( \alpha^r ))$, and $c_r(m,\alpha)= \min\{\alpha, (\frac{r}{r+1})^{\frac1r}\}$;
		
		\item[(b)] if $-1 < r \leq 0$, then $\widetilde{\alpha}(\rho, r, \alpha) = \Phi(-g_{\rho, r}^{-1}(c_r(m,\alpha)))$, and $c_r(m,\alpha) = (\sup_{\rho \in [0,1]}{g_{\rho, r}\left(-\Phi^{-1}(\alpha)\right)})^{\frac1r}$ is not a function of $m$, where $g_{\rho,r}$ is defined in \eqref{murho1};
		
		\item[(c)]  if $r =-1$, then $\widetilde{\alpha}(\rho, r, \alpha)= \alpha \ones\{\rho=0\},$ and $c_r(m,\alpha) =  \frac{\alpha}{1+\alpha  \log{m}}$ as $\alpha\to 0$;
	
		\item[(d)]  if $r <-1$, then $\widetilde{\alpha}(\rho, r, \alpha)= \alpha \ones\{\rho=0\},$ and $c_r(m,\alpha) = \alpha m^{\frac{1}{|r|}-1}$ as $\alpha\to 0$.
	\end{itemize}
 In all four cases, we have that $c_r(m,\alpha) \leq \alpha$.
\end{theorem}

The proof of \thmref{asymh} is in \appref{asymhpf}, where we mainly use the decomposition described above and  generalized law of large numbers. From \thmref{asymh}, we can see {\color{black}that} the calibrated threshold $c_r(m,\alpha)$ under positively equicorrelated Gaussian is less conservative than that under arbitrary correlation in \lemref{valid}; particularly, for $r>0$, by a factor of $(r+1)^{1/r}$; for $r=0$, by a factor of $\frac{|\log{\alpha}|+1}{|\log{\alpha}|}$; for $r=-1$, by a factor \footnote{Here we use the approximation $\alpha_{-1,m}\sim \log{m}$ as $m\to\infty$ in { \color{black} \citet[Proposition 6]{vovk2020combining}} to make the expression cleaner. But in \figref{ratio}, we use the numerical solution as stated in \lemref{valid}.} of $\frac{\log{m}}{\alpha\log{m} + 1}$; for $r<-1$, by a factor of $\frac{|r|}{|r|-1}$. \figref{ratio} displays these ratios for $r \in [-10,20]$. We can see that, as $|r|\to\infty$, our positively equicorrelated Gaussian calibration is almost the same {\color{black} as the} calibration derived by \citet[~\lemref{valid}]{vovk2020combining},  for arbitrary dependence, indicating that we do not pay much price for calibrating against positive equicorrelation to arbitrary dependence; while as $|r|\to 1$, the positively equicorrelated Gaussian calibration is much tighter than that of arbitrary correlated Gaussian in~\citet{vovk2020combining}.

\begin{figure}
		\centering
		\includegraphics[width=0.7\textwidth]{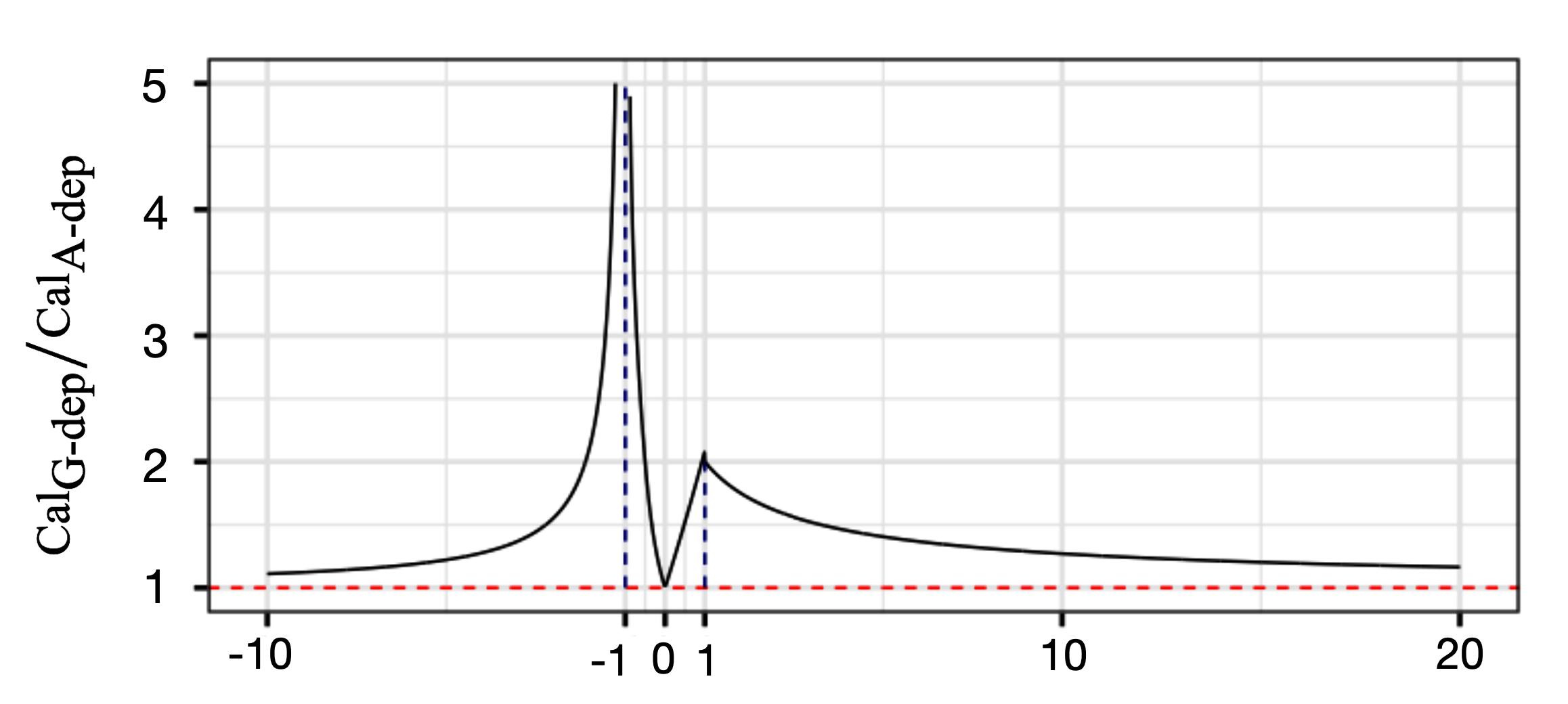}
		\caption{The theoretical ratio of calibrated threshold under positively equicorrelated Gaussian ($\text{Cal}_{\text{G-dep}}$) and under arbitrary dependence ($\text{Cal}_{\text{A-dep}}$) versus different $r$.}\label{fig:ratio}
\end{figure}

% On the other hand, \figref{ratioindep} depicts the price to pay for calibration against positively equicorrelated Gaussian and arbitrary dependence versus different $r$, where the price stands for the ratio of calibrated threshold under independence and the particular dependence structure. We can see that the price increases with $r$ for either positively equicorrelated Gaussian or arbitrary dependence, indicating that smaller $r$ is more robust to dependence.
% %\com{theory for Figure 5: calculate $C^0$ when $\rho=0$, compare to $C$.}
% 	\begin{figure}[H]
% 		\centering
% 		\includegraphics[width=0.5\textwidth]{plots/indep-dep-ratio.png}
% 		\caption{The estimated ratio of calibrated threshold under independence ($\text{Cal}_{\text{indep}}$) and under different types of dependence ($\text{Cal}_{\text{dep}}$): positively equicorrelated Gaussian (red line) or arbitrary dependence (blue line) versus different $r$, using $m=10^3$ hypotheses. }\label{fig:ratioindep}
% 	\end{figure}

\begin{figure}[H]
    \centering
    \begin{subfigure}{0.45\textwidth}
        \centering
    \caption{$m = 10^2$}
    \includegraphics[width=\textwidth]{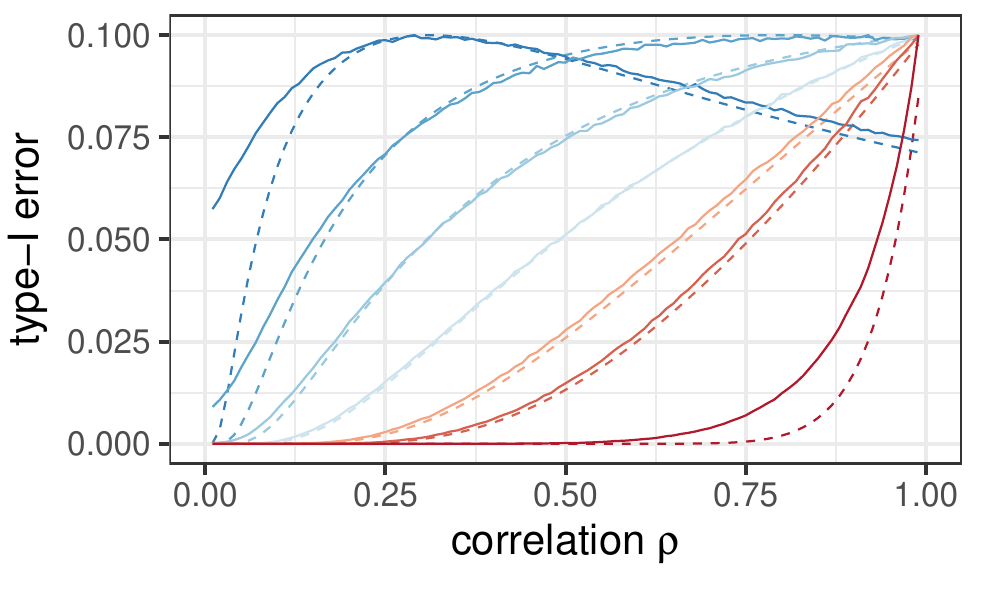}
    \end{subfigure}%
    \begin{subfigure}{0.55\textwidth}
        \centering
    \caption{$m = 10^5$}
    \includegraphics[width=\textwidth]{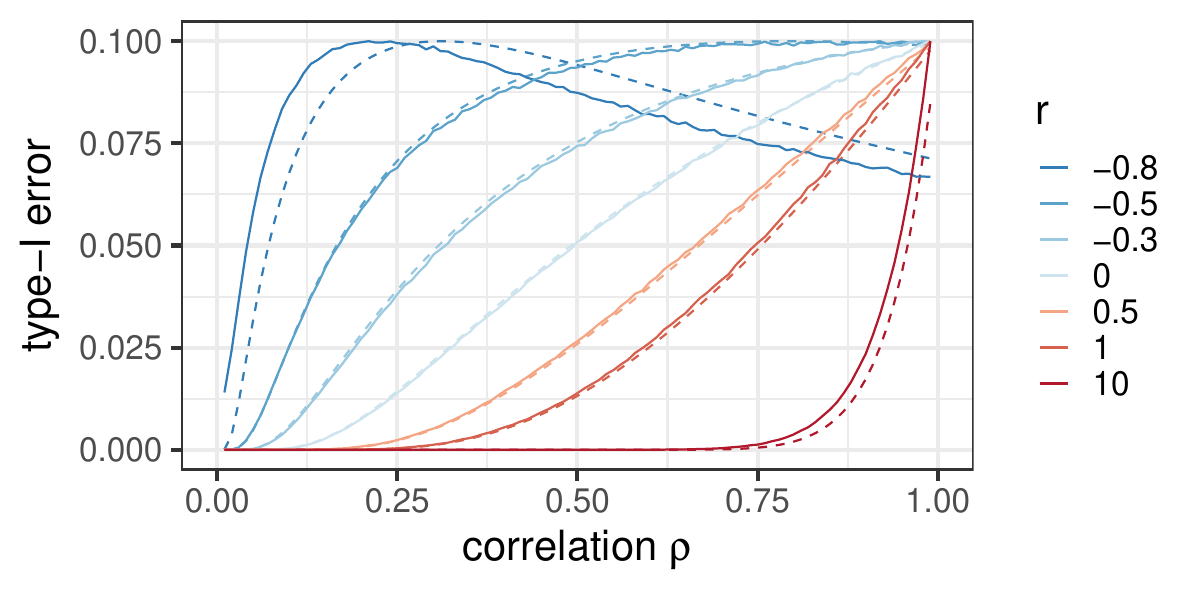}
    \end{subfigure}
	\caption{The asymptotic uniform type-I error $\widetilde{\alpha}(\rho,r,\alpha)$ using the calibrated positively equicorrelated Gaussian test (dotted line, identical in both subplots), and the empirical point-wise Type-I error $\widetilde{\alpha}_m^{\star}(\rho,\alpha,r)$ (solid line) with $m=10^2$ (left) and $m=10^5$ (right) of different method for adjusting dependence via simulation (averaging over $10^6$ trials) versus correlation given different $r>-1$ at confidence level $\alpha = 0.1$.}\label{fig:asymest}
\end{figure}

Next, we conduct some large-scale simulations (in terms of $m$) to justify the surrogate control in \thmref{asymh}. Explicitly, we compare our derived asymptotic type-I error $\widetilde{\alpha}(\rho,r,\alpha)$ under the surrogate calibration (i.e. $\mathcal{A}(r)\leq \alpha$) with the type-I error $\widetilde{\alpha}^\star_m(\rho, r,\alpha)$ under the ideal calibration (i.e. $\mathcal{A}^{\star}(r)\leq \alpha$), that is 
\begin{align}\label{Hrhom-star}
\widetilde{\alpha}^\star_m(\rho, r,\alpha) & := \text{Pr}_{\bigcap_{i=1}^{m} H_{mi}}\left\{\left(\frac1m \sum_{i = 1}^m p_{mi}^r\right)^{\frac{1}{r}} \leq c_r^\star(m,\alpha)\right\}\\
& \textnormal{with}\quad      c_r^{\star}(m,\alpha) := \sup\{c : \sup_{\rho\in[0,1]}\widetilde{\alpha}_m(\rho, r, c) \leq \alpha\}.
\end{align}
\figref{asymest} empirically shows that,  for $r>-1$, $c_r(m,\alpha)$ can also approximately achieve control $\widetilde{\mathcal{A}}^{\star}(r)\leq \alpha$ in the sense that $\widetilde{\alpha}(\rho,r,\alpha) \approx \widetilde{\alpha}_m^{\star}(\rho,r,\alpha)$ for each $\rho\in[0,1]$ if $m$ is large enough.

\begin{figure}[H]
    \centering
    \includegraphics[width=\textwidth]{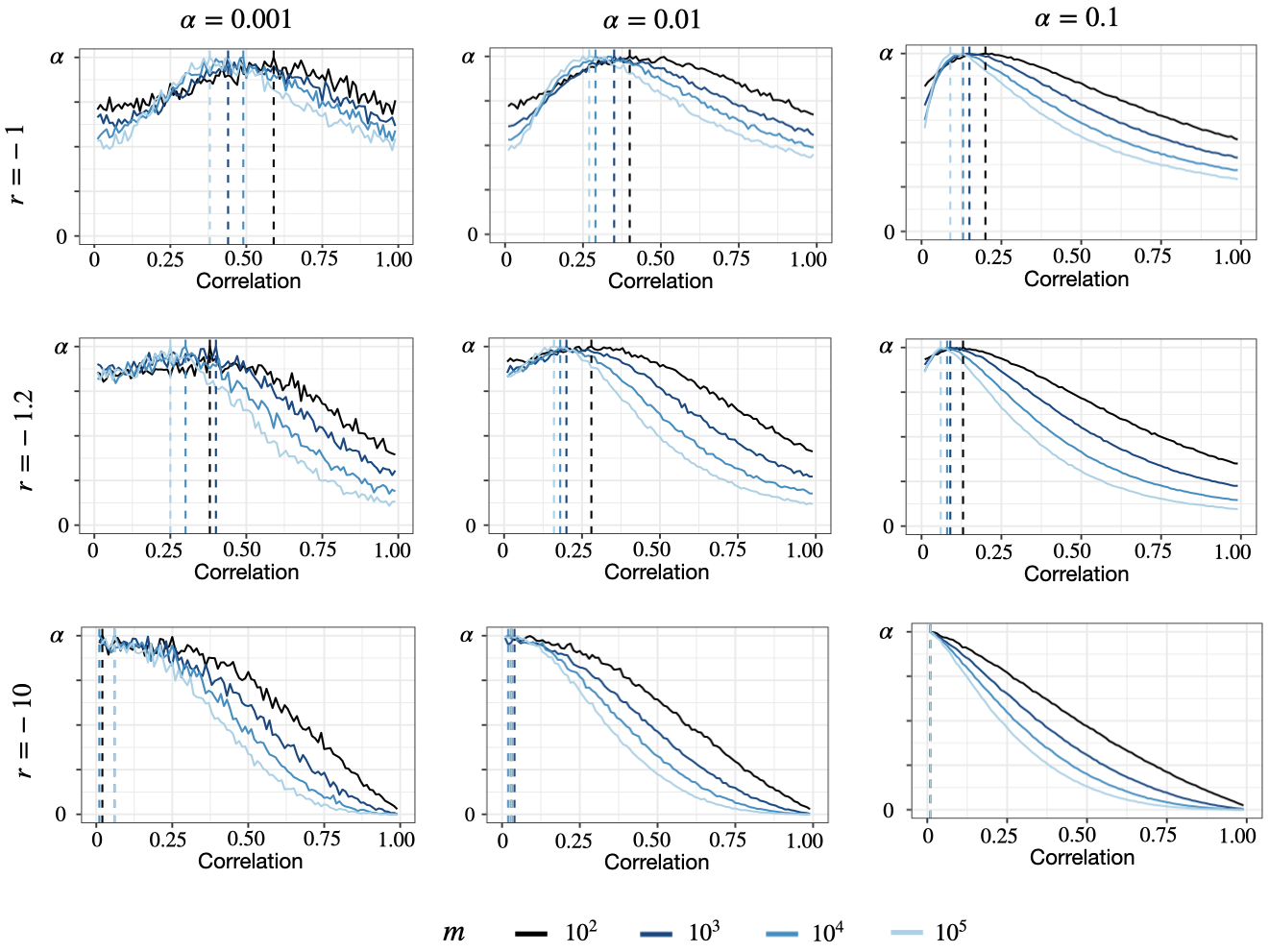}
	\caption{The empirical point-wise Type-I error $\widetilde{\alpha}_m^{\star}(\rho,r,\alpha)$ (solid line) with different \smash{$m\in\{10^2,10^3,10^4,10^5\}$} of different method for adjusting dependence via simulation (averaging over $10^6$ trials) versus correlation given different $r\leq -1$ at confidence level {\color{black}$\alpha = 0.001,0.01,0.1$ in first, second and third column respectively}. The dashed vertical lines indicate the ``worst case" correlation given fixed $m,r,\alpha$, that is  $\arg\max_{\rho\in[0,1]}\widetilde{\alpha}_m^{\star}(\rho,r,\alpha)$.}\label{fig:smallalpha}
\end{figure}

For $r \leq -1$ the approximation is much looser as the convergence (for point-wise $\rho$) in generalized law of large numbers is much slower and out of feasible simulation scope. Nevertheless, from \figref{smallalpha} one may see a trend of convergence with regard the current limited magnitude of $m$: the ``worst case" correlation given fixed $m,r,\alpha$, that is  $\arg\max_{\rho\in[0,1]}\widetilde{\alpha}_m^{\star}(\rho,r,\alpha)$ slowly approaches $0$ as $m$ grows, at different rate given different $\alpha$ (faster for $\alpha$ away from $0$), and the point-wise Type-I error $\widetilde{\alpha}_m^{\star}(\rho,r,\alpha)$ slowly approaches \smash{$\widetilde{\alpha}(\rho,r,\alpha) = \alpha\ones\{\rho=0\}$} as $m$ grows.

%\com{More interpretation for $r=-1$ for example. Do we observe these in practice?}

% Additionally, \figref{combtype1} demonstrates the asymptotic type-I error we derive in \thmref{asymh}, and the estimated Type-I error via simulation. One can see that our asymptotic results agree with the simulations.

%\com{Connect endpoints to Chen et al. paper.}

% 	\begin{figure}[H]
% 		\centering
% 		\includegraphics[width=\textwidth]{powerratio_comb.pdf}
% 		\caption{The convergence behaviour of $k(m)/m$, where $m$ is the number of hypotheses, and $k(m)$ is the size of the largest rejected set  with FDP bounded by $\gamma$. ``Comb" indicates the new averaging method.}\label{fig:comb}
% 	\end{figure}

\subsection{Power analysis}\label{sec:pow}
In this subsection, we study the power using the calibrated threshold $c_r(m,\alpha)$ derived in \secref{cal}. We look at the case $r>0$ and $r\leq -1$ separately, under different alternative settings. Given $\alpha$, denote the power function for different $r$ under distribution $G_{\mu_m,\pi_m,\rho}$ in \defref{modeldef} as
\begin{equation}\label{beta}
    \beta_{\mu_m,\pi_m,\rho}(r, \alpha) := \mathbb{P}_{G_{\mu_m,\pi_m,\rho}}\left\{\left(\frac1m \sum_{i=1}^m p_{mi}^r\right)^{\frac1r} \leq c_r(m,\alpha)\right\},
\end{equation}
with $c_r(m,\alpha)$ is specified in \thmref{asymh}. 
%\jellecomment{The measure $\PP{}$ depends on $\rho, \mu$. Maybe make that explicit? Why does $\widetilde\beta$ not have an argument $\mu$ and/or $\pi$?} 

In the following, we are interested in the asymptotic behaviour of $\beta_{\mu_m,\pi_m,\rho}(r, \alpha)$ under different settings for $\mu_m$ and $\pi_m$.

\begin{theorem}\label{thm:asymp1}
	Fix $r \geq 0$, and consider the positive equicorrelated Gaussian model in \defref{modeldef}, where
	\[
	\lim_{m\to\infty}\mu_m  = \mu \in [0, \infty], \quad \lim_{m\to\infty}\pi_m = \pi \in [0,1].
	\]
	Then for any $\alpha \in (0, (\frac{r}{r+1})^{\frac1r})$, $\rho \in [0,1]$, we have that the asymptotic power
	\begin{equation}
	    \lim_{m\to\infty} \beta_{\mu_m,\pi_m,\rho}(r, \alpha) = \lim_{m\to \infty}\PP{ \pi_{m}g_{\rho, r} (Z_0 + \tfrac{\mu_{m}}{\sqrt{\rho}}) + (1-\pi_m) g_{\rho, r} (Z_0) \leq  \alpha^r} \in  [ \widetilde{\alpha}(\rho, r, \alpha), 1 ],
	\end{equation}
 with $g_{\rho,r}$ defined in \eqref{murho1}, $\widetilde{\alpha}(\rho, r, \alpha)\in [0, \alpha]$ defined in \eqref{alpha}, and $Z_0$ as a standard Gaussian random variable. In particular, 
\begin{itemize}
    \item  if $\pi = 1$, then 
    \begin{align}\label{pi1}
    \lim_{m\to\infty} \beta_{\mu_m,\pi_m,\rho}(r, \alpha)  = \begin{cases}
    1, & \quad \text{if}\ \mu = \infty;\\
    \Phi\left(-g_{\rho,r}^{-1}(\alpha^r)+\frac{\mu}{\sqrt{\rho}}\right) , & \quad \text{if}\ 0< \mu <\infty;\\
    \widetilde{\alpha}(\rho, r, \alpha), & \quad \text{if}\ \mu =0.
    \end{cases}
    \end{align}
    
    \item if $0< \pi< 1$, then \begin{equation}\label{pi0}
        \lim_{m\to\infty} \beta_{\mu_m,\pi_m,\rho}(r, \alpha)  = \begin{cases}
    \Phi\left(-g_{\rho,r}^{-1}(\frac{\alpha^r}{1-\pi})\right), & \quad \text{if}\ \mu = \infty;\\
     \begin{aligned}[b]
       \PP{ \pi g_{\rho, r} (Z_0 + \frac{\mu}{\sqrt{\rho}})  + (1-\pi) g_{\rho, r} (Z_0) \leq \alpha^r},
     \end{aligned}
      & \quad \text{if}\ 0< \mu <\infty;\\
      \widetilde{\alpha}(\rho, r, \alpha), & \quad \text{if}\ \mu =0.
    \end{cases}
    \end{equation}
    \item if $\pi = 0$, then  $\lim_{m\to\infty} \beta_{\mu_m,\pi_m,\rho}(r, \alpha) = \widetilde{\alpha}(\rho, r, \alpha)$, no matter what value that $\mu$ takes. 
    \end{itemize}
\end{theorem}

% \com{In Equations (36) and (37), I find it very confusing that you put $\in$ symbols within the bracketed expressions.}
The proof of \thmref{asymp1} is in \appref{asymppf1}, where we use a triangular-array version of the generalized law of large numbers. {\color{black}As a quick sanity check, expressions in \eqref{pi1} are always in $\left[\widetilde{\alpha}(\rho, r, \alpha),\  1\right]$, while expressions in \eqref{pi0} are always in $\left[\widetilde{\alpha}(\rho, r, \alpha),\  \Phi\left(-g_{\rho,r}^{-1}(\frac{\alpha^r}{1-\pi})\right)\right]$.} Also we can verify some other intuitive facts: the power asymptotically goes to one when $\pi=1$ and $\mu=\infty$, while it drops to the lower bound $\widetilde{\alpha}(\rho, r, \alpha)$ in the worst case \smash{(i.e. $\pi=0$ or $\mu = 0$)}. \thmref{asymp1} also indicates some surprising findings that are somewhat exclusive to the case of $r>0$: the asymptotic power will not go to one even when $\pi=1$ if the signal strength is finite, that is $\mu <\infty$; a similar phenomenon occurs when $\mu=\infty$ but  $\pi <1$. These counter-intuitive behaviours happen due to the nature of the combination choices with $r>0$, where the combination is essentially a weighted average of $p$-values with a monotonic increasing transformation, thus will be dominated by large $p$-values. Therefore, in the case of $r>0$, as long as there is a non-diminishing part of observations that are most likely to generate large $p$-values (e.g. non-signals or weak signals), then we will lose power.
%\com{Give example of $r=1$ explicitly as we discussed, perhaps with $\pi=0.5$.}

To see an explicit example of \thmref{asymp1}, we consider the case that $r=1$, $\pi\in(0,1)$, with $\mu=\infty$ for simplicity. In this case the combination becomes $1-\pi$ times the arithmetic mean of $p$-values (i.e. $\frac{1-\pi}{m}\sum_{i=1}^m p_{mi}$). Therefore, from the definition the  asymptotic power, we know it equals $ \Phi(-g_{\rho,r}^{-1}(\frac{\alpha^r}{1-\pi}))) = \Phi(-\infty) = 0$ from the definition of $g_{0,r}^{-1}$, which agrees with \thmref{asymp1}. This example indicates a more general phenomenon: when $r\geq 1$, as long as there are some non-signals (i.e. $\pi <1$), the asymptotic power under independence will always be 0 no matter how strong the signal is (i.e. how large $\mu$ is). On the other hand, if $\rho=1$, then we have the asymptotic power equals  $\frac{\alpha}{1-\pi}$, which will be close to one if and only if $\pi\to1$. From this simple example, we justify for the behaviour many \citep{vovk2020combining,wilson2020generalized} observed in experiments: that is the combination choice with $r\geq 1$ will be powerless unless there are many strong signals with heavy dependence. 
% \jellecomment{This paragraph seems to repeat the information in the previous one. }{\color{black} This paragraph was suggested by Aaditya, in order to give a concrete example of last paragraph. }

In the following, we study the case when $r\leq-1$. Firstly, we look at the setting with moderate signal strength, that is $\mu_m = o(\sqrt{\log{m}})$. The following \thmref{r-1weak}  shows that, as long as the signal is not strong enough, the asymptotic power for $r\leq -1$ will always degenerate, no matter how dense the signal is.
\begin{theorem}\label{thm:r-1weak}
	Consider the positive equicorrelated Gaussian model in \defref{modeldef} where
	\begin{equation*}
	    \mu_m = o(\sqrt{\log{m}}), \quad \lim_{m\to\infty}\pi_m = \pi \in [0,1],
	\end{equation*}
	For $r \leq -1$, we have that for all $\rho \in [0,1]$ and $\alpha > 0$,
	\begin{equation}
	    \lim_{m\to \infty} \beta_{\mu_m,\pi_m,\rho}(r, \alpha)=  \alpha\ones\{\rho=0\}.
	\end{equation}
\end{theorem}	
\noindent The proof of \thmref{r-1weak} is in \appref{r-1weakpf}, where we mainly use results about limitation of infinitely divisible random variables. \thmref{r-1weak} indicates that, as long as the signal is not strong enough, the combination with $r\leq -1$ will be powerless unless under independence. Despite of  the observed robustness under dependency for cases of $r\leq -1$ in experiments conducted by \citet{wilson2019harmonic}, the robustness will diminish as the number of hypotheses goes to infinity, and the method actually becomes highly sensitive to the dependence in the end. This phenomenon arises from the fact that the gap between the calibrated threshold grows at least sub-linearly (i.e. $O(m^\epsilon)$ with $\epsilon>0$) for different $\rho$, therefore the conservativeness from calibration grows with the number of hypotheses, and results in high sensitivity to dependence in the end.

In the following, we study the setting when the signal is strong enough for the test to have power one. Specifically, that is when $\mu_m \geq O(\sqrt{\log{m}})$. 

\begin{theorem}\label{thm:r-1strong}
    For $r \leq -1$, consider the positive equicorrelated Gaussian model in \defref{modeldef} where
    \[
    \mu_m = \sqrt{2c\log{m}},\quad \textnormal{with } c >0.
    \]
    For all $\rho \in [0,1]$, if further one of the following is satisfied:
    \begin{itemize}
        \item[(a)] $\lim_{m\to\infty}\pi_m = \pi>0$, and $\sqrt{c} > 1-\sqrt{(1-\rho)}$;
        \item[(b)] $ \pi_m = m^{\gamma-1}$, where $0 < \gamma < 1$, and $\sqrt{c} > 1-\sqrt{\gamma(1-\rho)}$,
    \end{itemize}  
   then we have that, for all $\alpha > 0$,
    \begin{equation}
       \lim_{m\to \infty} \beta_{\mu_m,\pi_m,\rho}(r, \alpha) =1.
    \end{equation}
\end{theorem}
The proof of \thmref{r-1strong} is in \appref{r-1strongpf}, where we use a sandwiching argument, similar to \citet{liu2020cauchy}. \thmref{r-1strong} (a) indicates that, in order to achieve full power at different $\rho$, the signal strength needs to be stronger under heavy dependence. This conclusion agrees with the intuition since the power is fundamentally related to the tail of transformed $p$-value $p_{mi}^r$, which is thinner under heavy dependence while heavier under weak dependence.  \thmref{r-1strong} (b), following the sparse setting~\citep{donoho2004higher}, states that the asymptotic power will achieve one with probability one, as long as the signal strength $c$ and signal sparsity $\gamma$ achieve the detection boundary defined by $\rho$: $\sqrt{c} > 1-\sqrt{\gamma(1-\rho)}$. Note that the detection boundary in (b) grows with $\rho$, which indicates that the signal needs to be stronger/denser to achieve the sweet spot under heavier dependence.

\section{Experiments}\label{sec:exp}
In \secref{calibration} we derive theoretical results of local test $t_{\alpha}^r$ \eqref{tralpha} under equicorrelated Gaussian model (\defref{modeldef}) in an asymptotic setting ($m\to\infty$), which indicate behaviors that positive $r$ works better for heavy dependence and weak dense signals, while negative $r$ works better for weak dependence and strong sparse signals. In this section, we empirically present the evidence that the above behaviors of local test are largely preserved after going through closed testing.

\begin{figure}[H]
	\includegraphics[width=\textwidth]{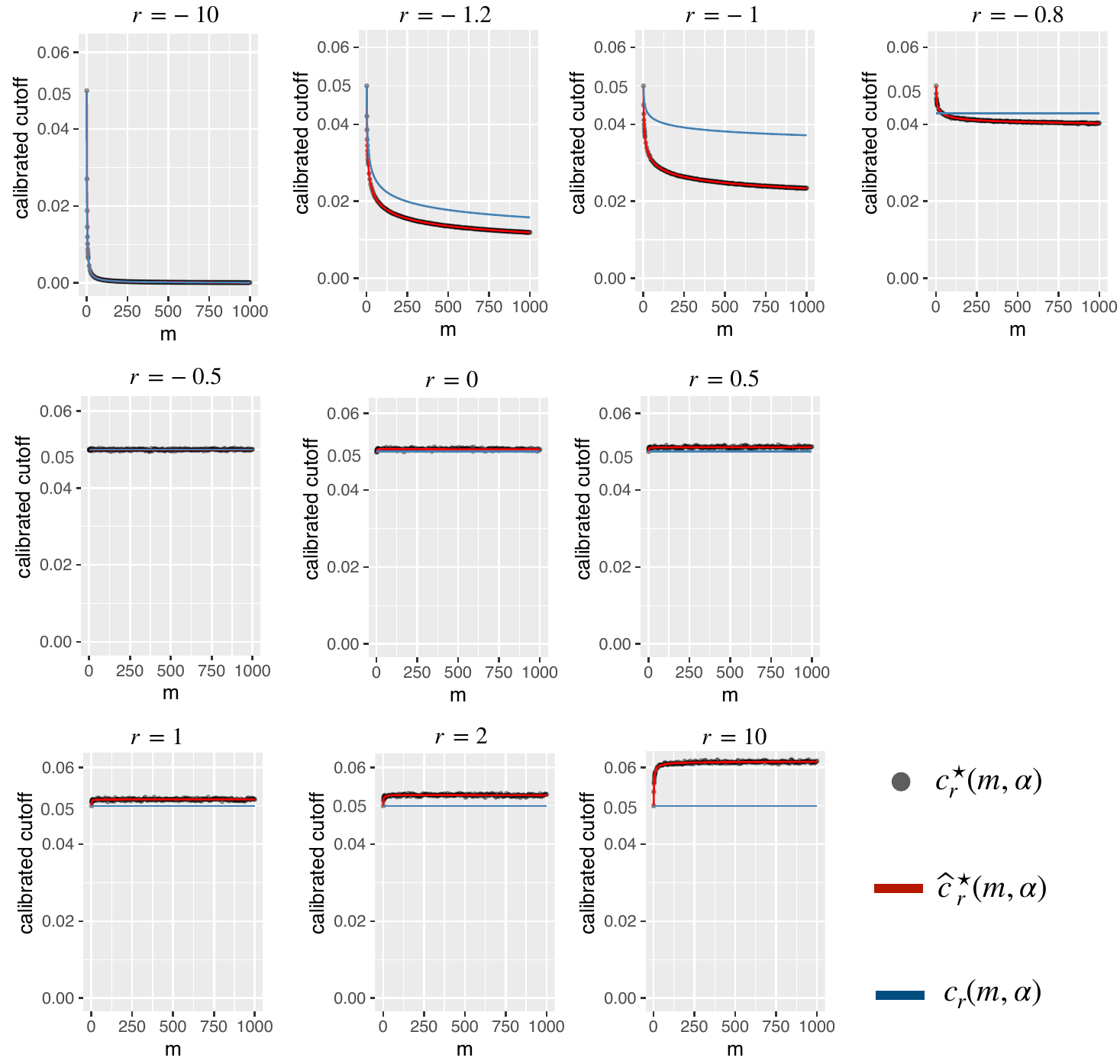}
	\caption{Calibration under $\rho$-equicorrelated Gaussian for $m\in[1000]$ for a worst case $\rho \in [0,1]$ (calculated using a grid of width 0.01). We compute the empirical $\alpha$ level calibrated cutoff $c_r^{\star}(m,\alpha)$ with $\alpha=0.05$ (black dots) for $(\frac1m \sum_{i=1}^m p_{i}^r)^{\frac1r}$ with grid points $m \in \{1,2,\dots,10,15,20,\dots,1000\}$ and $r\in\{-10,-1.2,-1,-0.8,-0.5,0,0.5,1,2,10\}$ via averaging over $10^6$ trials. Then we approximate $c_r^{\star}(m,\alpha)$ for all $m\in[1000]$ via fitting a smooth line $\widehat{c}^{\star}_{m}(\alpha,r)^{\frac{1}{r}}$ (red line) for the whole $m\in[1000]$, and use the fitted value as our final calibration. As for comparison, we also plot the theoretical calibrated cutoff $c_r(m,\alpha)$ (see definition \eqref{Cm}) derived for pointwise asymptotic type-I control in \secref{calibration}. In addition, for small $m$ ($1\sim10$), we just use empirical calibration for accurateness.}\label{fig:calfit}
\end{figure}

As the theoretical results for local test in \secref{calibration} are only asymptotic, while in closed testing, we need to consider all subsets of $[m]$; henceforth, our theoretical results will not be applicable for a large proportion of them. On the other hand, calibrating for subsets of size $1$ to $m$ is computationally expensive; therefore, we use the following approximation, that is calibrating for a few sizes in $[m]$ and then interpolating to the whole $[m]$ (see \figref{calfit} for the case when $m=1000$, $\alpha=0.05$). This empirical calibration with interpolation works well in terms of maintaining correct error control and gives nontrivial power (see \figref{fdpsig}--\ref{fig:fwercorr}).

\begin{remark}
Figure~\ref{fig:calfit} provides reasonable evidence that for $r \geq 1$, the worst-case dependence is not achieved by $\rho=1$ (perfect correlation) because if that was the case, there would be no violation of type-I error, but we observe above that the achieved error is larger than $\alpha=0.05$.
\end{remark}

In the rest of this section, we consider $m = 200$ tests, each based on different samples which are not necessarily independent, particularly, we assume the positively equicorrelated Gaussian model (\defref{modeldef}) among samples of the $m$ tests and test whether a given set of data has zero mean. In particular, we consider $\mu_m \equiv \mu$ for all $m$, and $\pi_m \equiv \pi$ for all $m$. 
In \figref{fdpsig}--\ref{fig:fwercorr} we investigate the algorithms presented in \thmref{gammaselection}, that is Algorithm~\ref{shortcuts-fdp} for finding the largest subset of $m$ such that FDP is controlled under $\gamma$, and Algorithm~\ref{shortcuts} for finding the largest subset of $[m]$ such that FWER is controlled under $\alpha$. Specifically, \figref{fdpsig}--\ref{fig:fdpcorr} show the results for finding the largest subset of $[m]$ with FDP controlled under $\gamma=0.2$, and \figref{fwersig}--\ref{fig:fwercorr} show the results for finding the largest subset of $[m]$ with FWER controlled under $\alpha=0.05$. We can see that, both FDP and FWER are controlled as we wanted, while $r<0$ often has non-trivial power (close to one) comparing with $r\geq 0$ when signals are strong enough (i.e. $\mu$ large enough), while they are both powerless otherwise  (i.e. $\mu$ not large enough). For weak signals specifically, we observe that $r>0$ have higher power comparing to $r<0$, especially under strong dependence (i.e. $\rho \gg 0$) and high signal density (i.e. $\pi \gg 0$). These findings generally agree with our asymptotic theory for local test in \secref{calibration}, that is $r<0$ achieves almost perfect power under setting with sparse strong signals, but is powerless when signals are not strong enough, in which case $r>0$ works better (especially given heavy dependence and dense signals). 
% \begin{figure}[H]
% 	\includegraphics[width=\textwidth]{plots/fdp_sig.png}
% 	\caption{The empirical FDP and power versus signal proportion $\pi$ using fitted calibration in \figref{calfit} and algorithm in \thmref{gammaselection}, with $\alpha=0.05$, $\gamma=0.2$, $m=200$, averaging over 1000 trials.}\label{fig:fdpsig}
% \end{figure}

\begin{figure}[H]
    \includegraphics[width=\textwidth]{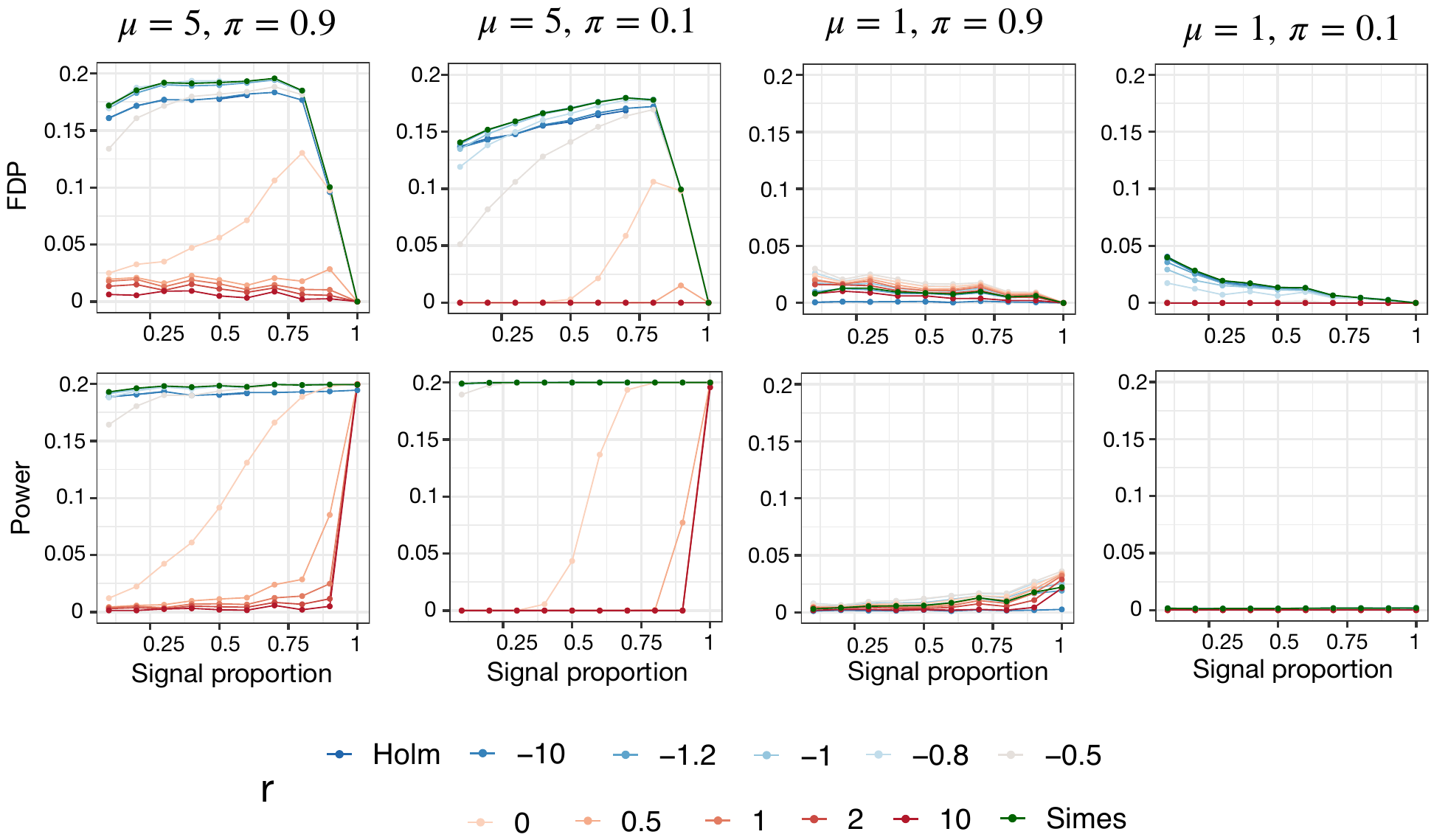}
	\caption{The empirical FDP and power versus the signal proportion $\pi$ under different settings using fitted calibration in \figref{calfit} and algorithm in \thmref{gammaselection}, with $\alpha=0.05$, $\gamma=0.2$, $m=200$, averaging over 1000 trials. }\label{fig:fdpsig}
\end{figure}

\begin{figure}[H]
    	\includegraphics[width=\textwidth]{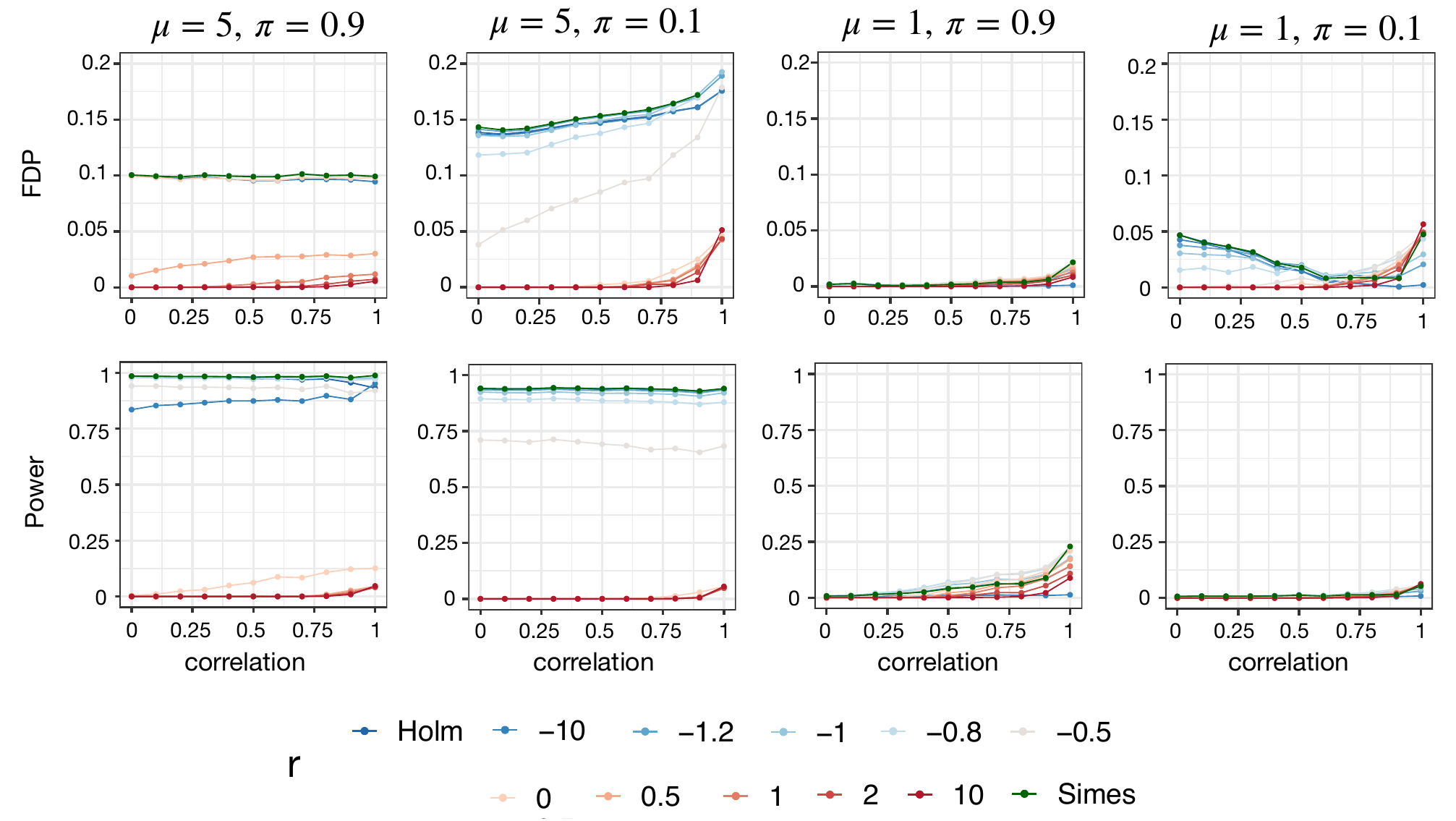}
	\caption{The empirical FDP and power versus correlation $\rho$ under different settings using fitted calibration in \figref{calfit} and algorithm in \thmref{gammaselection}, with $\alpha=0.05$, $\gamma=0.2$, $m=200$, averaging over 1000 trials.}\label{fig:fdpcorr}
\end{figure}

\begin{figure}[H]
\centering
	\includegraphics[width=\textwidth]{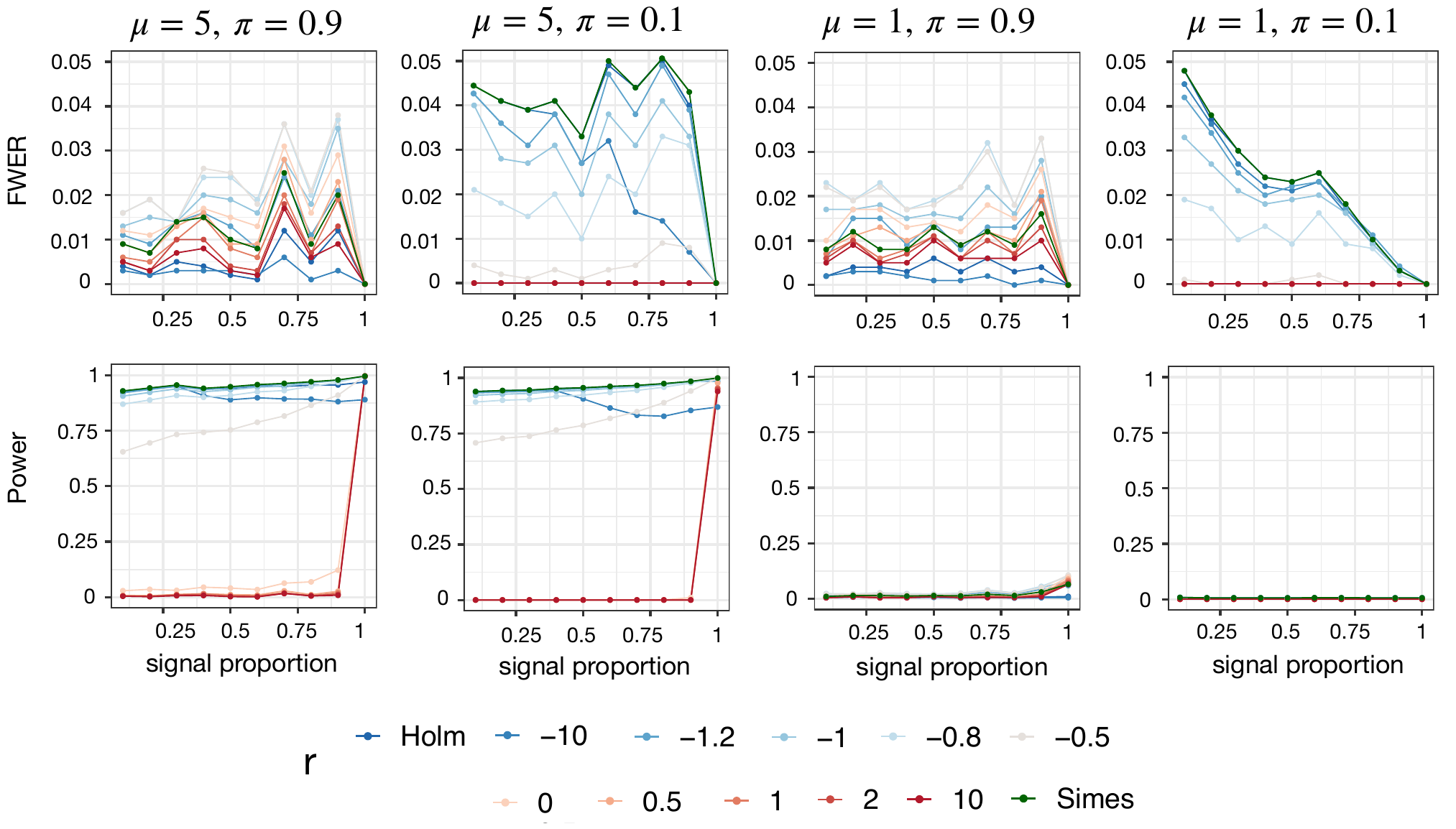}
	\caption{The empirical FWER and power versus signal proportion $\pi$ under different settings using fitted calibration in \figref{calfit} and algorithm in \thmref{gammaselection}, with $\alpha=0.05$, $m=200$, averaging over 1000 trials. }\label{fig:fwersig}
\end{figure}

\begin{figure}[H]
\centering
	\includegraphics[width=\textwidth]{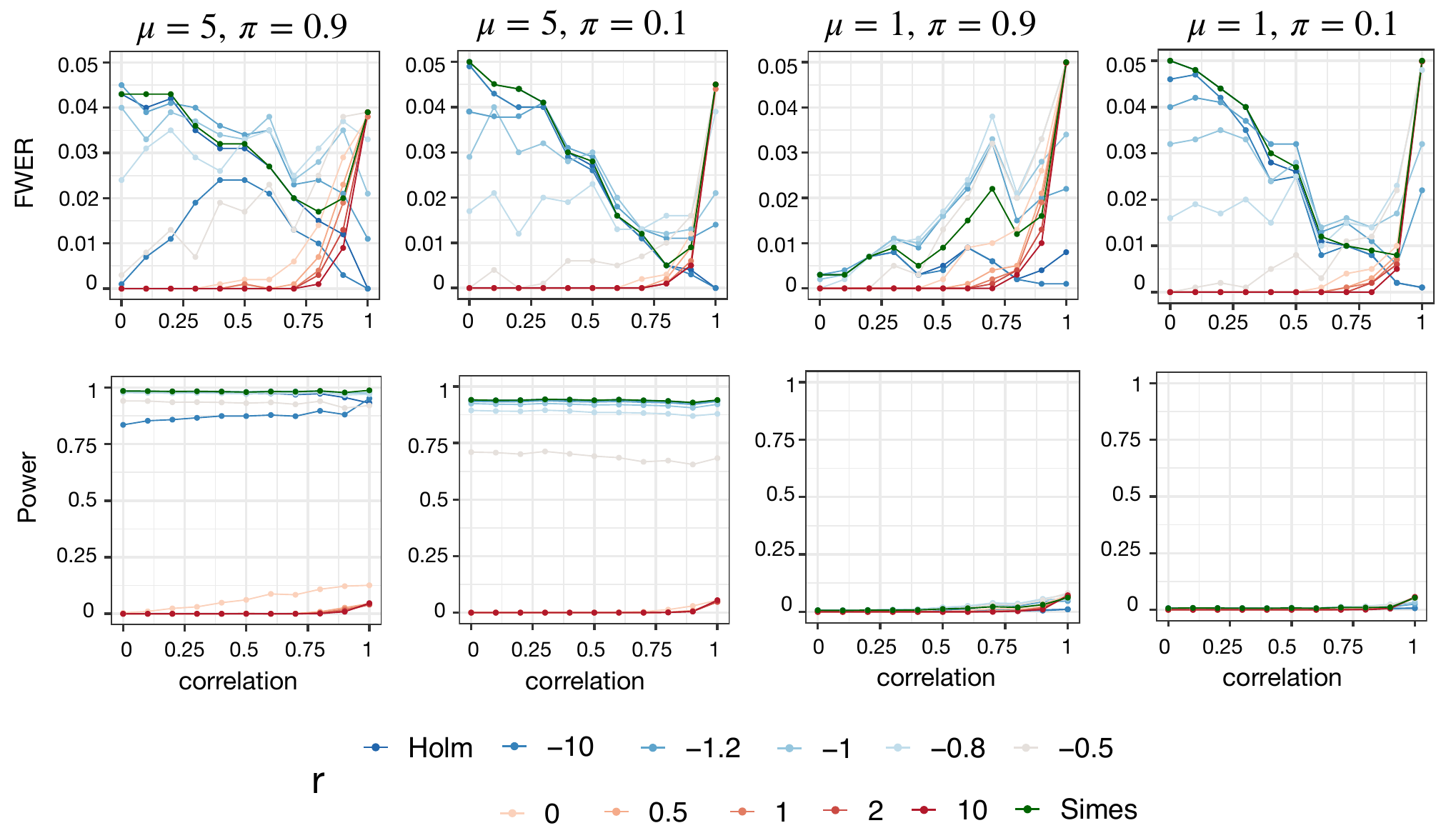}
	\caption{The empirical FWER and power versus correlation $\rho$ under different settings using fitted calibration in \figref{calfit} and algorithm in \thmref{gammaselection}, with $\alpha=0.05$, $m=200$, averaging over 1000 trials. }\label{fig:fwercorr}
\end{figure}

\section{Conclusion}\label{sec:con}
In this paper, we investigate the general case of closed testing with local tests that adopt a special property we called \emph{separability}, that is, the test is a function of summation of test scores for the individual hypothesis. With \emph{separability}, \emph{symmetry} and \emph{monotonicity} in local tests, we derive a class of novel, fast algorithms for various types of simultaneous inference. {\color{black}These algorithms have been implemented in the R package \texttt{sumSome}\footnote{\url{https://github.com/annavesely/sumSome}}.} We pair our algorithms with recent advances in separable global null test, i.e. the generalized mean based methods summarized \citep{vovk2020combining}, and obtain a series of simultaneous inference methods that are sufficient to handle many complex dependence structures and signal compositions. We provide guidance on choosing from these methods adaptively via theoretical investigation of the conservativeness and sensitivity for different choices of local tests in an equicorrelated Gaussian model. Specifically, we found that {\color{black} within the family of simultaneous inference methods using local tests introduced in \eqref{Talpha}:}
\begin{itemize}
    \item when signals are weak, all methods are powerless, while the ones with positive $r$ perform a bit better when signals are dense and highly correlated.
    
    \item when signals are strong, methods with negative $r$ are often able to achieve full power, and methods with positive $r$ are often still powerless; except in the case when signals are also dense, they are comparable.
\end{itemize}

{\color{black}We leave the following problems for future work. First, we note that the surrogate type-I error in \eqref{alpha} sometimes does not agree very well with the true type-I error in simulations (see the discussion of \figref{smallalpha}). We think this arises from the fact that the surrogate type-I error is based on asymptotic approximation. That is, the number of hypotheses $m$ goes to infinity. Meanwhile, we have only tried limited dimensions empirically (i.e., $m=10^5$). We think that the story will be more coherent when $m$ is much larger because we suspect that the convergence occurs very slowly. Nevertheless, how to derive tight and efficient calibration explicitly for small $m$ may be worth more attention. Secondly, in this work, we mainly focus on the equicorrelated Gaussian case, while the derivation of a tight calibration under arbitrarily correlated Gaussians will be intriguing,  though much harder. Finally, the theoretical power analysis is conducted only for the local test; formal theoretical analysis after closure would be desirable, though we expect that to be much harder.}

\bibliographystyle{plainnat}
\bibliography{main}

\begin{appendices}
\renewcommand{\thesection}{\Alph{section}}
\setcounter{section}{0}

\setcounter{equation}{39}
\section{Proof for \thmref{adjustedp}}\label{app:pfadjustedp}

%  By definition, $\ta{S} = 1$ is equivalent to $\overline{p}(S) \leq \alpha$ for any set $S$. Then noting that for any $\alpha \in [0,1]$,
% \begin{align*}
%     \ta{S} = 1 &\Leftrightarrow \{p(J) \leq \alpha, \textnormal{ for all } J \supseteq{S} \} 
%     \Leftrightarrow \{p(S) \leq \alpha,\  \ta{J} =1, \textnormal{ for all } J \supset{S} \}\\ 
%     &\Leftrightarrow \{p(S)\leq \alpha,\ \overline{p}(J) \leq \alpha, \textnormal{ for all } J \supset{S} \} 
%     \Leftrightarrow \alpha \geq  \max\{\sup_{J \supset{S}}\overline{p}(J), p(S)\}.
% \end{align*}
% Therefore, by definition, we have
% \begin{align}\label{pc}
%     \overline{p}(S) &= \inf\{\alpha \in [0,1]: \alpha \geq  \max\{\sup_{J \supset{S}}\overline{p}(J), p(S)\}\}
%     = \max\{\sup_{J \supset{S}}\overline{p}(J), p(S)\}.
% \end{align}

%  Given the monotonicity of $t_{\alpha}$ and \eqref{pc}, we have 
%  \[
%  \sup_{J \supset{S}}\overline{p}(J) = \sup_{1\leq i \leq|S^c|}\overline{p}(J_i),
%  \]
%  where $J_i$ is the set of all hypotheses in $[m]$ indices excluding the ones associate with the $i$ smallest scores in $S^c$. It is clear that RHS of above expression only require at most $|S^c|$ time to evaluate. Therefore, we have finished the proof. 
 
%  \qed

We have $\ta{S} = 1$ if and only if $p(J) \leq \alpha$ for all $S \subseteq J \subseteq [m]$, which happens if and only if $\alpha \geq \max_{S \subseteq J \subseteq [m]} p(J)$. Therefore,
\[
\overline{p}(S) = \max_{S \subseteq J \subseteq [m]} p(J).
\]
By monotonicity, we have that if $J \supseteq S$ then $p(J) \geq p(S \cup J^\star_{|J|-|S|})$, where $J_i^\star$, is the set of the indices of the $i$ largest $p$-values in $S^c$. Therefore, 
\[
\overline{p}(S) = \max_{|S| \leq i \leq m} p(S \cup J^\star_{i-|S|})
= \max_{0 \leq i \leq |S^c|} p(S \cup J^\star_i).
\]
Since $J^\star_0 \subseteq \ldots \subseteq J^\star_{|S^c|}$, it is clear that this expression can be calculated in $O(|S^c|) = O(m)$ time.  
\qed

\section{A more general framework of local tests design}\label{app:compdef}
We start with defining some terminology. Recall that a local test is an indicator function of whether to reject $S$ or not:
\begin{equation}
t_{\alpha}(S):\quad  \mathbb{R}^{|S|} \to \{0,1\}.
\end{equation}
% We treat the hypotheses
% identically, that is assuming a \emph{global symmetry} structure according to \citet{dobriban2020fast}. Then it makes sense to use the same local testing rule for each subset of a fixed size $s$. Essentially 
We consider $t_{\alpha}$ of the following form:
\begin{equation}\label{globalsym}
    t_{\alpha}(S) = \ones\{f(|S|;
    % T_{i_1}, \dots, T_{i_{|S|}}
    (T_i)_{i \in S} 
    ) \leq \alpha\},
\end{equation}
where the function $f$ depends on the size of $S$ and the vector of scores.
% So we write $t_{\alpha}(S) = t_{\alpha}(|S|; T_{i_1}, \dots, T_{i_{|S|}})$, indicating the local testing rule is a function whose expression only changes with size of $S$.
% \noindent In other words, $t_\alpha(S)$ 
% \begin{assumption}\label{assump:globalsym}\textbf{(Global symmetry)} The problem of inference on $m$ hypotheses $H_1,\dots,H_m$ is the same with that on  $H_{i_1},\dots,H_{i_m}$, where $(i_1,\dots,i_m)$ is a permutation of $(1,\dots,m)$.
% \end{assumption}
Given local tests of this form, there are two commonly satisfied conditions --- \emph{symmetry} and \emph{monotonicity} --- which makes the computation manageable: quadratic time shortcuts for simultaneous FDP inference and FWER control have been developed by \citet{goeman2011multiple} and  \citet{dobriban2020fast} respectively under  these two conditions.
\begin{condition}\label{cond:mono}
\textbf{(Monotonicity)} A local test of form \eqref{globalsym} is called monotonic if for any $s\geq 1$, any two sets of scores $(T_1, \dots, T_s)$ and $(T_1',\dots, T_s')$ with $T_i \leq T_i'$ for all $i=1,\dots,s$, we have
\begin{equation}
    f(s; T_1, \dots, T_s ) \geq f(s; T_1',\dots, T_s').
\end{equation}
\end{condition}
\emph{Monotonicity} is a reasonable requirement for a local test. Several well-known global null tests are monotonic, for example, Fisher's test, Simes' test, Bonferroni test, etc; as well as the tests based on generalized means \cite{vovk2020combining}.

\begin{condition}\label{cond:sym}
\textbf{(Symmetry)} A local test $t_{\alpha}$ of form \eqref{globalsym} is called symmetric if for any $s\geq 1$, any set of scores $(T_1, \dots, T_s)$, and any permutation $(i_1,\dots,i_s)$ of $(1,\dots,s)$, we have
\begin{equation}
    f(s; T_1, \dots, T_s) = f(s; T_{i_1}, \dots, T_{i_s}).
\end{equation}
\end{condition}

Another condition that we find could further reduce computation time is \emph{separability}. 
% It imposes more structural features on the local test, which can be used to derive faster shortcuts (linear time) for both simultaneous FWER control and simultaneous FDP bound calculation. 
This condition is relatively under-emphasized in the past closed testing literature; however, it has a long history in the global null testing literature with an increased recent interest \citep{vovk2020combining, wilson2019harmonic,chen2020trade}.

% usually follows the following form,  
% \begin{equation}\label{talphareform}
% t_{\alpha}(S) = \ones\{ f(T_{s_1}, \dots, T_{s_{|S|}}) \leq g(|S|,\alpha)\}
% \end{equation}
% where $f, T, g$ are defined in \secref{intro}. Note that we require $f,g$ to be fixed, and $f$ to be function of only scores related to each element in $S$, and $g$ to be function of only the cardinality of $S$ and the confidence level. We argue that this requirement is not trivial yet fair enough as it covers the most of current local test methods, and allows nontrivial shortcuts and power analysis in this paper to be conducted. Based on special properties of $f$, we define the following designs of local test that we will depend on heavily.

\begin{condition}
\label{cond:sep}
\textbf{(Separability)} A local test of form \eqref{globalsym} is called separable if for $s\geq 1$, and a set of scores $(T_1, \dots, T_s)$, there exists a series of transformation functions on $\mathbb{R}$ $\{h_i\}_{j=1}^s$ and a function $g$ on $\mathbb{R}^2$ such that
\begin{equation}\label{sep}
t_{\alpha}(s; T_1, \dots, T_s) = \ones\left\{ \sum_{j=1}^{s} h_j(T_{j}) \leq g(|S|,\alpha)\right\}.
\end{equation}
\end{condition}
 Recall the class of local tests $T_{\alpha}$ defined in \eqref{Talpha} of the main paper. It is easy to check that each of its element $t_{\alpha}^{(r)}$ is \emph{monotonic}, \emph{symmetric} for all $r$, and \emph{separable} iff $r\neq \pm \infty$.
\begin{remark}
A local test $t_{\alpha}$ of form \eqref{globalsym} with both \emph{symmetry} and \emph{separability} must admit the following form:
\begin{equation}
t_{\alpha}(S) = \ones\left\{ \sum_{i=1}^{|S|} h(T_{s_i}) \leq g(|S|,\alpha)\right\},
\end{equation}
that is the transformation functions in the summation are the same for each hypothesis.
\end{remark}

\section{Proof for \thmref{shortcut}}\label{app:shortcutpf}
Without loss of generality, we assume that all the scores (e.g. $p$-values) are already sorted in a descending order, that is \smash{$T_1 \geq T_2 \geq \dots \geq T_m$}. Denote $S = \{i_1, i_2, \dots, i_s\}$, with $1\leq s\leq m$ and $i_1 < i_2, <\dots< i_s$; and $S^{c} = \{j_1, j_2, \dots, j_{m-s}\}$ as the complement set of $S$, with $j_1 < j_2, <\dots< j_{m-s}$. 

% \paragraph{Proof for part (a).} By definition of $\overline{t}_{\alpha}(S)$, we can write it as
%      \begin{equation}\label{tclinear}
%       \overline{t}_{\alpha}(S) = \prod_{J:  S \subseteq J \subseteq [m]}t_{\alpha}(J) = \prod_{1 \leq i \leq |S^c|}\  \prod_{J_j: J_i \subseteq S^c, |J_i|=i } t_{\alpha}(S \cup J_i ).
%      \end{equation}
%      Using the monotonicity of local test $t_{\alpha}$, we have
%      \begin{equation}\label{tcmono}
%          t_{\alpha}(S \cup J_i ) \geq t_{\alpha}(S \cup J_i^{\star}), \quad \text{where } J_i^{\star} = \{j_{1}, \dots j_{i}\}.
%      \end{equation}
%      for all $J_i \subseteq S^c, |J_i|=i$ and $1 \leq i \leq |S^c|$, where $J_i^{\star}$ is the set of hypotheses indices that is associated with the $i$ largest scores in $S^c$.
     
%      Combining \eqref{tclinear} and \eqref{tcmono}, and the fact that $t_{\alpha}$ and $\overline{t}_{\alpha}$ are both $0/1$ indicator, therefore we get 
%      \begin{equation}
%       \overline{t}_{\alpha}(S) = \prod_{1 \leq i \leq |S^c|}t_{\alpha}(S \cup J_i^\star) := \prod_{1 \leq i \leq |S^c|}t_{\alpha}( I_i),\quad \text{where } I_i = S \cup J_i^\star
%      \end{equation}
%      as we claimed, which requires at most $O(m)$ computation in terms of local tests. 

% \paragraph{Proof for part (b)}

To prove the validity of Algorithm~\ref{shortcuts-fdp}, we first focus on the crucial line 9, and claim that when event 
\begin{equation}\label{line9}
\mathcal{E}:= \{\textnormal{ line 9 is evaluated with } k \leq \min\{s,a\}\}
\end{equation}
happens, the following four statement are true.
\begin{itemize}
    \item[(i)] $Q  = \sum_{j=1}^{k+b} u_j + \sum_{l=1}^{a-k-b} v_l;$
    
    \item[(ii)] $b  \geq 0;$
    
    \item[(iii)] $v_{a-k-b} \geq u_{k+b+1};$
    
    \item[(iv)] $u_{k+b} \geq v_{a-k-b+1}, \quad \text{if}\ b > 0$.
\end{itemize}
Note that  $\sum_{l=1}^0 v_l$ is defined to be 0, and not $v_1+v_0$; and $\sum_{j=1}^0 u_j$ is defined to be 0, and not $u_1+u_0$. 

We show this by induction. The first time event $\mathcal{E}$ defined in \eqref{line9} happens, we have $k=a=1, b=0$ and $Q = u_1$, so (i) and (ii) hold, (iii) holds since $v_0 \geq u_2$, and for (iv), since $b=0$ there is nothing to prove. 

{\color{black} Additionally, we need to prove that when $b>0$ for the first time, (iv) is satisfied. Note that the only way that $b>0$ for the first time, is through the satisfaction of the condition in line 3, specifically $u_{k+b_0+1}\geq v_{a-k-b_0}$ with $b_0=0$. The reason is that $a$ always increases by one in each while-loop of line 2, and after the first while-loop, $b$ can be at most 0, and we have $a=2$. Therefore under the condition that $u_{k+b_0+1}\geq v_{a-k-b_0}$ with $b_0=0$, the algorithm will go through line 5, and has $b=b_0+1>0$, and $u_{k+b}\geq v_{a-k-b+1}$. That is (iv) is satisfied, when $b$ satisfies $b>0$ for the first time.}

Now assume that (i)-(iv) hold the previous time the event $\mathcal{E}$ happens. Let $a_0, k_0, b_0$ and $Q_0$ be the value of $a,k,b$ and $Q$ during that previous step. There are five routes for event $\mathcal{E}$ to happen again, which we can characterize by the way $a,k,b$ are updated. We will discuss these routes one by one.

\begin{itemize}
    \item[1.] \textbf{Line} $\bm{9 \to 10 \to 11 \to 15 \to \mathcal{E}}$. In this case we update $a=a_0; b=b_0 -1; k=k_0+1$. We have $b \geq 0$ since $b_0 > 0$. We have $Q =Q_0$, and (i) holds since $k+b = k_0+b_0$. By the induction hypothesis,
    \[
    v_{a-k-b} = v_{a_0 -k_0 -b_0} \geq u_{k_0+b_0+1} = u_{k+r+1}. 
     \]
     If $b>0$, then also $b_0>0$, so, by the induction hypothesis,
     \[
    u_{k+b} = u_{k_0 + b_0} \geq v_{a_0 -k_0 -b_0+1} =v_{a -k -b+1}. 
     \]
     
     \item[2.] \textbf{Line} $\bm{9 \to 10 \to 13 \to 15 \to \mathcal{E}}$. In this case we update $a=a_0$; $b=b_0=0$; $k=k_0+1$. Clearly, (ii) holds. We have 
     \[
     Q = \sum_{j=1}^{k_0+1} u_j + \sum_{l=1}^{a_0-k_0-1} v_l,
     \]
     which reduces to (i). By the induction hypothesis,
     \[
     v_{a-k-b} = v_{a_0-k_0-b_0-1} \geq v_{a_0-k_0-b_0} \geq u_{k_0+b_0+1} = u_{k+b} \geq u_{k+b+1},
     \]
     so (iii) follows. Since $b=0$ there is nothing to prove (iv).
     
     \item[3.] \textbf{Line} $\bm{9 \to 2 \to 3 \to 4,5 \to \mathcal{E}}$. In this case we update $a=a_0+1; b=b_0+1; k=k_0$. We get (ii) from the induction assumption since $b=b_0+1\geq 1$. We obtain (i) since 
     \[
     Q = \sum_{j=1}^{k_0+b_0+1} u_j + \sum_{l=1}^{a_0-k_0-b_0}v_l.
     \]
     By the induction hypotheses,
     \[
     v_{a-k-b} = v_{a_0-k_0-b_0} \geq u_{k_0+b_0+1} = u_{k+b} \geq u_{k+b+1}.
     \]
     Also, $b>0$ and 
     \[
     u_{k+b} = u_{k_0+b_0+1} \geq v_{a-k_0-b_0} = v_{a-k-b+1}. 
     \]
     
     \item[4.] \textbf{Line} $\bm{9 \to 2 \to 3 \to 7 \to \mathcal{E}}$. We update $a = a_0+1$, $k=k_0$ and $b=b_0$. We get (ii) trivially, and (i) since 
     \[
     Q = \sum_{j=1}^{k_0+b_0} u_j + \sum_{l=1}^{a_0-k_0-b_0+1} v_l.
     \]
     We have (iii) since 
     \[
     v_{a-k-r} = v_{a-k_0-b_0} \geq u_{k_0+b_0+1} = u_{k+b+1}.
     \]
     By the induction hypothesis, we get (iv), since, 
     \[
     u_{k+b} = u_{k_0+b_0} \geq v_{a_0-k_0-b_0+1} = v_{a-k-b} \geq v_{a-k-b+1}.
     \]
     
     \item[5.] \textbf{Line} $\bm{9 \to 10 \to 13 \to 15 \to 9 \to 2 \to 7 \to \mathcal{E}}$. First we use proof by contradiction that this case happens only if $k_0 = \min\{a_0,s\}$: {\color{black} Otherwise, for step ``$\bm{9 \to 2}$'' to happen in the route, we need $Q_0+u_{k_0+1}-v_{a_0-k_0}\leq C(a_0,\alpha)$ to break the while loop in line 9. This cannot happen since we need $b_0=0$ to reach line 13 in this route, which indicates $b_0$ was not updated in line 5 (otherwise $b_0>0$ from our induction hypothesis (ii)), i.e.  $u_{k_0+b_0+1}> v_{a_o-k_0-b_0}$ in line 3. Therefore we in fact have $Q_0+u_{k_0+1}-v_{a_0-k_0}>Q_0>C(a_0,\alpha)$, which is a contradiction to what we require.} 
     
     Consequently, we update $a=a_0+1$, $k=k_0+1$, and, since $u_{k+1} \leq v_0$ by definition of $v_0$, $b=b_0=0$. So first (ii) holds, and also we get (i) since 
     \[
     Q = \sum_{j=1}^{k_0} u_j + \sum_{l=1}^{a_0-k_0} v_l +u_{k_0+1}-v_{a_0-k_0} + v_{a-k} = \sum_{j=1}^{k} u_j + \sum_{l=1}^{a-k} v_l =\sum_{j=1}^{k+b} u_j + \sum_{l=1}^{a-k-b} v_l.
     \]
     Moreover, since $a=k+b$ and $b=0$, we have (iii) because $v_{a-k-b} = v_0\geq u_1 \geq u_{k+b+1}.$ There is nothing to prove (iv) since $b=0$.
    
\end{itemize}

Since we have exhausted the possibilities to get from one happening of event $\mathcal{E}$ to the next, the above analysis proves (i)-(iv). It follows from (i)-(iv) that, in line 9, $Q = W_{a,k}$, where 
\[
W_{a,k} = \max\{ Q_{I}: |I \cap S|\geq k, |I| = a\}.
\]
To see why this is true, note that (i) $Q$ is a sum of $a$ terms, of which at least $k$ terms are from $S$. The sum is the largest possible such sum since the $k$ largest scores in $S$ are used, and by (iii) and (iv), of the $a-k$ remaining scores, the smallest score that is included in the sum is larger than or equal to the largest score that is not included. Note that, if $k \leq k'$, $W_{a,k} \geq W_{a,k'}$.

Now, suppose $\overline{e}_{\alpha}(S) = e > 0$. Then there exists some $I \subseteq S$ with $|I| = e$ and some $J \supseteq I$ such that $Q_J > g(|J|,\alpha)$. In the algorithm, if $a = |J|$ and $k \leq e$, we have $Q = W_{a,k} \geq W_{a,e}\geq Q_J > g(|J|,\alpha)$, so the algorithm enters the \emph{while} loop in line 9, incrementing $k$ while keeping $a$ fixed. This step is repeated at least until $k \geq e+1$. Since $k$ is non-decreasing in the steps of the algorithm, it returns $k-1 \geq e$. The same holds trivially if $e=0$. 

If $\overline{e}_{\alpha}(S) = e < s$, then for every $I \subset S$ with $|I| > e$ we have $\overline{t}_{\alpha}(I) = 1$, so for all $J \supseteq I$, we have $Q_J \leq g(|J|, \alpha)$. In particular, this holds for the worst case set, so for every $e+1\leq a \leq m$, we have $W_{a,e+1} \leq g(a,\alpha)$. If $k=e+1$, therefore, the algorithm never enters the \emph{while} loop in line 9, and consequently never increments $k$ further. The algorithm therefore ends with $k \leq e+1$ and returns $k-1 \leq e$. The same holds trivially if $e=s$. 

To sum up, since $k-1 \geq e$ and $k-1 \leq e$, we have $k-1 = e$.  
\\

{\color{black} Finally, we prove that the algorithm takes $O(m)$ time to run. First, note that there are $m$ for-loop iterations. In each for-loop iteration, it is obvious that apart from the while-loop, the algorithm takes constant time. For the while-loop part, we can additionally show that the total calculations that it takes over all the for-loop iteration is at most $s$. A key observation for the proof is that $k$ can only be updated through the while-loop part, and $k$ always increases by one each time going through one while-loop iteration. Since the global upper bound for $k$ is $s$, the number of total iterations for the while-loop over all for-loop iterations is at most $s$. Since each while-loop iteration also takes constant time, the algorithm takes at most $m+s$ steps of $O(1)$ calculation. In conclusion, the algorithm takes $O(m)$ time to run.}

\section{Discussion of consonance}\label{app:closedconso}
Firstly we formally state the definition of consonance.
\begin{definition}\label{def:conso}
\textbf{(Consonance)} A closed testing procedure is \emph{consonant} if the local tests for every composite hypothesis $S\in 2^{[m]}$ are chosen in such a way that rejection of $S$ after closure implies rejection of at least one of its elementary hypotheses after closure.
\end{definition}

\begin{lemma}\label{lem:closedconsonant}
 The closed testing procedure using local test $t_{\alpha}^{(r)}$ defined in \eqref{tralpha} is consonant if and only if $r = \pm{\infty}$.
\end{lemma}
\begin{proof}
We will prove this proposition by analyzing different $r$ case by case. Firstly, we show that closed testing using local test $t_{\alpha}^{(r)}$ \eqref{tralpha} is consonant when $r=\pm\infty$.
\begin{enumerate}
    \item When $r=\infty$, note that
\begin{equation}\label{inf}
t_{\alpha}^{(\infty)}(S) = \ident\{\max_{i\in S} p_i \leq \alpha\}.    
\end{equation}
Therefore, rejecting $S$ after closure implies rejecting all the sets containing it locally, including the set $[m]$, which in turn indicates rejection of all the sets locally, and after closure as well, therefore trivially, we have all the subsets of $S$ being rejected after closure. In conclusion, the corresponding closed testing when $r=\infty$ is consonant.

    \item When $r=-\infty$, note that
\begin{equation}\label{-inf}
    t_{\alpha}^{(-\infty)}(S) = \ident\{|S|\min_{i\in S} p_i \leq \alpha\}.
\end{equation}
Therefore $\overline{t}_{\alpha}^{-\infty}(S)=1$ implies
\begin{equation}\label{-inf1}
     t_{\alpha}^{-\infty}(B)=1 \quad \text{for all } B \in \mathcal{B} := \{I \subseteq[m]: S \subset I\},
\end{equation}
and particularly $t_{\alpha}^{-\infty}(S)=1$, which in turn gives us
\begin{equation}\label{-inf2}
     t_{\alpha}^{-\infty}(A)=1 \quad \text{for all } A \in \mathcal{A} :=\{I \subseteq [m]: I\subset S, \min_{i \in S}p_i \in I \}
\end{equation}
from the expression in \eqref{-inf}.

On the other hand, note that for some $A \in \mathcal{A}$, and $J \supseteq A$, we have either $J \supseteq S$, or $|J| \not\supseteq S$. In the former case, $J$ is rejected locally due to fact \eqref{-inf1}. In the later case, if $|J| \leq |S|$, we have $|J|\min_{i\in J} p_i \leq |S|\min_{i\in A} p_i = |S|\min_{i\in S} p_i \leq \alpha$; if $|J| > |S|$, then there must exist $B \in \mathcal{B}$ such that $|B| = |J|$, which implies $|J|\min_{i\in J} p_i = |B|\min_{i\in B} p_i \leq \alpha$ due to fact \eqref{-inf2}. Therefore $J$ is still rejected locally. Hence there exists at least one subset $A$ of $S$ that is rejected after closure, that is the corresponding closed testing is consonant.
    
\end{enumerate}
\noindent
Then we use counter examples to show that closed testing using local test $t_{\alpha}^{(r)}$ \eqref{tralpha} is not consonant when $r\neq \pm\infty$.
\begin{enumerate}
    \item When $0< r < \infty$, note that 
     \[
       t_{\alpha}^{(r)} = \ident\Big\{\sum_{i=1}^m p_i^r \leq  \frac{m\alpha^r}{2(r+1)}\Big\}
     \]
     Let \smash{$\beta_{r,\alpha} = \frac{\alpha^r}{2(r+1)}$}, the local testing rule becomes $\sum_{i=1}^m p_i^r \leq  m \beta_{r,\alpha}$. Note that $0 \leq \beta_{r,\alpha} \leq \frac{1}{2(r+1)} \leq 1/2$ for any $r>0$. For the case $m=3$, and $p_1^r = \beta_{r,\alpha} /3$, $p_2^r = \beta_{r,\alpha} /2$, $p_3^r = 2\beta_{r,\alpha}$, we have $p_1^r + p_2^r + p_3^r < 3\beta_{r,\alpha}$, $p_1^r + p_2^r < 2\beta_{r,\alpha}$, while $p_1^r + p_3^r > 2\beta_{r,\alpha}$ and $p_1^r < \beta_{r,\alpha}$. Therefore, we reject $H_1\cap H_2 \cap H_3$, $H_1\cap H_2$ but neither $H_1$ nor $H_2$ after closure, therefore the rejection $H_1\cap H_2$ is not consonant.

     \item When $r =0$, note that 
     \[
       t_{\alpha}^{(0)} = \ident\Big\{\sum_{i=1}^{m}\log{\frac{1}{p_i}} >  m\log{\frac{e}{\alpha}}\Big\}.
     \] 
     Let \smash{$\beta_{\alpha} =\log{\frac{e}{\alpha}}$}, and $q_i = \log{(\frac{1}{p_i})}$, then the local testing rule becomes $\sum_{i=1}^{m}q_i \geq  m\beta_{\alpha}$. Note that $1\leq \beta_{\alpha}<\infty$, and $1\leq q_i \leq \infty$. For $m=3$, let $\alpha  = e^{0.9}$, and $q_1 = 1.6\beta_{\alpha}$, $q_2 = 1.4\beta_{\alpha}$, $q_3 = 0.1\beta_{\alpha}$, therefore we will reject $H_1\cap H_2 \cap H_3$ and $H_1\cap H_2$ after closure, but neither $H_1$ nor $H_2$, which indicates that the rejection $H_1\cap H_2$ is not consonant. 
     
     \item When $-1 < r < 0$, note that 
     \[
       t_{\alpha}^{(r)} = \ident\Big\{\sum_{i=1}^m p_i^r \geq  \frac{m\alpha^r}{2(r+1)}\Big\}.
     \] 
     Let \smash{$\beta_{r,\alpha} = \frac{\alpha^r}{2(r+1)}$}, then the local testing rule becomes \smash{$\sum_{i=1}^m p_i^r \geq  m \beta_{r,\alpha}$}. Note that $1/2 \leq \frac{1}{2(r+1)} \leq \beta_{r,\alpha} < \infty$ for any $-1<r<0$. For the case $m=3$, let $\alpha = \sqrt[r]{20(r+1)}$, and $p_1^r = 1.6\beta_{r,\alpha}$, $p_2^r = 1.4\beta_{r,\alpha} $, $p_3^r = 0.1\beta_{r,\alpha}$. Note the fact that $\beta_{r,\alpha}\geq 10$, we have that $p_1^r + p_2^r + p_3^r > 3\beta_{r,\alpha}$, $p_1^r + p_2^r > 2\beta_{r,\alpha}$, while $p_1^r + p_3^r < 2\beta_{r,\alpha}$ $p_2^r + p_3^r < 2\beta_{r,\alpha}$. Therefore, we reject $H_1\cap H_2 \cap H_3$, $H_1\cap H_2$
but neither $H_1$ nor $H_2$ after closure, therefore the rejection $H_1\cap H_2$ is not consonant.
     
     \item When $r = -1$, note that 
     \[
       t_{\alpha}^{(-1)} = \ident\Big\{\sum_{i=1}^m \frac{1}{p_i} \geq \frac{e m\log{m}}{\alpha}\Big\}.
     \] 
     Let $\beta_{\alpha} = \frac{e}{\alpha}$, $q_i = 1/p_i$, then the testing rule becomes $\sum_{i=1}^m q_i \geq m\log{m}\beta_{\alpha}$. For the case $m = 5$, let $q_1 = q_2 = 2\beta_{\alpha}\log{5}$, $q_3 = q_4 = q_5 = \frac{1}{3}\beta_{\alpha}\log{5}$, then we have 
     \[\sum_{i=1}^5 q_i  = 5\log{5}\beta_{\alpha},\quad \sum_{i=1}^4 q_i  = 4\frac23\log{5}\beta_{\alpha}\geq 4\log{4}\beta_{\alpha},\quad \text{and}\sum_{i=1, i\neq 2}^5 q_i  = 3\log{5}\beta_{\alpha}\leq 4\log{4}\beta_{\alpha}.
     \] 
     Therefore, we must locally reject $\cap_{i=1}^{5}H_i$, $\cap_{i=1}^{4}H_i$, but not $\cap_{i=1, i\neq 2}^{5}H_i$. Therefore, after closure, we will reject $\cap_{i=1}^{5}H_i$ and $\cap_{i=1}^{4}H_i$, but we will not reject $H_1$, nor $H_2$, $H_3$ or $H_4$, therefore the rejection $\cap_{i=1}^{4}H_i$ is not consonant.
     
     \item When $-\infty < r < -1 $, note that 
     \[
       t_{\alpha}^{(r)} = \ident\Big\{\sum_{i=1}^{m}p_i^r \geq m^{-r}\alpha^{r+1}\Big\}.
     \] 
     Let $t = -r$, $\beta_{t,\alpha} = \alpha^{-t+1}$, $q_i  = 1/p_i$, then the local test becomes $\sum_{i=1}^m q_i^t \geq m^t \beta_{t,\alpha}$. For the case $m\geq \max\{3, \frac{\sqrt[t]{3}}{\sqrt[t]{3} -1}\}$, let $q_1^t = q_2^t \frac12 m^t \beta_{\alpha}$, and $q_3^t = \dots = q_m^t  = \frac{1}{6(m-2)}m^t \beta_{\alpha}$ (choose  $\alpha$ such that $\beta_{t, \alpha} > 6$), then we have that, $\sum_{i=1}^m q_i^t \geq m^t\beta_{\alpha}$ ,  $\sum_{i=1}^{m-1} q_i^t \geq m^t\beta_{\alpha}$, $\sum_{i=1, i\neq 2}^{m} q_i^t < m^t\beta_{\alpha}$. Therefore, we will reject $\cap_{i\in[m-1]}H_i$ after closure, but we cannot reject $H_1$ after closure, therefore the rejection $\cap_{i\in[m],i\neq 2}H_i$ is not consonant.

\end{enumerate}
\end{proof}

\section{Algorithms for post-hoc auto-selection shortcuts}\label{app:algoauto}
\setcounter{algocf}{1}
\begin{algorithm}[H]
	\SetAlgoLined
	\KwIn{A sequence of sorted scores $T_1, \dots, T_m$ which satisfies $T_{1}\geq\dots\geq T_{m}$;
	a local test rule of form \eqref{sepsym} with a monotonically increasing transformation function $h$ and thresholding function $g$; confidence level $\alpha$.}
	\KwOut{Largest set $S$ with zero false discoveries among all possible subsets of $[m]$, equivalently,  the set of hypotheses with strong FWER control of level $\alpha$.}
	{\textbf{Initialization:}\\
	\nonl $ \textnormal{transformed scores }  u_{1}, \dots, u_{m}$ where $u_{i}= h(T_{i})\  \textnormal{for}\ 1\leq i\leq m$;\\
	\nonl iteration related indices $k \gets 1; \ s \gets 1; $\\
	\nonl accumulated scores $Q \gets u_{k}$;\\
	\nonl layer-wise thresholding 
	$c \gets g(s, \alpha)$. }
	
	\While{$k<m$ \textbf{and} $s<m$}{
		\eIf{$Q > c$}{
			\eIf{$s\geq k$}{$c\leftarrow c - u_{k}$
			\\
			$Q \leftarrow Q - u_{k}$}
			{$Q \leftarrow Q - u_{k} + u_{k+1}$;}
			$k \leftarrow k+1$}	
		{ 
			$c \leftarrow c + g(s+1, \alpha) - g(s, \alpha)$
			
			\eIf{$s \geq k$}{
				$Q \leftarrow Q + u_{k+1}$
			}
			{
				$c \leftarrow c - u_{s}$
			}
			$s \leftarrow s+1$   
		}
	}
	\Return{ $S = \{k, \dots, m\}$}
	\caption{Shortcut for auto-selection of the largest rejection set with zero FDP}\label{shortcuts}
\end{algorithm}
\newpage
\begin{algorithm}[H]
	\SetAlgoLined
	\KwIn{confidence level $\alpha \in (0,1)$; desired FDP bound $\gamma \in (0,1)$; incremental candidate sets \smash{$S_1\subset S_2 \dots\subset S_n$ with $|S_k| = k$.}}
	\KwOut{the largest $S_k$ such that $\overline{e}_{\alpha}(S_k)\leq \gamma|S_k|$. }
	
    {\textbf{Initialization:} $k \gets 1$}\\
	\While{$k \geq 1$}{
		$\overline{e} \leftarrow \overline{e}_{\alpha}(S_k)$\\
		\eIf{$\overline{e}/k \leq \gamma $}{
			\Return{$S_k$}}
	  	{
			$k \leftarrow \lfloor \frac{k-\overline{e}}{1-\gamma}\rfloor$		
	     }
		}
	\Return{$\emptyset$} 
	\caption{Shortcut for auto-selection of the largest rejection set with bounded FDP}\label{shortcuts-gamma}
\end{algorithm}

% \section{Proof for \thmref{gammaselection}}\label{app:selectionpf}
		
\section{Proof for \thmref{gammaselection}}\label{app:pfgammaselection}
We first prove part (a) of \thmref{gammaselection}. Denote $e_k = \overline{e}(S_k)$, and $d_k =k-e_k$. Consider we are at the iteration when $k=i$ with $i\in[n]$. If $\frac{e_i}{i} \leq \gamma$, then we return $i$, otherwise, we need to look for $j< i$ such that $\frac{e_j}{j} \leq \gamma$.  Note that, for $j<i$, if $\frac{d_i}{1-\gamma} < j$, we have
\begin{equation}
    \frac{d_j}{j} \stackrel{(*)}{\leq} \frac{d_i}{j} < 1-\gamma, \quad \textnormal{and in turn }  \frac{e_j}{j} > \gamma,
\end{equation}
where (*) is true from Lemma 3 in \citet{goeman2019only}.
Therefore we cannot have $\frac{e_j}{j} \leq \gamma$ for $j < i$ if $j > \frac{d_i}{1-\gamma}$, so we directly skip those iterations in batches, and that gives us Algorithm~\ref{shortcuts-gamma}. \\

\noindent
Next, we prove part (b) and (c) which requires additional assumptions about the local tests. Part (b) is follows immediately from \thmref{shortcut}. Hence, we only prove part (c) of the theorem in the following. Without loss of generality, assume that all the $m$ scores are sorted as $T_1 \geq T_2 \geq \dots \geq T_m$. 

Firstly, we claim that the largest set $S\subseteq [m]$ with $\overline{e}_{\alpha}(S)=0$ admits strong FWER control at level $\alpha$. Note, from the definition of $\overline{e}_{\alpha}(S) $ in \eqref{fdpbound}, for any $S$ such that $\overline{e}_{\alpha}(S) =0$, each of its elementary subset is rejected by closed testing at level $\alpha$; and conversely, for any hypotheses set $S$ that is a collection of elementary hypotheses rejected by closed testing at level $\alpha$, each of its subsets is also rejected by closed testing, and hence $\overline{e}_{\alpha}(S)=0$. Therefore, the largest set $S\subseteq [m]$ with $\overline{e}_{\alpha}(S)=0$ is the collection of all the elementary hypotheses that are rejected by closed testing at level $\alpha$. Then recalling the well-known fact that the collection of all elementary hypotheses rejected by closed testing at level $\alpha$ is a hypothesis set with strong FWER control at level $\alpha$, we have proved our claim.

Then we show that finding the collection of all the elementary hypotheses rejected by closed testing at level $\alpha$ is equivalent to finding a cutoff in the ordered scores. From the monotonicity of the local test, it is easy to see that, for any $k \in [m]$, if closed testing rejects $T_k$, it must also reject $T_i$ for all $i > k$. Therefore, the final rejection sets must be of the form $\{T_{k^{\star}},\dots, T_m\}$, where $k^{\star}$ is a cutoff we are interested in finding in the ordered scores.
		
Finally, we show that Algorithm~\ref{shortcuts} is constructed in a way to find the correct cutoff, which is realized via searching from the largest score and stopped at the first one (which is our cutoff $k^{\star}$) rejected by closed testing. Note that we reject $H_k$ via closed testing if and only if each composite hypothesis containing it can be rejected locally. Using the monotonicity of the local test, this is saying that, for each $s = 1, \dots, m$, we have:
\begin{align}\label{conditionfwer}
	\begin{cases}
		\sum_{i = 1}^{s} h(T_{i})\leq  g(s,\alpha),\quad \text{if}\ s \geq k;\\
		h(T_k) + \sum_{i = 1}^{s-1} h(T_i) \leq  g(s, \alpha),\quad \text{otherwise}.
		\end{cases}
	\end{align}
With simple rearrangement, one may see that Algorithm~\ref{shortcuts} starts with $k = 1$, increase $k$ by 1 in each of its updates with $k$, and stops at the first time that \eqref{conditionfwer} is satisfied, when $k$ is the cutoff $k^{\star}$ of our interests. Therefore, we have finished the proof for part (c).

\section{Proof for \propref{worstcorr}}\label{app:worstcorrpf}
We call $X_{mi}$ as just $X_i$ in this proof for brevity. Note that, for $r\geq 1$,
\begin{equation}
    \widetilde{\alpha}_m(\Sigma,r,c):= \textnormal{Pr}_{\cap_{i=1}^m H_{mi}}\left\{\left(\frac1m \sum_{i = 1}^m p_{mi}^r \right)^\frac1r \leq c\right\} = \PP{\frac1m \sum_{i = 1}^m h_r(X_i) \leq C},
\end{equation}
where $h_r(x):= \Phi(-x)^r$, and $C := c^r$. Note that $h_r$ is a convex function for $x\geq 0$ when $r\geq1$. Indeed, taking second derivative of $h_r$ with respect to $x$, we have
\[
\frac{d^2 h_r(x)}{d x^2} = r\Phi(-x)^{r-2}\phi^2(x)\left[ x \Phi(-x) + r - 1 \right] .
\]
% shares the same sign with $x \Phi(-x)+r-1$. Since $x \Phi(-x)$ is positive for $x>0$, and is upper bounded by a constant $t^{\star}$ ($\approx 0.1699712$), we infer that for $r \geq 1 $, $h_r(x)$ is a convex function for $x \geq 0$, as claimed above. 
% Hence
% \begin{equation}\label{conv}
% \PP{\frac1m \sum_{i = 1}^m h_r(X_i) \leq C} \leq \PP{ h_r(\frac1m\sum_{i = 1}^m X_i) \leq C},\quad \textnormal{if}\ X_i >0\ \forall\ i\in[m].
% \end{equation}
When $C < \frac{1}{2^r m^r}$, the event $E_1:=\{ \frac1m \sum_{i = 1}^m h_r(X_i) \leq C \}$ implies\footnote{To see this, note that
% \begin{equation}
    $\frac1m \sum_{i = 1}^m h_r(X_i) \leq C$ implies that $\max_{i\in [m]} h_r(X_i) \leq m C$, which  happens if and only if $\min_{i\in[m]} X_i \geq -\Phi^{-1}(m C^{\frac1r}) > 0$,
% \end{equation}
where the last inequality is because $C < \frac{1}{2^r m^r}$.}
the event $E_2:=\{X_i>0,\ \forall i\in[m]\}$. Then, $E_1$ and $E_2$ together imply the event $E_3:=\{ h_r(\frac1m \sum_{i = 1}^m X_i) \leq C \}$ due to convexity of $h_r(x)$ for 
$x \geq 0$ and $r \geq 1$. 
Therefore,
% we have that event $\{X_i >0, \forall i\in[m]\}$ is implied by event $\{\frac1m \sum_{i = 1}^m h_r(X_i) \leq C\}$, given that $C < \frac{1}{2^r m^r}$, which in turn gives us the following due to \eqref{conv}, 
\begin{align}\label{rhoall}
\PP{\frac1m \sum_{i = 1}^m p_i^r \leq  C} &= \PP{\frac1m \sum_{i = 1}^m h_r(X_i)  \leq  C} \leq \PP{ h_r(\frac1m \sum_{i = 1}^m X_i) \leq C}  \nonumber \\ & = \PP{ \Phi(-\frac1m \sum_{i = 1}^m X_i)^r \leq C}
= \PP{\frac1m \sum_{i = 1}^m X_i \geq -\Phi^{-1}(C^{\frac1r}) } \nonumber \\ &= \PP{\frac1m \sum_{i = 1}^m X_i \geq C_2} 
= \text{Pr}_{\mathbf{X}\sim N(0, \Sigma)}\left\{\frac1m \mathbb{I}_m^T X \geq  C_2\right\} \nonumber\\ & \stackrel{(*)}{=} \text{Pr}_{Z\sim N(0, \sigma^2_{\Sigma})}\left\{Z  \geq C_2\right\} = 1-\Phi(C_2/\sigma_{\Sigma}),
\end{align}
where $C_2 = -\Phi^{-1}(C^{\frac1r})>0$, and $\sigma^2_{\Sigma} = \frac{1}{m^2}\mathbb{I}_m^{T}\Sigma\mathbb{I}_m \in \mathbb{R}$, where $\mathbb{I}_m$ is vector of all ones in $\mathbb{R}^m$.  Particularly, (*) is true due to the fact that Gaussianity is preserved under affine transformations. 

On the other hand, under fully dependence (i.e. $\rho_{ij} \equiv 1$ for all $i,j$), we have 
\begin{align}\label{rho1}
\PP{\frac1m \sum_{i = 1}^m p_i^r  \leq  C} &= \PP{ p_1^r  \leq   C} = \PP{ \Phi(-X_1)^r  \leq   C} \nonumber \\ &= \PP{X_1  \geq   C_2} = 1-\Phi( C_2).
\end{align}
Therefore combining \eqref{rhoall} and \eqref{rho1}, and the fact that $C_2>0$, we have 
\begin{align}
1-\Phi(C_2)&\leq \sup_{\Sigma \in \mathcal{M}_m} \PP{\frac1m \sum_{i = 1}^m p_i^r  \leq  C} \nonumber\\&\leq \sup_{\Sigma \in \mathcal{M}_m} 1-\Phi(C_2/\sigma_{\Sigma}) \\
&\stackrel{(a)}{=} \sup_{\Sigma \in \mathcal{M}_m^{E}} 1-\Phi(C_2/\sigma_{\Sigma}) \nonumber\\&\stackrel{(b)}{=} \sup_{\rho \in [-\frac{1}{m},1]} 1-\Phi\left(\frac{C_2}{\frac1m+\frac{m-1}{m}\rho}\right) = 1-\Phi(C_2),
\end{align}
where $\mathcal{M}_m$ is the class of all correlation matrix, and $\mathcal{M}_m^{E}$ is the class of all equicorrelation matrices with correlation $\rho\in[-\frac{1}{m},1]$. Specifically, (a) is true since $\sigma_\Sigma$ only depends on the average of all entries in $\Sigma$, and (b) is true since  $\sigma_\Sigma = \frac{1}{m} + \frac{m-1}{m}\rho$ for any $\Sigma$ in $\mathcal{M}_m^{E}$. In conclusion, we have
\begin{equation}
\sup_{\Sigma \in \mathcal{M}_m} \PP{\frac1m \sum_{i = 1}^m p_i^r  \leq  C} = 1-\Phi(C_2), 
\end{equation}
for all $r\geq 1$ and the supremum is achieved at full dependence. Transforming back to the original representation in \eqref{Hrhom}, we have completed our proof.

\section{Proof for \thmref{asymh}}\label{app:asymhpf}
Recalling decomposition \eqref{decomposition} in the main paper,  we can rewrite $\widetilde{\alpha}_m(\rho, r, c)$ as the following, which makes the link to the Generalized Law of Large Numbers clearer:
%\com{watch for punctuation, here and lots of other equations on this page itself.}
\begin{align}
\widetilde{\alpha}_m(\rho, r, c) &= \mathbb{E}_{Z_0}\left[\PPst{\left(\frac1m \sum_{i=1}^m p_i^r\right)^\frac1r \leq c)}{ Z_0= z_0}\right] \nonumber\\
&=\mathbb{E}_{Z_0}\left[\PPst{\textnormal{sign}(r)\frac1m \sum_{i=1}^m p_i^r \leq \textnormal{sign}(r)\cdot C)}{ Z_0= z_0}\right],
\end{align}
where we use the conditional independence amongst $\{p_i\}_{i=1}^m$, and replace $C:=c^r$. Then we have
\begin{align*}
\limsup_{m\to \infty}\widetilde{\alpha}_m(\rho, r, c) &= \limsup_{m\to \infty}\mathbb{E}_{Z_0}\left[\PPst{\textnormal{sign}(r)\frac1m \sum_{i=1}^m p_i^r \leq \textnormal{sign}(r)\cdot C}{Z_0= z_0}\right]\\ 
&= \mathbb{E}_{Z_0}\left[\limsup_{m\to \infty}\PPst{\textnormal{sign}(r)\frac1m \sum_{i=1}^m p_i^r \leq \textnormal{sign}(r)\cdot C}{ Z_0 = z_0}\right],
\end{align*}
where the last equality is true by applying dominance convergence theorem since the inner probability is integrable. 
% which is integrable with regard normal density.
In the following, we focus on quantifying the limitation of the conditional probability 
\begin{equation}\label{condprob}
\PPst{\textnormal{sign}(r)\frac1m \sum_{i=1}^m p_i^r \leq \textnormal{sign}(r)\cdot C}{ Z_0=z_0}, 
\end{equation} for which we need \lemref{cdf} to characterizes the distribution of $p_i^r$ for $r \neq \pm \infty$.
\begin{lemma}\label{lem:cdf}
	Denote the CDF of $p_i^r$ (with $p_i$ defined in \eqref{decomp2}) conditioning on $Z_0 = z_0$ as $F_{r, \rho, z_0}$, and the corresponding density as $f_{r, \rho, z_0}$, we have that:
	\begin{equation}\label{Fy}
	F_{r, \rho, z_0}(y) = \Phi\left(\textnormal{sign}(r)\frac{\Phi^{-1}(y^\frac1r)+\sqrt{\rho}z_0}{\sqrt{1-\rho}}\right), \quad\text{and}\quad   f_{r, \rho, z_0}(y) = O(y^{-(\frac{1}{(\rho-1)r} +1)})  \quad \textnormal{as}\ y^r \to 0,
	\end{equation}
where we take $y^{-(\frac{1}{(\rho-1)r} +1)} = \exp(-\frac{y}{\rho-1})$ when $r=0$.
\end{lemma}
\begin{proof}
	Without loss of generality, we only prove for the case $r\geq 0$. Firstly, when $r>0$, we have:
	\begin{align}\label{Fy}
	F_{r, \rho, z_0}(y) &= \PPst{p_i^r \leq y}{ Z_0 = z_0} = \PP{\Phi(-\sqrt{\rho}z_0-\sqrt{1-\rho} Z_i) \leq y^\frac1r} \nonumber\\
	&= \PP{ Z_i \leq \frac{\Phi^{-1}(y^\frac1r)+\sqrt{\rho}z_0}{\sqrt{1-\rho}}} = \Phi\left(\frac{\Phi^{-1}(y^\frac1r)+\sqrt{\rho}z_0}{\sqrt{1-\rho}}\right),
	\end{align}
	and also the density
	\begin{align}
	f_{r, \rho, z_0}(y) &= \frac{d F_{r, \rho, z_0}(y)}{dy} \propto \frac{y^{\frac1r-1}}{r\sqrt{1-\rho}}\phi\left(\frac{\Phi^{-1}(y^\frac1r)+\sqrt{\rho}Z_0}{\sqrt{1-\rho}}\right) \Big / \phi(\Phi^{-1}(y^\frac1r)) \\
	& \propto  \frac{y^{\frac1r-1}}{r\sqrt{1-\rho}}\exp\left(-\frac{\rho \Phi^{-1}(y^\frac1r)^2 + \sqrt{\rho}Z_0\Phi^{-1}(y^\frac1r)}{2(1-\rho)}\right).
	\end{align}
	Using the approximation 
	\[
	\Phi^{-1}(x) =O\left(-\sqrt{\log{\frac{1}{x^2}}}\right)\quad \text{when}\ \  x\to 0,
	\]
	we have that, 
	\begin{align}\label{frpos}
	f_{r, \rho, z_0}(y) =  O(y^{-(\frac{1}{(\rho-1)r} +1)}) \quad \textnormal{as}\ y \to 0.
	\end{align}
	
	\noindent
	For $r=0$, we have:
	\begin{align}\label{Fy}
	F_{r, \rho, z_0}(y) &= \PPst{\log{p_i} \leq y}{ Z_0 = z_0} = \PP{\Phi(-\sqrt{\rho}z_0-\sqrt{1-\rho} Z_i) \leq \exp{(y)} } \nonumber\\
	&= \PP{ Z_i \leq \frac{\Phi^{-1}(\exp{(y)})+\sqrt{\rho}z_0}{\sqrt{1-\rho}}} = \Phi\left(\frac{\Phi^{-1}(\exp{(y)})+\sqrt{\rho}z_0}{\sqrt{1-\rho}}\right),
	\end{align}
	and also the density
	\begin{align}
	f_{r, \rho, z_0}(y) &= \frac{d F_{r, \rho, z_0}(y)}{dy} \propto \frac{\exp{(y)}}{r\sqrt{1-\rho}}\phi\left(\frac{\Phi^{-1}(\exp{(y)})+\sqrt{\rho}Z_0}{\sqrt{1-\rho}}\right) \Big / \phi(\Phi^{-1}(\exp{(y)})) \\
	& \propto  \frac{\exp{(y)}}{r\sqrt{1-\rho}}\exp\left(-\frac{\rho \Phi^{-1}(\exp{(y)})^2 + \sqrt{\rho}Z_0\Phi^{-1}(\exp{(y)})}{2(1-\rho)}\right).
	\end{align}
	Again using the approximation 
	\[
	\Phi^{-1}(x) =O\left(-\sqrt{\log{\frac{1}{x^2}}}\right)\quad \text{when}\ \  x\to 0,
	\]
	we have that, 
	\begin{align}\label{frpos}
	f_{r, \rho, z_0}(y) =  O(\exp{(\frac{y}{1-\rho})}) \quad \textnormal{as}\ y \to -\infty, \textnormal{i.e. } \log{y} \to 0.
	\end{align}\\
\end{proof}

\paragraph{(a) \& (b) $\bm{r > -1}$:}
When $r>-1$, using \lemref{cdf} we have that $\EEst{p_1^r}{Z_0= z_0}<\infty$ for any $\rho \in [0,1]$, therefore by the Law of Large Numbers, we have
\begin{equation}
\frac1m \sum_{i = 1}^m p_i^r\  \big| \ Z_0= z_0 \ \stackrel{d}{\longrightarrow}\ \EEst{p_1^r}{Z_0 = z_0},
\end{equation}
where $\stackrel{d}{\to}$ means converge in distribution. Therefore,
\begin{align}
&\limsup_{m\to \infty}\PPst{\textnormal{sign}(r)\frac1m \sum_{i=1}^m p_i^r \leq \textnormal{sign}(r)\cdot C}{ Z_0 = z_0} \nonumber\\& = \PP{ \textnormal{sign}(r)\EEst{p_1^r}{ Z_0=z_0} \leq \textnormal{sign}(r)\cdot C},
\end{align}
and hence
\begin{align}\label{hlim}
 \limsup_{m\to\infty}\widetilde{\alpha}_m(\rho, r, c)& = \mathbb{E}_{Z_0}[\PP{ \textnormal{sign}(r)\EEst{p_1^r}{ Z_0 = z_0} \leq \textnormal{sign}(r) \cdot C}]:= h(\rho, r, C). 
\end{align}
Recall that the conditional mean $g_{\rho,r}(z_0):=  \EEst{p_i^r}{Z_0 = z_0}$ in \eqref{murho1}, we have
\begin{align}\label{murho}
g_{\rho, r}(z_0) &=  \int \Phi(-\sqrt{\rho}\ z_0 - \sqrt{1-\rho}\ x)^r \phi(x) dx \nonumber \\&= \frac{1}{\sqrt{1-\rho}}\int \phi(\frac{y-\sqrt{\rho}z_0}{\sqrt{1-\rho}}) \Phi(-y)^r dy,
\end{align}
 where $\phi$ as the standard normal p.d.f. From expression in \eqref{murho}, it is easy to see that $g_{\rho, r}(z_0)$ is monotonically non-increasing in $z_0$ when $r\geq 0$, while monotonically non-decreasing in $z_0$ when $r< 0$. Therefore, using this monotonicity, we have explicit expression
\begin{equation}\label{h}
    h(\rho, r, C) = \Phi\left(-g_{\rho, r}^{-1}(C)\right).
\end{equation}
Recall the relationship $C\equiv c^r$, and the definition of $c_r(m,\alpha)$ that
\[
c_r(m,\alpha) := \sup\{c: \sup_{\rho \in [0,1]}\limsup_{m\to\infty} \widetilde{\alpha}_m(\rho, \alpha, c) \leq \alpha\},
\]
where the supremum over $c$ is taking over $\mathbb{R}$, we omit it for simplicity.

For $r > 0$, plugging in expression \eqref{h} in \eqref{hlim}, we have the following closed expression
{\color{black}
\begin{align*}
  c_r(m,\alpha) &= \left(\sup\{C: \sup_{\rho \in [0,1]}\Phi(-g_{\rho, r}^{-1}(C)) \leq \alpha\}\right)^{\frac1r}.
\end{align*}
Denote $C_\rho:= \sup\{C: \Phi(-g_{\rho, r}^{-1}(C)) \leq \alpha\}$, we claim that $c_r(m,\alpha)$ is equivalent to $\left(\inf_{\rho\in[0,1]} C_{\rho}\right)^{\frac1r}$. To prove this claim, first note that
\begin{equation}\label{eq1}
    \sup\{C: \sup_{\rho \in [0,1]}\Phi(-g_{\rho, r}^{-1}(C)) \leq \alpha\} = \sup \bigcap_{\rho\in [0,1]}\{C: \Phi(-g_{\rho, r}^{-1}(C)) \leq \alpha\}.
\end{equation}
This is true due to the following simple  reasoning. For each $c$ such that $\sup_{\rho \in [0,1]}\Phi(-g_{\rho, r}^{-1}(c)) \leq \alpha$, we have
\begin{equation*}
    c \in \bigcap_{\rho\in [0,1]}\{C: \Phi(-g_{\rho, r}^{-1}(C)) \leq \alpha\}.
\end{equation*}
On the other hand, for each $c \in \bigcap_{\rho\in [0,1]}\{C: \Phi(-g_{\rho, r}^{-1}(C)) \leq \alpha\}$, we have
\begin{equation*}
    \sup_{\rho \in [0,1]}\Phi(-g_{\rho, r}^{-1}(c)) \leq \alpha.
\end{equation*}
Therefore
\begin{equation*}
    \{C: \sup_{\rho \in [0,1]}\Phi(-g_{\rho, r}^{-1}(C)) \leq \alpha\} = \bigcap_{\rho\in [0,1]}\{C: \Phi(-g_{\rho, r}^{-1}(C)) \leq \alpha\},
\end{equation*}
and taking supremum on both sides we have \eqref{eq1}.

Then we show that
\begin{equation}\label{eq2}
    \sup \bigcap_{\rho\in [0,1]}\{C: \Phi(-g_{\rho, r}^{-1}(C)) \leq \alpha\} = \inf_{\rho\in [0,1]}\{\sup \{C: \Phi(-g_{\rho, r}^{-1}(C)) \leq \alpha\}\}.
\end{equation}
For clarity, denote the LHS and RHS of \eqref{eq2} as $C_L$ and $C_R$ respectively. It is easy to see that
\begin{equation*}
    C_R \in \bigcap_{\rho\in [0,1]}\{C: \Phi(-g_{\rho, r}^{-1}(C)) \leq \alpha\} ,
\end{equation*}
and hence $C_R \leq C_L$. Also note that, for each $\rho \in [0,1]$,
\begin{equation*}
    \sup\{C: \Phi(-g_{\rho, r}^{-1}(C)) \leq \alpha\}  =  \sup\{C: C \leq g_{\rho, r}\left(- \Phi^{-1}(\alpha)\right)\} = g_{\rho, r}\left(- \Phi^{-1}(\alpha)\right) < \infty.
\end{equation*}
Therefore, there must exist one $\rho^\star$ such that 
\begin{equation*}
C_R = \sup\{C: \Phi(-g_{\rho^\star, r}^{-1}(C)) \leq \alpha\} \geq C_L.
\end{equation*}
Therefore, $C_L=C_R$. Combining with \eqref{eq1}, finally we conclude
\begin{equation}
    c_{r}(m,\alpha) = \left(\inf_{\rho\in[0,1]} C_{\rho}\right)^{\frac1r},
\end{equation} 
where $C_\rho:= \sup\{C: \Phi(-g_{\rho, r}^{-1}(C)) \leq \alpha\}$.
}

Using the monotonicity of $g_{\rho,r}(x)$ with regard $x$, and the fact that $g_{\rho, r}(x)$ decreases with $\rho$ when $\rho>0, x < 0$ and $r\geq 0$ (easy to verify that the derivative with regard $\rho$ is always negative), we further have
\begin{equation}
    C_\rho = \sup\{C: C\leq g_{\rho, r}\left(-\Phi^{-1}(\alpha)\right)\} = g_{\rho, r}\left(-\Phi^{-1}(\alpha)\right),
\end{equation}
therefore $c_r(m,\alpha)$ can be simplified as
\begin{align}\label{Cr1}
 c_r(m,\alpha) = \left(\inf_{\rho \in [0,1]}g_{\rho, r}\left(-\Phi^{-1}(\alpha)\right)\right)^{\frac1r} = \min{\left\{ \alpha, \left(\frac{r}{r+1}\right)^{\frac1r}\right\}}.
\end{align}
Similarly, for $-1<r \leq 0$, as $c:=C^{\frac1r}$ decreases with $C$, we now have the closed expression
\begin{align*}
  c_r(m,\alpha) = \left(\inf \left\{C: \sup_{\rho \in [0,1]} \Phi(-g_{\rho, r}^{-1}(C)) \leq \alpha\right\}\right)^{\frac1r} = \left(\sup_{\rho \in [0,1]} C_\rho \right)^{\frac1r}, 
\end{align*}
where $C_\rho:= \inf\{C:  \Phi(-g_{\rho, r}^{-1}(C)) \leq \alpha\}$. Using the monotonicity of $g_{\rho,r}(x)$ with regard $x$, and the fact that $g_{\rho, r}(x)$ decreases with $\rho$ when $\rho>0, x < 0$ and $r< 0$ (easy to verify that the derivative with regard $\rho$ is always negative), we further have
\begin{equation}
    C_\rho = \inf\{C: C\geq g_{\rho, r}\left(-\Phi^{-1}(\alpha)\right)\} = g_{\rho, r}\left(-\Phi^{-1}(\alpha)\right).
\end{equation}
Therefore,
\begin{align}\label{Cr2}
c_r(m,\alpha) & = \left(\sup_{\rho \in [0,1]}{g_{\rho, r}\left(-\Phi^{-1}(\alpha)\right)}\right)^{\frac1r}.
\end{align}
\noindent
Finally, we have that,
\begin{align}
\widetilde{\alpha}(\rho, r, \alpha) &= \limsup_{m\to\infty}\widetilde{\alpha}_m(\rho, r, c_r(m,\alpha))\nonumber\\
&= \begin{cases}
\Phi(-g_{\rho, r}^{-1}(\alpha^r)), & \textnormal{if } r>0;\\
\Phi(-g_{\rho, r}^{-1}(c_r(m,\alpha)^r), &\textnormal{if } -1\leq r\leq 0.
\end{cases}
\end{align}
And 
\begin{align}\label{Cr}
    c_r(m,\alpha) = 
    \begin{cases}
    \min{\{ \alpha, \left(\frac{r}{r+1}\right)^{\frac1r}\}}, &\textnormal{if } r>0;\\
    \left(\sup_{\rho \in [0,1]}g_{\rho, r}\left(-\Phi^{-1}(\alpha)\right)\right)^{\frac1r}&\textnormal{if } -1\leq r\leq 0.
    \end{cases}
\end{align}

\paragraph{(c) \& (d): $\bm{r \leq -1}$.}

When $r \leq -1$, things get a bit tricky, since according to \lemref{cdf}, $\EEst{p_i^r}{Z_0=z_0}$ may not exist. In the following, we will use the stable law stated in \lemref{stablelaw} to derive the asymptotic behaviour of $\widetilde{\alpha}_m(\rho, r, c)$ for $r < 0$. 
\begin{lemma}(Generalized LLN \citep{uchaikin2011chance})\label{lem:stablelaw}
	Consider a sequence of i.i.d random variables $X_1, X_2, \dots, X_m$ which shares the same distribution with $X$, where $X$ has support on $[1,\infty]$ and density $f$ satisfying the following: 
	\[
	f(x) = O(x^{-(\beta+1)}),\quad \textnormal{as}\ x \to \infty\ \textnormal{with}\ \beta >0.
	\]
	Denote $\overline{X}_m := \frac 1m \sum_{i=1}^{m} X_i$, we have that
	\begin{itemize}
		\item[(a)] if $\ 0<\beta < 1$, then $\  m^{1-\frac1\beta}\overline{X}_m \stackrel{d}{\to} Y$;
		\item[(b)] if $\ \beta = 1$, then $\ \overline{X}_m - \log{m} \stackrel{d}{\to} Y$;
		\item[(c)] if $\ 1 < \beta < 2$, then $\  m^{1-\frac1\beta} (\overline{X}_m - \EE{X}) \stackrel{d}{\to} Y$;
		\item[(d)] if $\ \beta \geq 2$, then $\   \overline{X}_m \stackrel{d}{\to} \EE{X}$,
	\end{itemize}
	where $Y$ is some random variable that shares the same tail behaviour with $X$.
\end{lemma}

\noindent
Then, for $r \leq -1$, from \lemref{cdf}, we have $\beta = \frac{1}{(\rho-1)r}$. Let
\begin{align}
C(\alpha, r, m, \rho) := \begin{cases}
C_{\alpha, r} m^{-1+(\rho-1)r}, & \textnormal{if}\ 0 \leq \rho < 1 + \frac1r;\\
C_{\alpha, r} + \log{m} , & \textnormal{if}\ \rho = 1 + \frac1r;\\
C_{\alpha, r} m^{-1+(\rho-1)r} + \EEst{p_1^r}{Z_0= z_0}, & \textnormal{if}\ 1+\frac{1}{r} < \rho \leq 1+\frac{1}{2r};\\
C_{\alpha, r} + \EEst{p_1^r}{Z_0= z_0}, & \textnormal{if}\ 1+\frac{1}{2r} < \rho \leq 1,
\end{cases}
\end{align}
where $C_{\alpha, r}$ is some constant that depends only on $\alpha$ and $r$ that we will specify later. Using \lemref{stablelaw}, we have
\begin{align}
 &\quad \lim_{m\to\infty}\PPst{\frac1m \sum_{i=1}^m p_i^r \geq C(\alpha, r, m, \rho)}{ Z_0 = z_0}  \nonumber\\&= \PPst{Y \geq C_{\alpha, r}}{Z_0 = z_0} \stackrel{(*)}{=} \PPst{p_1^r \geq C_{\alpha, r}}{Z_0 = z_0} + o(1)\nonumber\\
& = F_{r,\rho,z_0}(C_{\alpha, r}) + o(1) \quad \textnormal{as }\ \alpha \to 0,
\end{align}
where $Y$ is the random variable comes from the limitation in \lemref{stablelaw}, which shares the same tail behaviour of $p_1^r$, therefore we have the approximation $(*)$. 

Recalling definitions in \eqref{Hrhom}, \eqref{alpha} of the main paper, our goal is to find $c$ such that
\[
\sup_{\rho\in[0,1]}\limsup_{m\to\infty}\widetilde{\alpha}_m(\rho, r, c) \leq \alpha,
\]
or equivalently find $C$ such that
\[\sup_{\rho\in[0,1]}\limsup_{m\to\infty}\widetilde{\alpha}_m(\rho, r, C^{\frac1r}) \leq \alpha.
\]
Note that $C(\alpha, r, m, \rho)$ is monotonically non-increasing in $\rho$, and $C(\alpha, r, m, 0)$ dominates $C(\alpha, r, m, \rho)$ for any \smash{$0<\rho\leq 1$}. Therefore, to calibrate for arbitrary $\rho \in [0,1]$, that is to find a critical value that does not depend on $\rho$, we have no choice but let $C = C(\alpha, r, m, 0)$, and hence 
\begin{align*}
   &\quad \quad  \sup_{\rho\in[0,1]}\limsup_{m\to\infty}\widetilde{\alpha}_m(\rho, r, C(\alpha, r, m, 0)^{\frac1r})  = \widetilde{\alpha}_m(\rho, r, C(\alpha, r, m, 0)^{\frac1r})\\
   & = \mathbb{E}_{Z_0}\left[1- F_{r,0, Z_0}(C_{\alpha, r})\right] = \mathbb{E}_{Z_0}\left[\Phi\left(\Phi^{-1}(C_{\alpha, r}^\frac1r)\right)\right] = C_{\alpha, r}^\frac1r \leq \alpha,
\end{align*}
which indicates we should set $C_{\alpha, r} = \alpha^r$ to achieve the upper bound. 

Therefore we have 
\begin{align}
c_r(m,\alpha) = \left(C(\alpha, r,m,0)\right)^\frac1r = 
\begin{cases}
\alpha m^{\frac{1}{|r|}-1}, \quad \textnormal{if}\ r <-1;\\
\frac{\alpha}{1 + \alpha m}, \quad \textnormal{if}\ r = -1,\\
\end{cases}
\end{align}
and correspondingly
\[
\widetilde{\alpha}(\rho, r, \alpha) = \limsup_{m\to\infty}\widetilde{\alpha}_m(\rho, r, c_r(m,\alpha)) = \alpha \mathbb{I}\{\rho=0\},
\]
where the last equality is true due to the nature of stable law, where the tail behavior determines the rate of growth, and the mismatch of the growth rate leads to degenerate asymptotic probability.

Here we finish the proof for \thmref{asymh}.

\section{Proof for \thmref{asymp1}}\label{app:asymppf1}

In the following, we are interested in calculating the asymptotic power using the calibrated threshold $c_r(m,\alpha)$ derived in \thmref{asymh}. In particular, the power can rewritten as
\begin{equation}
 \beta_{\mu_m,\pi_m,\rho}(r, \alpha) :=  \PP{\textnormal{sign}(r)\frac1m \sum_{i=1}^m p_{mi}^r \leq \textnormal{sign}(r) C_r(m,\alpha)},
\end{equation}
where $C_r(m,\alpha)= c_r(m,\alpha)^r$, and $p_{mi}=\Phi^{-1}(-X_{mi})$ for all $i$. Using similar decomposition as in the proof of \thmref{asymh}, we have that, for all $i = 1, 2, \dots, m$,
\begin{align}\label{decomp}
& X_{mi} = \mu_{mi} + \sqrt{\rho}\ Z_{0} + \sqrt{1-\rho}\ Z_{i}, \nonumber\\
& p_{mi} = \Phi(-X_{mi}) = \Phi\left(-\mu_{mi}-\sqrt{\rho}\ Z_{0} - \sqrt{1-\rho}\ Z_{i}\right),
\end{align}
where variable $Z_{0} \sim N(0,1)$ , $Z_{i} \stackrel{\text{iid}}{\sim} N(0,1)$, $\{Z_{0}\} \independent{\{\mu_{mi}, Z_i\}_{i=1}^m}$, and $\mu_{mi} \stackrel{iid}{\sim}\mu_m B_{mi}$ with \smash{$B_{mi}\stackrel{\text{iid}}{\sim} \textnormal{Bernoulli}(\pi_m)$} for all $i = 1, 2, \dots ,m$. Also, we have the conditional independence: 
\begin{equation}
    p_{m1}, p_{m2}, \dots, p_{mm} \textnormal{ are independent conditioning on } Z_{0}.
\end{equation}
Then the asymptotic power is given by
\begin{align}\label{G1}
\ \ \ \ \  \lim_{m\to\infty}\beta_{\mu_m,\pi_m,\rho}(r, \alpha) 
&= \mathbb{E}_{Z_0}\left[\lim_{m\to \infty}\PPst{\textnormal{sign}(r)\frac1m \sum_{i=1}^m p_{mi}^r \leq \textnormal{sign}(r)C_r(m,\alpha)}{Z_0}\right].
% & =  \mathbb{E}_{Z_0}\left[\lim_{m\to \infty}\pi_{m}\PPst{\textnormal{sign}(r)\frac1m \sum_{i=1}^m p_{mi}^r \leq \textnormal{sign}(r)C_r}{Z_0, \mu = \mu_{m}}\right]
% \nonumber\\
% & \ \ \ \ +  \mathbb{E}_{Z_0}\left[\lim_{m\to \infty}(1-\pi_{m})\PPst{\textnormal{sign}(r)\frac1m \sum_{i=1}^m p_{mi}^r \leq \textnormal{sign}(r)C_r}{Z_0, \mu = 0}\right]\nonumber\\
% &= \mathbb{E}_{Z_0}\left[\lim_{m\to \infty}\pi_{m}\PPst{\textnormal{sign}(r)\frac1m \sum_{i=1}^m p_{mi}^r \leq \textnormal{sign}(r)C_r}{Z_0, \mu = \mu_{m}}\right] + \lim_{m\to \infty}(1-\pi_{m}) \widetilde{\alpha}_m(\rho, r, \alpha)
\end{align}
When $r>0$, we can use the law of large numbers of triangular array, that is
\begin{equation}
    \sup_{m}\EEst{p_{mi}^{2r}}{Z_0=z_0}< \infty\quad \Rightarrow \quad {\color{black} \frac1m \sum_{i=1}^m p_{mi}^r  - \EEst{p_{mi}^r}{Z_0=z_0} \stackrel{p}{\to} 0}
\end{equation}
almost surely for all possible value of $z_0$.
% \com{I don't understand the notation on the right-hand side of Eq. (88). What is the conditioning supposed to mean? }
% {\color{red} whether this is a.s. on $Z_0$.}
Then we get
\begin{align}\label{G2}
& \ \ \ \ \lim_{m\to \infty}\PPst{\frac1m \sum_{i=1}^m p_{mi}^r \leq C_r(m,\alpha)}{Z_0}  = \lim_{m\to \infty}\PP{\EEst{p_{m1}^r}{Z_0} \leq  \alpha^r} \nonumber\\
& = \lim_{m\to \infty}\PP{\pi_{m}\EEst{p_{m1}^r}{Z_0, \mu_{m1} = \mu_m} + (1-\pi_{m})\EEst{p_{m1}^r}{Z_0, \mu_{m1} =0}\leq  \alpha^r}.
\end{align}
%\[ 
%\lim_{m\to \infty}\pi_{Am}\Phi\left(-\mu_{\rho, r}^{-1}(\textnormal{sign}(r)C_r) + \frac{\mu_{Am}}{\sqrt{\rho}}\right)  = \lim_{m\to \infty}\pi_{Am}  (1-\frac{\phi(\mu_{Am}/\sqrt{\rho}-C) }{ \mu_{Am}/\sqrt{\rho}-C}),
%\]
Combining \eqref{G1} and \eqref{G2}, we have that, when $r>0$, 
\begin{align}
\lim_{m\to\infty}\beta_{\mu_m,\pi_m,\rho}(r, \alpha) &= \PP{ \pi g_{\rho, r} (Z_0 + \frac{\mu}{\sqrt{\rho}}) + (1-\pi) g_{\rho, r} (Z_0) \leq \alpha^r}
\end{align}
where $g_{\rho, r}$ is defined in \eqref{murho1}. From this expression, the following cases can be specified,
\begin{itemize}
    \item if $\pi = 1$, then \begin{align}\lim_{m\to\infty}\beta_{\mu_m,\pi_m,\rho}(r, \alpha) = \begin{cases}
    1, & \quad \text{if}\ \mu = \infty;\\
    \Phi\left(-g_{\rho,r}^{-1}(\alpha^r)+\frac{\mu}{\sqrt{\rho}}\right), & \quad \text{if}\ 0< \mu <\infty;\\
    \widetilde{\alpha}(\rho, r, \alpha), & \quad \text{if}\ \mu =0.
    \end{cases}
    \end{align}
    \item if $0< \pi< 1$, then \begin{align}\lim_{m\to\infty}\beta_{\mu_m,\pi_m,\rho}(r, \alpha) = \begin{cases}
    \Phi\left(-g_{\rho,r}^{-1}(\frac{\alpha^r}{1-\pi})\right), & \quad \text{if}\ \mu = \infty;\\
    \PP{ \pi g_{\rho, r} (Z_0 + \frac{\mu}{\sqrt{\rho}}) + (1-\pi) g_{\rho, r} (Z_0) \leq \alpha^r}, & \quad \text{if}\ 0< \mu <\infty;\\
    \widetilde{\alpha}(\rho, r, \alpha), & \quad \text{if}\ \mu =0.
    \end{cases}
    \end{align}
    \item if $\pi = 0$, then  $\lim_{m\to\infty}\beta_{\mu_m,\pi_m,\rho}(r, \alpha) \equiv \widetilde{\alpha}(\rho, r, \alpha)$, no matter what value that $\mu$ takes. 
    \end{itemize}
Therefore, we complete the proof.

\section{Proof for \thmref{r-1weak}}\label{app:r-1weakpf}

When $ r \leq-1 $, We utilize the following results: as long as the triangular array $\{Y_{mi}, i = 1, \dots, i_m\}$ satisfy the uniformly asymptotically negligible (UAN) condition, that is for any $\epsilon >0$,
\begin{equation}\label{UAN}
\lim_{m\to\infty} \max_{i}\PP{|Y_{mi}|>\epsilon} = 0,    
\end{equation}
then we have that, $\lim_{m\to\infty} \sum_{i} Y_{mi}$ converge to an infinitely divisible distribution under certain conditions. The specific argument is formally stated in the following \lemref{infdivi}.

\begin{lemma}\label{lem:infdivi} (Theorem 3.2.2 in \citep{key0062975m})
Consider an triangular array $\{Y_{mk}, k = 1, \dots, k_m\}$, such that the UAN condition is fulfilled,
that is for any $\epsilon >0$
\begin{equation}\label{uan}
    \lim_{m\to\infty} \max_{k}\mu_{mk}\{|y|>\epsilon\} = 0,
\end{equation}
where $\mu_{mk}$ is the distribution function for $Y_{mk}$, and denote $S_m := Y_{m1} + \dots + Y_{m,k_m}$. 

Then there exists a deterministic sequence $a_m$ such that sequence $S_m - a_m$ converges weakly to an infinitely divisible random variable $Y$ if and only if the following conditions are fulfilled:

% \begin{enumerate}
    % \item If there exists a random variable $Y$ and a sequence of real numbers $a_m$ such that 
    % \begin{equation}
    %      S_m - a_m \stackrel{d}{\to} Y, \quad, m\to \infty,
    % \end{equation}
    % then $Y$ has an infinite divisible distribution; moreover, for any infinitely distribution $P_{\inf}$, there exists a triangular array $\{Z_{mk}, k= 1, \dots, k_m\}$ such that $S_m \stackrel{d}{\to} P_{\inf}$. 

    % \item 

    \begin{enumerate}
        \item for any $A = (-\infty, x)$ with $x<0$, and $A = (x, \infty)$ with $x>0$ such that $\nu(\partial A) = 0$,
        \begin{equation}\label{sigma}
            \nu(A) := \lim_{m\to \infty}\sum_{k = 1}^{k_m} \mu_{mk}(A)
        \end{equation}
        is a L\'evy measure, i.e. a $\sigma$-finite Borel measure on $\mathbb{R}\setminus{0}$ such that \smash{$\int_{\mathbb{R}\setminus{0}} \min\{1, x^2\}\nu(d x) < \infty$}.
        
        \item moreover, 
        \begin{align}\label{sigma}
            & \lim_{\tau \to 0} \limsup_{m\to \infty}\sum_{k=1}^{k_m} \textnormal{Var}\left(Z_{mk}\ones\{|Z_{mk}|<\tau\}\right)\nonumber\\
            = & \lim_{\tau \to 0} \liminf_{m\to \infty}\sum_{k=1}^{k_m}\textnormal{Var}\left(Z_{mk}\ones\{|Z_{mk}|<\tau\}\right) = \sigma^2 < \infty.
        \end{align}
        
    \end{enumerate}
    Particularly, Y has the characteristic exponent
        \begin{equation}
        \phi(t) = - \frac12 \sigma^2 t^2 + \int_{\mathbb{R}\setminus\{0\}} (e^{itx} - 1 - itx\ones\{|x|\leq 1\}) \nu (dx),  
    \end{equation}
    and $a_m$ can be choosen by
    \begin{equation}\label{an}
        a_m = \sum_{k=1}^{k_m} \int _{|x|< 1} x \mu_{nk}(dx) + o(1)
        \end{equation}
    given that $\nu(\{x:|x|=1\}) = 0$.    \\
\end{lemma}

\noindent
In our case, let 
\[
Y_{mi} = \frac{1}{m^{-r}} (p_{mi}^r - a_{r,m}) | Z_0,
\] 
where $a_{r,m} = 0$ if $r < -1$, and $a_{r,m} = \log{m}$ if $r=-1$. We firstly check the UAN condition \eqref{UAN}. Note that 
\begin{align}\label{uan}
& \ \ \ \ \ \ \ \lim_{m\to\infty} \max_{i}\PPst{|\frac{1}{m^{-r}} (p_{mi}^r - a_{r,m})|>\epsilon}{Z_0 = z_0} = \lim_{m\to\infty}\PPst{|\frac{1}{m^{-r}}(p_{mi}^r - a_{r,m})|>\epsilon}{Z_0} \nonumber\\
&= \lim_{m\to\infty}\PPst{ p_{mi}^r > m^{-r}\epsilon + a_{r,m}}{Z_0} + \PPst{p_{mi}^r <-m^{-r}\epsilon +a_{r,m}}{Z_0} \nonumber\\
&= \lim_{m\to\infty}\pi_m\PPst{ p_{mi}^r > m^{-r}\epsilon + a_{r,m}}{Z_0,\mu_{mi} = \mu_m}+ (1-\pi_m)\PPst{ p_{mi}^r > m^{-r}\epsilon + a_{r,m}}{Z_0,\mu_{mi} = 0} \nonumber\\ & \ \ \ \ \ \ \ \ \ + \pi_m\PPst{p_{mi}^r <-m^{-r}\epsilon +a_{r,m}}{Z_0, \mu_{mi} = \mu_m} + (1-\pi_m)\PPst{p_{mi}^r <-m^{-r}\epsilon +a_{r,m}}{Z_0, \mu_{mi} = 0} \nonumber\\
& = \lim_{m\to\infty}\pi_m\Phi\left(\frac{\Phi^{-1}((m^{-r}\epsilon + a_{r,m})^\frac1r) + \mu_m+\sqrt{\rho}z_0}{\sqrt{1-\rho}}\right)+ (1-\pi_m)\Phi\left(\frac{\Phi^{-1}((m^{-r}\epsilon + a_{r,m})^\frac1r) +\sqrt{\rho}z_0}{\sqrt{1-\rho}}\right) \nonumber\\& \ \ \ \ \ \  + \pi_m\Phi\left(-\frac{\Phi^{-1}((-m^{-r}\epsilon +a_{r,m})^\frac1r) + \mu_m+\sqrt{\rho}z_0}{\sqrt{1-\rho}}\right) + (1-\pi_m)\Phi\left(-\frac{\Phi^{-1}((-m^{-r}\epsilon +a_{r,m})^\frac1r) + \sqrt{\rho}z_0}{\sqrt{1-\rho}}\right).
%&= \lim_{m\to\infty}\PP{p_{mi} <(m\epsilon + a_{r,m})^{1/r}} + \PP{p_{mi} >(-m\epsilon +a_{r,m})^{1/r}} \nonumber\\
%&= \lim_{m\to\infty}\PP{\Phi(-X_{mi} -\mu_{m}) <(m\epsilon +a_{r,m})^{1/r}} + \PP{\Phi(-X_{mi} - \mu_{m}) >(-m\epsilon + \mu_{r,m})^{1/r}} \nonumber\\
%&= \lim_{m\to\infty}\PP{X_{mi} < \Phi^{-1}((m\epsilon +a_{r,m})^{1/r}) + \mu_{m} } + \PP{ X_{mi}  < - \Phi^{-1}((-m\epsilon + a_{r,m})^{1/r}) - \mu_{m}} \nonumber\\
%& = \Phi(\Phi^{-1}(m\epsilon + a_{r,m})^{1/r} + \mu_m) + \Phi(-\Phi^{-1}(-m\epsilon + a_{r,m})^{1/r} - \mu_m)
\end{align}
For $r<-1$, we have $a_{r,m} = 0$, and thus,
\eqref{uan} can be simplified as
\begin{align}\label{uan-1}
   \lim_{m\to\infty}\pi_m\Phi\left(\frac{\Phi^{-1}((m^{-r}\epsilon)^\frac1r) + \mu_m+\sqrt{\rho}z_0}{\sqrt{1-\rho}}\right) = \lim_{m\to\infty}\pi_m\Phi\left(\frac{\Phi^{-1}(\epsilon^\frac1r\frac1m) + \mu_m+\sqrt{\rho}z_0}{\sqrt{1-\rho}}\right),
\end{align}
while on the other hand, for $r=-1$, we have $a_{r,m}=\log{m}$, and
\eqref{uan} can also be simplified as
\begin{align}\label{uan-2}
   \lim_{m\to\infty}\pi_m\Phi\left(\frac{\Phi^{-1}((m\epsilon + \log{m})^\frac1r) + \mu_m+\sqrt{\rho}z_0}{\sqrt{1-\rho}}\right) = \lim_{m\to\infty}\pi_m\Phi\left(\frac{\Phi^{-1}(\epsilon^\frac1r\frac1m) + \mu_m+\sqrt{\rho}z_0}{\sqrt{1-\rho}}\right).
\end{align}
Therefore, in order to make \eqref{uan-1} and \eqref{uan-2} goes to zero, we only need to make $\mu_m$ grows slower than \smash{$|\Phi^{-1}(\frac{1}{m})| = O(\sqrt{\log{m}})$}, that is $\mu_m = o(\sqrt{\log{m}})$. \\

\noindent
Returning to the proof of the theorem, we first consider the case $\rho>0$, under which we will prove that for each $i$, $Y_{mi} =o_p(1)$ when $r<-1$, and $Y_{mi} = o(\log m)$ when $r = -1$, as $m \to \infty$. We prove this by applying \lemref{infdivi}, during which we check the condition 1 and 2 in it. 

As for condition 1 in \lemref{infdivi} for $r\leq-1$, defining $\nu(x) := 1-\lim_{m\to\infty} m \PP{Y_{mi} > x}$ for all $x>0$, it can be simplified to checking that
\begin{align}\label{levy}
    1-\nu(1) + \int_{0<x<1}x^2\nu(dx) < \infty.
\end{align}
Note that 
\begin{align}
    \PP{Y_{mi} > x} & = \PPst{m^r P_{mi}^r > x }{Z_0 = z_0} = \PPst{P_{mi} < \frac{x^{\frac1r}}{m} }{Z_0 = z_0}\nonumber\\
    & = \pi_m\PPst{P_{mi} < \frac{x^{\frac1r}}{m} }{Z_0 = z_0, \mu_{mi} = \mu_m} + (1-\pi_m)\PPst{P_{mi} < \frac{x^{\frac1r}}{m} }{Z_0 = z_0, \mu_{mi} = 0} \nonumber\\
    & = \pi_m\Phi\left(\frac{\Phi^{-1}((x^{\frac1r}/m) + \mu_m +\sqrt{\rho}z_0}{\sqrt{1-\rho}}\right) + (1-\pi_m)\Phi\left(\frac{\Phi^{-1}((x^{\frac1r}/m) +\sqrt{\rho}z_0}{\sqrt{1-\rho}}\right) \nonumber\\
    % & = \pi_m \Phi\left(\frac{-\sqrt{2\log{(m x^{-\frac1r})}} +o(1) + \mu_m +\sqrt{\rho}z_0}{\sqrt{1-\rho}}\right) + \\ & (1-\pi_m) \Phi\left(\frac{-\sqrt{2\log{(m x^{-\frac1r})}} +o(1) +\sqrt{\rho}z_0}{\sqrt{1-\rho}}\right) \nonumber\\
    & = \pi_m \Phi\left(\frac{-\sqrt{2\log{(m x^{-\frac1r})}}  +\sqrt{\rho}z_0}{\sqrt{1-\rho}}\right) + (1-\pi_m) \Phi\left(\frac{-\sqrt{2\log{(m x^{-\frac1r})}} +\sqrt{\rho}z_0}{\sqrt{1-\rho}}\right) + o(1)\nonumber\\
    & = \Phi\left(\frac{-\sqrt{2\log{(m x^{-\frac1r})}}+\sqrt{\rho}z_0}{\sqrt{1-\rho}}\right) + o(1) = O(m^{-\frac{1}{1-\rho}}x^{\frac{1}{(1-\rho)r}}),
\end{align}
therefore, we have 
\begin{equation}\label{cond1-1}
    \nu(x) = 1-\lim_{m\to\infty} m\PP{Y_{mi}>x} = 1-x^{\frac{1}{(1-\rho)r}}\lim_{m\to\infty} m^{-\frac{\rho}{1-\rho}} = 0,
\end{equation}
since $\rho>0$.
% and when $\rho = 0$, we have that 
% \begin{equation}\label{cond1-2}
%     \int_{0<x<1}x^2 \nu(d x) = \int_{0<x<1}-\frac1r x^{\frac1r +1} d x = -\frac{1}{1+2r} < \infty.
% \end{equation}
Therefore \eqref{levy} is true, and in particular $\nu(x) = 0$ when $\rho > 0$.

Afterwards, we check condition 2 in \lemref{infdivi}, which simplifies to verifying
\begin{equation}
    \lim_{\tau\to 0} \lim_{m\to \infty} m \text{Var}(Y_{mi}\ones\{Y_{mi} < \tau\}) < \infty
\end{equation}
in our setting. Using the similar technique that we will use to calculate $a_m$, i.e. the truncated first moment, we have the following about the truncated second moment for any fixed truncation position $\tau>0$,
\begin{align}
    &  \text{Var}(Y_{mi}\ones\{Y_{mi} < \tau\}) \leq  m\EE{Y_{mi}^2\ones\{Y_{mi} < \tau\}} \nonumber\\
   & = o(m^{1-\frac{1}{1-\rho}}\log^{\frac{1-2r}{2}} (m)) \to 0, \quad \textnormal{as } m \to \infty,\   \textnormal{since } \rho > 0.
\end{align}
Therefore, the limit distribution does not have a normal term when $\rho>0$.\\

\noindent
Lastly, we compute $a_m$ via \eqref{an}, that is
\begin{align}\label{am_comp}
    a_m &= m \EE{Y_{mi}\ones\{Y_{mi}<1\}} = m^{r+1} \EE{P_{mi}^r\ones\{P_{mi}^r<\frac{1}{m^r}\}} \nonumber\\&
    =  -\frac{m^{r+1}}{r\sqrt{1-\rho}} \int_{1}^{m^{-r}} y^{\frac1r}\exp\left(-\frac{\rho \Phi^{-1}(y^\frac1r)^2 + 2 A_m\Phi^{-1}(y^\frac1r) + A_m^2}{2(1-\rho)}\right)dy,
\end{align}
where $A_m = \sqrt{\rho}z_0 + \mu_m = o(\sqrt{\log{m}})$. Let $x = \Phi^{-1}(y^{\frac1r})$, we have that \eqref{am_comp} equals
\begin{align}\label{am_comp_trans}
   & \quad \frac{m^{r+1}}{\sqrt{1-\rho}}\int_{\Phi^{-1}(\frac1m)}^{\infty} \Phi(x)^r \exp{\left(-\frac{x^2 + 2A_m x +A_m^2}{2(1-\rho)}\right)} d x  \nonumber\\ & = \frac{m^{r+1}}{\sqrt{1-\rho}}\left(\int_{\Phi^{-1}(\frac1m)}^1 + \int_{1}^{\infty}\right) \Phi(x)^r \exp{\left(-\frac{x^2 + 2A_m x + A_m^2}{2(1-\rho)}\right)}dx = \frac{m^{r+1}}{\sqrt{1-\rho}}(\mathbb{I}_1 + \mathbb{I}_2).
\end{align}
Using the following well-known Mill's inequality \citep{gordon1941values}, that is for any $u>0$,
\begin{equation}\label{mill}
    \frac{u}{1+u^2}\phi(u) \leq \Phi(-u) \leq \frac{1}{u}\phi(u),  
\end{equation}
we have that 
\begin{align}\label{i1-1}
    \mathbb{I}_1 & \leq  \int_{\Phi^{-1}(\frac1m)}^1 (\frac{-x}{1+x^2})^r \phi(-x)^r \exp{\left(-\frac{x^2 + 2A_m x + A_m^2}{2(1-\rho)}\right)}dx\nonumber\\
    & = \frac{1}{\sqrt{2\pi}}\int_{1}^{-\Phi^{-1}(\frac1m)} (\frac{x}{1+x^2})^r \exp{\left(-\frac{[r(1-\rho)+1] x^2 - 2A_m x + A_m^2}{2(1-\rho)}\right)}dx\nonumber\\
    & = \frac{1}{\sqrt{2\pi}}\int_{1}^{-\Phi^{-1}(\frac1m)} (\frac1x+x)^{-r} \exp{\left(-\frac{[r(1-\rho)+1] x^2 - 2A_m x + A_m^2}{2(1-\rho)}\right)}dx\nonumber\\
    & \leq \frac{2^{-r}}{\sqrt{2\pi}}\int_{1}^{-\Phi^{-1}(\frac1m)} x^{-r} \exp{\left(-\frac{[r(1-\rho)+1] x^2 - 2A_m x + A_m^2}{2(1-\rho)}\right)}dx \nonumber\\ & = \frac{2^s}{\sqrt{2\pi}\exp(c_m)}\int_{1}^{-\Phi^{-1}(\frac1m)} x^s \exp\left(\frac{a}{2} x^2+b_m x\right) dx,
\end{align}
and
\begin{align}\label{i1-2}
    \mathbb{I}_1 & \geq  \int_{\Phi^{-1}(\frac1m)}^1 (\frac{1}{-x})^r \phi(-x)^r \exp{\left(-\frac{x^2 + 2A_m x + A_m^2}{2(1-\rho)}\right)}dx\nonumber\\
    &= \frac{1}{\sqrt{2\pi}\exp(c_m)}\int_{1}^{-\Phi^{-1}(\frac1m)} x^s \exp\left(\frac{a}{2} x^2+b_m x\right) dx,
\end{align}
where $s = -r\geq 1$; $a = \frac{r(\rho-1)-1}{1-\rho} = s-\frac{1}{1-\rho}$; $b_m = \frac{A_m}{1-\rho}>0$; $c_m = \frac{A_m^2}{2(1-\rho)}$. \\

\noindent
Combining \eqref{i1-1} and \eqref{i1-2}, we have that
\begin{align}\label{i1}
    \mathbb{I}_1 = O\left(\exp(c_m)\int_{1}^{-\Phi^{-1}(\frac1m)} x^s \exp\left(\frac{a}{2} x^2+b_m x\right) dx\right).
\end{align}
In the following, we first consider $a>0$, under which case we demonstrate the rate of $\mathbb{I}_1$. Then we argue that the case with $a\leq 0$ will only lead to a slower rate. 

Let $h_m(x) = x^s \exp\left(\frac{a}{2} x^2+b_m x\right)$. When $x>1$, we have
\begin{equation}
    \frac{\partial^2 h_m(x)}{\partial x^2} = \left[(ax+b_m)^2 x^s + a(2s+1)x^s + 2s b_m x^{s-1} + s(s-1)x^{s-2}\right] \exp\left(\frac{a}2x^2 + b_m x\right) \geq 0,
\end{equation}
that is, $h_m$ is convex in $x$ for $x>1$. Plugging into \eqref{i1}, we have 
\begin{align}\label{i11}
 \mathbb{I}_1 &\stackrel{<}{\sim} \frac{2^{s-1}}{\sqrt{2\pi}\exp(c_m)} |\Phi^{-1}(\frac1m)| \left[ \exp\left(\frac{a}{2} + b_m\right) + |\Phi^{-1}(\frac1m)|^s \exp\left(\frac{a}{2}\Phi^{-1}(\frac1m)^2 + b_m \Phi^{-1}(\frac1m)\right)\right]\nonumber\\
 &= o\left(m^{-r-\frac{1}{1-\rho}} \log^{\frac{1-r}{2}}{(m)}\right) \textnormal{ as } m \to \infty.
\end{align}
On the other hand, we have 
\begin{align}\label{i2}
\mathbb{I}_2 &\leq 2^{-r} \int_{1}^{\infty} \exp{\left(-\frac{x^2 + 2A_m x +A_m^2}{2(1-\rho)}\right)}< 2^{-r} \int_{-\infty}^{\infty} \exp{\left(-\frac{(x + A_m)^2 }{2(1-\rho)}\right)}= 2^{-r}{\sqrt {2\pi(1-\rho)}},
\end{align}
using the fact
\[
\int_{-\infty }^{\infty }\exp{(-ax^{2})}dx = {\sqrt {\pi  \over a}},\quad (a>0).
\]
Finally, plugging \eqref{i11} and \eqref{i2} into \eqref{am_comp_trans}, we have that
\begin{align}\label{am}
    a_m = o(m^{1-\frac{1}{1-\rho}}\log^{\frac{1-r}{2}}(m)) \to 0 \quad \text{as}\ m \to \infty,\ \text{since}\ \rho > 0.
\end{align}
Based on the above calculations, we can finally apply \lemref{infdivi} and have
\begin{align}
\sum_{i=1}^m Y_{mi}-a_m \stackrel{p}{\to} 0 \quad \text{for all}\ r\leq -1.
\end{align}
%where $Y_{r,\rho, z_0}$ is some random variable with non-degenerate infinitely divisible distribution. 
Therefore, when $r<-1$,
\begin{align}
\lim_{m\to\infty}\beta_{\mu_m,\pi_m,\rho}(r, \alpha) & = \lim_{m\to\infty} \beta_m(\rho, r, \alpha) = \lim_{m\to\infty} \EE{ \PPst{\frac1m \sum p_{mi}^r \geq C_r(m,\alpha) }{Z_0} }\nonumber\\ 
& = \lim_{m\to\infty} \EE{ \PPst{\frac{1}{m^{-r}} \sum p_{mi}^r \geq m^{1+r} C_r(m,\alpha) }{Z_0} }\nonumber\\ 
&=  \EE{ \lim_{m\to\infty} \PP{\sum Y_{mi} \geq m^{r+1}\alpha^r m^{-1-r}} } \nonumber \\
&=  \lim_{m\to\infty} \PP{\sum Y_{mi} - a_m  \geq \alpha^r - a_m } = 0;
\end{align}
and similarly when $r=-1$,
\begin{align}
\lim_{m\to\infty}\beta_{\mu_m,\pi_m,\rho}(r, \alpha) & = \lim_{m\to\infty} \beta_m(\rho, \alpha,r) = \lim_{m\to\infty} \EE{ \PPst{\frac{1}{m} \sum p_{mi}^r \geq \frac{1}{\alpha} + \log m}{Z_0} }\nonumber\\ 
& = \lim_{m\to\infty} \PP{ \sum Y_{mi} - a_m \geq \frac{1}{\alpha} + \log m -a_m} = 0.
\end{align}
In conclusion, for $r \leq -1$, we have that, $\beta(\rho,r,\alpha) =0$ as long as $\mu_m = o(\sqrt{\log{m}})$ and $\rho>0$.

On the other hand, recall that in \thmref{asymh} we derive that the calibrated threshold under equicorrelation when $r\leq -1$ in fact equals to that under independence. Therefore when $\rho = 0$, for all $r\leq -1$ we have 
\begin{align}
    \lim_{m\to\infty}\beta_{\mu_m,\pi_m,\rho}(r, \alpha) & = \lim_{m\to\infty} \beta_m(\rho, \alpha,r) = \lim_{m\to\infty} \PP{\frac{1}{m} \sum p_{mi}^r \geq C_r(m,\alpha)}\nonumber\\
    & = \lim_{m\to\infty} \PP{\frac{1}{m} \sum p_{mi}^r \geq C_r(m,\alpha)} = \alpha.
\end{align}

Here we finish the proof for \thmref{r-1weak}.

% \begin{align}
% \lim_{m\to\infty}\beta_{\mu_m,\pi_m,\rho}(r, \alpha) & = \lim_{m\to\infty} G_m(\rho, r) = \lim_{m\to\infty} \EE{ \PPst{\frac1m \sum p_{mi}^r \geq C_r(m,\alpha) }{Z_0} }\nonumber\\ 
% & = \lim_{m\to\infty} \EE{ \PPst{\frac{1}{m^{-r}} \sum p_{mi}^r \geq m^{1+r} C_r(m,\alpha) }{Z_0} }\nonumber\\ 
% &=  \EE{ \lim_{m\to\infty} \PPst{\sum Y_{mi} \geq m^{r+1}\alpha^r m^{-1-r}}{Z_0} } \\
% &=  \EE{ \lim_{m\to\infty} \PPst{\sum Y_{mi}  \geq \alpha^r }{Z_0} } = \EE{ \PPst{Y_{r,\rho,z_0}  \geq \alpha^r }{Z_0}}   = \mathbb{E}_{Z_0}\left[ F_{Y_{r, \rho, z_0}}(\alpha^r)\right] . 
% \end{align}

% And similarly for $r=-1$,
% \begin{align}
% \lim_{m\to\infty}\beta_{\mu_m,\pi_m,\rho}(r, \alpha) & = \lim_{m\to\infty} \beta_m(\rho, \alpha,r) = \lim_{m\to\infty} \EE{ \PPst{\frac1m \sum p_{mi}^r \geq C_r(m,\alpha) }{Z_0} }\nonumber\\ 
% & = \lim_{m\to\infty} \EE{ \PPst{\frac{1}{m} \sum (p_{mi}^r-\log{m}) \geq C_r(m,\alpha) -\log{m} }{Z_0} }\nonumber\\ 
% &=  \EE{ \lim_{m\to\infty} \PPst{\sum Y_{mi} \geq 1/\alpha}{Z_0} }= \EE{ \PPst{Y_{r,\rho,z_0}  \geq 1/\alpha }{Z_0}}   = \mathbb{E}_{Z_0}\left[ F_{Y_{r, \rho, z_0}}(\alpha^r)\right] . 
% \end{align}

\section{Proof for \thmref{r-1strong}}\label{app:r-1strongpf}
Using the calibrated threshold $c_r(m,\alpha)$ derived in \thmref{asymh}, we have that 
\begin{align}
    \beta_{\mu_m,\pi_m,\rho}(r, \alpha) & = \PP{\frac1m\sum_{i=1}^{m}p_{mi}^r \geq C_r(m,\alpha)} = \PP{\sum_{i=1}^{m}m^r (p_{mi}^r-a_{rm}) \geq \alpha^r},
\end{align}
where $C_r(m,\alpha) = c_r(m,\alpha)^r$, $a_{rm} = 0$ for $r <-1$, and $a_{rm} = \log{m}$ for $r=-1$. 

Therefore, we only need to prove that $\sum_{i=1}^m m^r (P_{mi}^r-a_{rm}) \to \infty$ with probability one, where $a_{rm} = 0$ for $r <-1$, and $a_{rm} = \log{m}$ for $r=-1$. Since
\begin{align}
    \sum_{i=1}^m m^r (P_{mi}^r-a_{rm}) \geq \max_{i}\{m^r (P_{mi}^r-a_{rm})\} = (m \min\{P_{mi}\})^r-m^r a_{rm},
\end{align}
and with part (a) we have
\begin{align}
    \min_{i}\{P_{mi}\} &= \Phi(-\sqrt{1-\rho} \max_{i}\{Z_i + \mu_{mi}/\sqrt{1-\rho}\} - \sqrt{\rho}Z_0) \nonumber\\
    & = \Phi(-\sqrt{1-\rho} \sqrt{2\log{m}} -  \mu_{m} - \sqrt{\rho}Z_0) + o_p(1)\nonumber\\
    & = O_p(m^{-(\sqrt{(1-\rho)} + \sqrt{c})^2}).
\end{align}
Therefore, we have that 
\begin{align}
    (m \min\{P_{mi}\})^r-m^r a_{rm} = O_p(m^{-r\left((\sqrt{(1-\rho)} + \sqrt{c})^2 -1\right)}) \to \infty,
\end{align}
with probability one, since $ \sqrt{c} > 1-\sqrt{(1-\rho)}$.
Hence we have proved the argument for part (a). Similarly, as for part (b) we have
\begin{align}
    \min_{i}\{P_{mi}\} &= \Phi(-\sqrt{1-\rho} \max_{i}\{Z_i + \mu_{mi}/\sqrt{1-\rho}\} - \sqrt{\rho}Z_0) \\
    &\leq\Phi(-\sqrt{1-\rho} \sqrt{2\gamma\log{m}} - \mu_m - \sqrt{\rho}Z_0) + o_p (1) \\
    & = O_p(m^{-(\sqrt{\gamma (1-\rho)} + \sqrt{c})^2})
\end{align}
Therefore, we have that 
\begin{align}
    (m \min\{P_{mi}\})^r-m^r a_{rm} = O_p(m^{-r\left((\sqrt{\gamma(1-\rho)} + \sqrt{c})^2 -1\right)}) \to \infty
\end{align}
with probability one, since $\sqrt{c} > 1-\sqrt{\gamma(1-\rho)}$.
Hence we have concluded the proof.

% Notice that if $a_{r,m}\to\infty$, then the above term is dominate by the first term in the final summation. Let $a_{r,m} = m^{r\rho}$, then we have that, 

% Let $b_{r,m} = \Phi(-\mu_m)^r$, and $a_{r,m} = b_{r,m}^2$. We consider the scenario when $\lim_{m\to\infty}\mu_m = \infty$, and $\lim_{m\to\infty}\pi_m > 0$

% Therefore,
%  \begin{align}
%  \Phi^{-1}((m\epsilon +  a_{r,m})^{1/r})  + \mu_m & =  \Phi^{-1}((m\epsilon +  b_{r,m}^2)^{1/r}) -   \Phi^{-1}((b_{r,m})^{1/r})\\
%  & := f(m\epsilon +  b_{r,m}^2) - f(b_{r,m}) \leq f(b_{r,m}^2) - f(b_{r,m})
%  \end{align}
%  where $f (x)= \Phi^{-1}(x^{\frac1r})$.  It is easy to verify that $f$ is a convex function, and $d f(x) = O(\frac{1}{rx}) < 0$, and we have that  $f(x^2)-f(x) < 0.3 f(x)$ and therefore $f(b_{r,m}^2) - f(b_{r,m}) < 0.3 f(b_{r,m})\to -\infty$ as $b_{r,m}$ goes to infinity. Therefore, we have the first part in \eqref{uan} goes to zero as $m\to 0$;
%  similarly, we have the second part goes to zero as $m\to 0$ .

\end{appendices}

\end{document}